\documentclass[11pt]{article}
\usepackage{amssymb,amsmath,bm,mathrsfs,makeidx,amsfonts,graphicx,amsthm}
\usepackage[square, comma, sort&compress, numbers]{natbib}
\usepackage{epsfig}
\usepackage{rotating}
\usepackage{mathtools}
\usepackage{bm}
\usepackage{extarrows}
\usepackage[colorlinks,linkcolor=red,anchorcolor=blue,citecolor=blue]{hyperref}
\usepackage{booktabs}
\numberwithin{equation}{section}
\def\beginn{\begin{eqnarray*}}
\def\endn{\end{eqnarray*}}
\def\beginy{\begin{eqnarray}}
\def\endy{\end{eqnarray}}
\def\begine{\begin{enumerate}}
\def\ende{\end{enumerate}}

\def\be{\begin{equation}}
\def\ee{\end{equation}}
\def\bea{\begin{eqnarray}}
\def\eea{\end{eqnarray}}

\numberwithin{equation}{section}
\theoremstyle{plain}
\newtheorem{thm}{Theorem}[section]
\newtheorem{lem}{Lemma}

\newtheorem{rmk}{Remark}

\newtheorem{deff}{Definition}

\newtheorem{cond}{Condition}
\newcommand{\non}{\nonumber \\}
\newcommand{\bbA}{{\bf A}}

\newcommand{\bbB}{{\bf B}}
\newcommand{\bbC}{{\bf C}}

\newcommand{\bbD}{{\bf D}}

\newcommand{\bbe}{{\bf e}}

\newcommand{\bbF}{{\bf F}}
\newcommand{\bbP}{{\bf P}}

\newcommand{\bbG}{{\bf G}}
\newcommand{\bbK}{{\bf K}}
\newcommand{\bbH}{{\bf H}}

\newcommand{\bbw}{{\bf w}}
\newcommand{\bbI}{{\bf I}}

\newcommand{\bbJ}{{\bf J}}

\newcommand{\bbM}{{\bf M}}

\newcommand{\bbn}{{\bf n}}
\newcommand{\bbQ}{{\bf Q}}

\newcommand{\bbR}{{\bf R}}
\newcommand{\bbr}{{\bf r}}

\newcommand{\bbT}{{\bf T}}

\newcommand{\bbU}{{\bf U}}

\newcommand{\bbV}{{\bf V}}
\newcommand{\bbv}{{\bf v}}

\newcommand{\bbW}{{\bf W}}

\newcommand{\bbX}{{\bf X}}

\newcommand{\bbx}{{\bf x}}

\newcommand{\bbY}{{\bf Y}}
\newcommand{\bby}{{\bf y}}
\newcommand{\bbZ}{{\bf Z}}

\newcommand{\bbL}{{\bf L}}

\newcommand{\ep}{\ensuremath{\epsilon}}

{
\begin{document}
\title{The Tracy-Widom law for the Largest Eigenvalue of F Type Matrix}
%\author{X. Han, G. M. Pan and B. Zhang}
\author{X. Han, G. M. Pan and B. Zhang\\
Division of Mathematical Sciences \\ Nanyang Technological University,
 Singapore %\\ Email: \texttt{gmpan@ntu.edu.sg}
\vspace{0.2in}
%Guangming Pan\\
%Division of Mathematical Sciences, Nanyang Technological University\\  Singapore, 637371 \\ email: \texttt{gmpan@ntu.edu.sg}
}
\date{}
\maketitle
\begin{abstract}

 Let $\bbA_p=\frac{\bbY\bbY^*}{m}$ and $\bbB_p=\frac{\bbX\bbX^*}{n}$ be two independent random matrices where $\bbX=(X_{ij})_{p \times n}$ and $\bbY=(Y_{ij})_{p \times m}$ respectively consist of real (or complex) independent random variables with $\mathbb{E}X_{ij}=\mathbb{E}Y_{ij}=0$, $\mathbb{E}|X_{ij}|^2=\mathbb{E}|Y_{ij}|^2=1$. Denote by $\lambda_{1}$ the largest root of the determinantal equation $\det(\lambda \bbA_p-\bbB_p)=0$. We establish the Tracy-Widom type universality for $\lambda_{1}$ under some moment conditions on $X_{ij}$ and $Y_{ij}$ when $p/m$ and $p/n$ approach positive constants as $p\rightarrow\infty$.

{\small \bf KE\bbY WORDS}:  Tracy-Widom distribution, largest eigenvalue, sample covariance matrix, F matrix.
\end{abstract}

\section{Introduction}

%In recent years, Big data analysis is growing more and more important.
High-dimensional data now commonly arise in many scientific fields such as genomics, image processing, microarray, proteomics and finance, to name but
a few. It is well-known that the classical theory of multivariate statistical analysis for the fixed dimension p and large sample size n may lose its validity when handling high-dimensional data. % which will cause the so-called curse of dimensionality if we tackle some statistical problems by the traditional  method.
 %In view of this, the development of theoretical results is necessary in big data research.
A popular tool in analyzing large covariance matrices and hence high-dimensional
data is random matrix theory. The spectral analysis of high-dimensional sample covariance matrices has
attracted considerable interests among statisticians, probabilitists and mathematicians since the
seminal work of Marcenko and Pastur \cite{MP67} about the limiting spectral distribution for a class of sample covariance matrices.  One can refer
to the monograph of Bai and Silverstein \cite{BS06} for a comprehensive summary and references therein.

The largest eigenvalue of covariance matrices plays an important role in multivariate statistical analysis such as principle component analysis (PCA), multivariate analysis of variance (MANOVA) and discriminant
analysis. One may refer to \cite{M1982} for more details. In this paper we focus on the largest eigenvalue of the F type matrices. Suppose that
\begin{equation}\label{b11}
\bbA_p=\frac{\bbY\bbY^*}{m},\quad\bbB_p=\frac{\bbX\bbX^*}{n}
\end{equation}
are two independent random matrices where $\bbX=(X_{ij})_{p \times n}$ and $\bbY=(Y_{ij})_{p \times m}$ respectively consist of real (or complex) independent random variables with $\mathbb{E}X_{ij}=\mathbb{E}Y_{ij}=0$ and $\mathbb{E}|X_{ij}|^2=\mathbb{E}|Y_{ij}|^2=1$. Consider the determinantal equation
\begin{equation}\label{b9}
\det(\lambda \bbA_p-\bbB_p)=0.
\end{equation}
When $\bbA_p$ is invertible, the roots to (\ref{b9}) are the eigenvalues of a F matrix
\begin{equation}\label{b12}
\bbA_p^{-1}\bbB_p,
\end{equation}
referred to as a Fisher matrix in the literature. The determinantal equation (\ref{b9}) %Denote its largest root by $\lambda_{1}$.
is closely connected with the generalized eigenproblem
\begin{equation}\label{b10}
\det[\lambda (\bbA_p+\bbB_p)-\bbB_p]=0.
\end{equation}
We illustrate this in the next section. Many classical multivariate statistical tests are based on the roots of (\ref{b9}) or (\ref{b10}). For instance, one may use them to test the equality of two covariance matrices and the general linear hypothesis. In the framework of multivariate analysis of variance (MANOVA), $\bbA_p$ represents the within group covariance matrix while $\bbB_p$ means the between groups covariance matrix. A one-way MANOVA can be used to examine the hypothesis of equality of the mean vectors of interest.

Tracy and Widom in \cite{TW1994,TW1996} first discovered the limiting distributions of the largest eigenvalue for the large Gaussian Wigner ensemble, thus named as Tracy-Widom's law. Since their pioneer work study toward the largest eigenvalues of large random matrices becomes flourishing. To name a few we mention \cite{Johansson2000}, \cite{J2001}, \cite{K2007}, \cite{FP2009} and \cite{Soshnikov2002}. Among them we would mention El Karoui \cite{K2007} which handled the largest eigenvalue of Wishart matrices for the nonnull population covariance matrix and provided a kind of condition on the population covariance matrix to ensure the Tracy-Widow law (see (\ref{3.45}) below).

A follow-up to the above results is to establish the so-called universality property for generally distributed large random matrices. Specifically speaking, the universality property states that the limiting behavior of an eigenvalue statistic usually is not dependent on the distribution of the matrix entries. Indeed, the Tracy-Widom law has been established for the general sample covariance matrices under very general assumptions on the distributions of the entries of $\bbX$. The
readers can refer to \cite{TV2011}, \cite{TV2012}, \cite{LHY2011}, \cite{LAY2013}, \cite{PY11}, \cite{WK2012}, \cite{BPZ2014a}, \cite{JK14}, \cite{KY14} for some representative developments on this topic. When proving universality an important tool is the Lindeberg comparison strategy (see Tao and Vu in \cite{TV2011} and Erdos, Yau and Yin \cite{LHY2011}) and an important input when applying Lindeberg's comparison strategy is the strong local law developed by Erdos, Schlein and Yau in \cite{ESY20092} and Erdos, Yau and Yin in \cite{LHY2011}.

Johnstone in \cite{J08} proved that the largest root of (\ref{b11}) converges to Tracy and Widom's distribution of type one after appropriate centering and scaling when the dimension $p$ of the matrices $\bbA_p$ and $\bbB_p$ is even, $\lim\limits_{p\rightarrow\infty}p/m<1$ and $\bbB_p$ and $\bbA_p$ are both Wishart matrices. It is believed that the limiting distribution should not be affected by the dimension p. Indeed, numerical investigations both in \cite{J08} and \cite{J09} suggest that the Tracy and Widom approximation in the odd dimension case works as well as in the even dimension case. Besides, as it can be guessed, the Tracy and Widom approximation should not rely on the Gaussian assumption.  However, theoretical support for these remains open. Furthermore, when $\bbA_p$ is not invertible the limiting distribution of the largest root to (\ref{b11}) is unknown yet even under the gaussian assumption.

In this paper, we prove the universality of the largest root of (\ref{b9}) by imposing some moment conditions on $\bbA_p$ and $\bbB_p$. Specifically speaking we prove that the largest root of (\ref{b9}) converges in distribution to the Tracy and Widom law for the general distributions of the entries of $\bbX$ and $\bbY$ no matter what the dimension p is, even or odd. Moreover the result holds when $\lim\limits_{p\rightarrow\infty}p/m<1$ or $\lim\limits_{p\rightarrow\infty}p/m>1$, corresponding to invertible $\bbA_p$ and non-invertible $\bbA_p$. This result also implies the asymptotic distribution of the largest root of (\ref{b10}).

At this point it is also appropriate to mention some related work about the roots of (\ref{b9}). The limiting spectral distribution of the roots was derived by \cite{WK80} and \cite{BS06}. One may also find the limits of the largest root and the smallest root in \cite{BS06}. Central limit theorem about linear spectral statistics was established in \cite{ZSR}. Very recently, the so-called spiked F model has been investigated by \cite{DJO15} and \cite{WY}. We would like to point out that they prove the local asymptotic normality or asymptotic normality for the largest eigenvalue of the spiked F model, which is completely different from our setting.

 We conclude this section by outlining some ideas in the proof and presenting the structure of the rest of the paper. When $\bbA_p$ is invertible, the roots to (\ref{b9}) become those of the F matrix $\bbA_p^{-1}\bbB_p$ so that we may work on $\bbA_p^{-1}\bbB_p$. Roughly
 speaking, $\bbA_p^{-1}\bbB_p$ can be viewed as a kind of general sample covariance matrix $\bbT_n^{1/2}\bbX\bbX^*\bbT_n^{1/2}$ with $\bbT_n$ being a population covariance matrix by conditioning on $\bbB_p$. Denote the largest root of (\ref{b9}) by $\lambda_1$. The key idea is to break $\lambda_1$ into a sum of two parts as follows
 \begin{equation}\label{b13}
\lambda_1-\mu_p=(\lambda_1-\hat{\mu}_p)+(\hat{\mu}_p-\mu_p),
 \end{equation}
 where $\hat{\mu}_p$ is an appropriate value when $\bbB_p$ is given and $\mu_p$ is an appropriate value when $\bbB_p$ is not given (their definitions are given in the later sections). However we can not condition on $\bbB_p$ directly. Instead we first construct an appropriate event so that we can handle the first term on the right hand of (\ref{b13}) on the event to apply the earlier results about $\bbT_n^{1/2}\bbX\bbX^*\bbT_n^{1/2}$. Particularly we need to verify the condition (\ref{3.45}) below. Once this is done, the next step is to prove that the second term on the right hand of (\ref{b13}) after scaling converges to zero in probability. This approach is different from that used in the literature in proving universality for the local eigenvalue statistics.

 %General speaking, to investigate the Tracy-Widom law of the largest root of $\det[\bbX\bbX^*-\lambda (\bbX\bbX^*+\bbY\bbY^*)]=0$ under the condition $\bbY\bbY^*$ is invertible, we translate it to the largest eigenvalue of F matrix, i.e. $(\bbY\bbY^*)^{-1}\bbX\bbX^*$. After fixing \bbY, we can find out the largest eigenvalue of $(\bbY\bbY^*)^{-1}\bbX\bbX^*$ can be considered as the largest eigenvalue of the sample covariance matrix $\bbX^*TT^*\bbX$, where T=$(\bbY\bbY^*)^{-1/2}$. For the largest eigenvalue of covariance matrix, when T=I, Tracy-Widom law  was obtained in \cite{J2001} for real Wishart matrices and in \cite{J2000} for complex Wishart matrices. In \cite{PY11}, this result was extended to a large class of distributions for real matrices when T=I. For general T, under some regularity condition on T, Tracy-Widom law was also obtained in \cite{K2007} for complex covariance matrix. For the real case, \cite{BPZ2014} showed the universality of any two sample covariance matrices if their first four(or two, if T is diagonal matrix) moments are the same and Tracy-Widom law was obtained in \cite{JK14} under the same condition. \cite{KY14} get Tracy-Widom law under more permissive conditions.

Unfortunately, when $\bbA_p$ is not invertible we can not work on F matrices $\bbA^{-1}\bbB_p$ anymore. To overcome the difficulty we instead start from the determinantal equation (\ref{b9}). It turns out that the largest root $\lambda_1$ can then be linked to the largest root of some F matrix when $\bbX$ consists of Gaussian random variables. Therefore the result about F matrices $\bbA^{-1}\bbB_p$ is applicable.  % will see that the largest root of $\det[\bbX\bbX^*-\lambda (\bbX\bbX^*+\bbY\bbY^*)]=0$ is always equal to 1, so we will focus on the largest(smaller than 1) root of the determinant $\det[\bbX\bbX^*-\lambda (\bbX\bbX^*+\bbY\bbY^*)]=0$. %But the limit distribution of it hasn't figured out even under the gaussian condition. In section 6,
For general distributions we find that it is equivalent to working on such a ``covariance-type'' matrix
\begin{equation}\label{1231.1}
\bbD^{-\frac{1}{2}}\bbU_1\bbX(\bbI-\bbX^*\bbU^*_2(\bbU_2\bbX\bbX^*\bbU^*_2)^{-1}\bbU_2\bbX)\bbX^*\bbU^*_1\bbD^{-\frac{1}{2}}.
\end{equation}
The definitions of $\bbD$ and $\bbU_j,j=1,2$ are given in the later section. This matrix is much more complicated than general sample covariance matrices.  To deal with (\ref{1231.1}) we construct a $3\times3$ block linearization matrix
\begin{eqnarray}\label{1205.1*}
\bbH=\bbH(\bbX)= \left(
  \begin{array}{ccc}
    -z\bbI & 0 & \bbD^{-1/2}\bbU_1\bbX\\
    0 & 0 & \bbU_2\bbX\\
  \bbX^T\bbU^T_1\bbD^{-1/2} & \bbX^T \bbU^T_2& -\bbI\\
  \end{array}
\right),
\end{eqnarray}
where $z=E+i\eta$ is a complex number with a positive imaginary part. It turns out that the upper left block of the $3\times 3$ block matrix $\bbH^{-1}$ is the Stieltjes transform of (\ref{1231.1}) by simple calculations. We next develop the strong local law around the right end support $\mu_p$ by using a type of Lindeberg's comparison strategy raised in \cite{KY14} and then use it to prove edge universality by adapting the approach used in \cite{LHY2011} and \cite{BPZ2014a}.

 %we consider the gaussian case first. It is easy to see that it is equivalent to the previous case that $\bbY\bbY^*$ is invertible, which implies the Tracy-Widom law. For the non-gaussian case, the classical ``linearization'' method does not work, we use a $3\times3$ block matrix to linearize (\ref{1231.1}) and use the interpolation method raised by \cite{KY14} to prove the local law. But this matrix will bring a new challenge in proving the universality of the largest eigenvalue because the block matrix is not ``Wigner'' type, which is difficult to directly use the previous result of \cite{LHY2011}. Fortunately, we can use similar interpolation method to tackle it, which is succinct and powerful.

%Moreover, we also proved the invertibility of  $U_2\bbX\bbX^*U^*_2$, which is Lemma \ref{1121-1}. This result tell us the matrix $U_2\bbX\bbX^*U^*_2$ is invertible with probability $1-n^{-D}$ for any large constant D. Furthermore, The smallest eigenvalue of $U_2\bbX\bbX^*U^*_2$ is bounded away from 0.

The paper is organized as follows. Section 2 is to give the main results. A statistical application and Tracy-Widom approximation will be discussed in Section 3. Section 4 is devoted to proving the main result when $\bbA_p$ is invertible. In section 5 we will show the equivalence between the asymptotic means and asymptotic variances respectively given by \cite{J08} and by this paper. Sections 6 and 7 will prove the main result when $\bbA_p$ is not invertible. %In section 8, we prove the invertibility of the matrix $U_2\bbX\bbX^*U^*_2$, which is  also important in probability.

\section{The main results}
%Set $m=m(p) $ and $n=n(p)$.
Throughout the paper we make the following conditions.
\begin{cond}\label{cond1}
% if each entry of Z satisfy the following moments conditions:
Assume that $\{Z_{ij}\}$ are independent random variables with
$\mathbb{E} Z_{ij}=0, \mathbb{E }|Z_{ij}|^2=1.$
For all $k \in N$, there is a constant $C_k$ such that
$
\mathbb{E}|Z_{ij}|^k \leq C_k.
$
%for all $i$ and $j$.
In addition, if $\{Z_{ij}\}$ are complex, then
$\mathbb{E} Z_{ij}^2=0.$
\end{cond}
We say that a random matrix $\bbZ=(Z_{ij})$ satisfies Condition \ref{cond1} if its entries $\{Z_{ij}\}$ satisfy Condition \ref{cond1}.

\begin{cond}\label{cond2}
Assume that random matrices
$\bbX=(\bbX_{ij})_{p,n}$ and $\bbY=(\bbY_{ij})_{p,m}$ are independent.
\end{cond}

\begin{cond}\label{cond3} Set $m=m(p) $ and $n=n(p)$. Suppose that
$$
\lim_{p \rightarrow \infty} \frac{p}{m}=d_1>0,\quad \lim_{p \rightarrow \infty} \frac{p}{n}=d_2>0,\quad 0<\lim_{p \rightarrow \infty} \frac{p}{m+n}< 1.
$$
\end{cond}
% be a $p \times n$  random matrix; $\bbX_{ij}$  are independent random variables
%which satisfy
%\begin{equation}\label{2.1}
%E \bbX_{ij}=0, E |\bbX_{ij}|^2=1.
%\end{equation}
%For all $k \in N$, there is a constant $C_k$ such that
%
%\begin{equation}\label{2.2}
%E|\bbX_{ij}|^k \leq C_k
%\end{equation}
%for all $i$ and $j$.
%
%Let $\bbY=(\bbY_{ij})_{p,m}$ be a $p \times m$  random matrix; $\bbY_{ij}$  are independent random variables which satisfy
%\begin{equation}\label{2.3}
%E \bbY_{ij}=0, E |\bbY_{ij}|^2=1.
%\end{equation}
%In addition, for complex case we assume that $E \bbX_{ij}^2=E \bbY_{ij}^2=0$,
%and there is a constant $\theta >0$ such that for $u>1$,
%
%\begin{equation}\label{2.4}
%P(|\bbY_{ij}|>u) \leq \theta^{-1} exp(-u^{\theta})
%\end{equation}
%
%for all $i$ and $j$.

%\subsection{Some definitions}
%Below is to give some useful definitions.

%\begin{deff}

To present the main results uniformly we define $\breve{m}=\max\{m,p\}$, $\breve{n}=\min\{n,m+n-p\}$ and $\breve{p}=\min\{m,p\}$. Moreover let
\begin{equation}\label{5.1}
\sin^2(\gamma/2)=\frac{\min\{\breve{p},\breve{n}\}-1/2}{\breve{m}+\breve{n}-1},\ \ \sin^2(\psi/2)=\frac{\max\{\breve{p},\breve{n}\}-1/2}{\breve{m}+\breve{n}-1}.
\end{equation}
\begin{equation}\label{5.2}
\mu_{J,p}=\tan^2(\frac{\gamma+\psi}{2}),\ \
%\end{equation}
%and
%\begin{equation}\label{r5.023}
\sigma_{J,p}^3=\mu_{J,p}^3\frac{16}{(\breve{m}+\breve{n}-1)^2}\frac{1}{\sin(\gamma)\sin(\psi)\sin^2(\gamma+\psi)}.
\end{equation}
Formulas (\ref{5.2}) can be found in \cite{J08} when $d_1<1$.

We below present alternative expressions of $\mu_{J,p}$ and $\sigma_{J,p}$. To this end, define a modified density of the Marchenko-Pastur law \cite{MP67} (MP law) by
\begin{equation}\label{2.5}
\varrho_{p}(x)=\frac{1}{2 \pi x \frac{\breve{p}}{\breve{m}}} \sqrt{(b_{p}-x)(x-a_{p})} \mathbf{I}(a_{p} \leq x \leq b_{p}),
\end{equation}
where $a_p=(1-\sqrt{\frac{\breve{p}}{\breve{m}}})^2$ and $b_p=(1+\sqrt{\frac{\breve{p}}{\breve{m}}})^2$. %replacing the limit of $p/m$ by $p/m$.
 Let $\gamma_{1} \geq \gamma_{2} \geq \cdots \geq \gamma_{p}$ satisfy
\begin{equation}\label{2.6}
\int_{\gamma_{j}}^{+\infty}\varrho_{p}(x)dx=\frac{j}{p},
\end{equation}
with $\gamma_{0} =b_{p}$ and $\gamma_{p} =a_{p}$. %$\gamma_{j}$ depends on $p$. %In fact all the values in this paper may depend on $p$.
%\end{deff}
Moreover suppose that $c_{p} \in [0, a_{p})$ satisfies the equation
\begin{equation}\label{2.7}
\int_{- \infty}^{+\infty} (\frac{c_{p}}{x-c_{p}})^2 \varrho_{p}(x)dx=\frac{n}{p}.
\end{equation}
One may easily check the existence and uniqueness of $c_{p}$. %in Lemma \ref{0312-1}.
Define
%\end{deff}
\begin{equation}\label{2.8}
\mu_{p}=\frac{1}{c_{p}}(1+\frac{p}{n}\int_{- \infty}^{+\infty} (\frac{c_{p}}{x-c_{p}}) \varrho_{p}(x)dx)
\end{equation}
and
\begin{equation}\label{2.9}
\frac{1}{\sigma_{p}^3}=\frac{1}{c_{p}^3}(1+\frac{p}{n}\int_{- \infty}^{+\infty} (\frac{c_{p}}{x-c_{p}})^3 \varrho_{p}(x)dx).
\end{equation}
It turns out that (\ref{5.2}) and (\ref{2.8})-(\ref{2.9}) are equivalent subject to some scaling, which is verified in Section 5.

%We below present alternative expressions of $\mu_{p}$ and $\sigma_{p}$. Let $0< \gamma/2 < \frac{\pi}{2} $ and $0< \psi/2 < \frac{\pi}{2}$ satisfy
%\begin{equation}\label{5.1}
%\sin^2(\gamma/2)=\frac{\min\{p,n\}-1/2}{m+n-1},\quad \sin^2(\psi/2)=\frac{\max\{p,n\}-1/2}{m+n-1}.
%\end{equation}
%Then \begin{equation}\label{5.2}
%\mu_{J,p}=\tan^2(\frac{\gamma+\psi}{2}),\quad
%\end{equation}
%and
%\begin{equation}\label{5.3}
%\sigma_{J,p}^3=\mu_{p}^3\frac{16}{(m+n-1)^2}\frac{1}{\sin(\gamma)\sin(\psi)\sin^2(\gamma+\psi)}.
%\end{equation}

% We below distinguish two cases: $\bbY\bbY^*$ is invertible
%and $\bbY\bbY^*$ is not invertible.

We also need the following moment match condition.
\begin{deff}[moment matching]\label{1222-1}
Let $\bbX^1=(x^1_{ij})_{M\times N}$ and $\bbX^0=(x^0_{ij})_{M\times N}$ be two  matrices satsfing Condition \ref{cond1} . We say that $\bbX^1$ matches $\bbX^0$ to order q, if  for the integers i,j,l and k satisfing  $1 \le i \le M$, $1 \le j \le N$ ,$0\le l,k$  and  $l+k\le q$, they have the relationship
\begin{eqnarray}\label{1129.1}
\mathbb{E}\left[(\Im x_{ij}^1)^l(\Re x_{ij}^1)^k\right]=\mathbb{E}\left[(\Im x_{ij}^0)^l(\Re x_{ij}^0)^k\right]+O(\exp(-(\log p)^C)),
\end{eqnarray}
where $C$ is some positive constant bigger than one, $\Re x$  is the real part and $\Im x$  is the imaginary part of x.
\end{deff}
Throughout the paper we use $\bbX^0$ to stand for the random matrix consisting of independent Gaussian random variables with mean zero and variance one.% and $\bbX^1$ to mean $\bbX$.

%\subsection{The main results}
%\subsection{Invertible $\bbY\bbY^*$}

Denote the type-i Tracy-Widom distribution by $F_i$, i=1, 2(see \cite{TW1996}). Set $\bbB_p=\frac{\bbX\bbX^*}{\breve{n}}$ and $\bbA_p=\frac{\bbY\bbY^*}{\breve{m}}$. We are now in a position to state the main results about F type matrices.
\begin{thm}\label{t1}
Suppose that the real random matrices $\bbX$ and $\bbY$ satisfy Conditions \ref{cond1}-\ref{cond3}. Moreover suppose that $0< d_2 < \infty$. %$p$, $m=m(p)$ and $n=n(p)$% $ \rightarrow \infty$ together in such a way that
%\begin{equation}\label{2.16}
%\lim_{p \rightarrow \infty} \frac{p}{m}=d_1 ,\lim_{p \rightarrow \infty} \frac{p}{n}=d_2,
%\end{equation}.
Denote the largest root of $\det(\lambda \bbA_p-\bbB_p)=0$ by $\lambda_{1}$.
  %$\det\Big(\lambda \bbY\bbY^*-\bbX\bbX^*\Big )=0$ by $\frac{n}{m}\lambda_{1}$.
\begin{itemize}
\item[(i)] If $0<d_1<1$, then
\begin{equation}\label{2.16}
\lim_{p \rightarrow \infty} P( \frac{\frac{\breve{n}}{\breve{m}}\lambda_{1}-\mu_{J,p}}{\sigma_{J,p}} \leq s)=F_1(s).
\end{equation}
 %  Then
%\begin{equation}\label{2.17}
%\lim_{p \rightarrow \infty} P(\sigma_{p} n^{2/3} (\lambda_{1}-\mu_{p}) \leq s)=F_1(s).
%\end{equation}
\item[(ii)] If $d_1>1$ and $\bbX$ matches the standard $\bbX^0$ to order 3, then (\ref{2.16}) still holds.
\end{itemize}
\end{thm}

\begin{rmk}\label{r1*}
When $\bbX$ and $\bbY$ are complex random matrices, Theorem \ref{t1} still holds but the Tracy-Widom distribution $F_1(s)$ should be replaced by $F_2(s)$.

If $0<d_1<1$, then $\bbA_p$ is invertible. In this case the largest eigenvalue $\lambda_1$ is that of F matrices $\bbA^{-1}_p\bbB_p$. If $d_1>1$, then $\bbA_p$ is not invertible.
\end{rmk}

\begin{rmk}
%When the case satisfies the conditions of
Theorem \ref{t1} immediately implies the distribution of the largest root of $\det(\lambda (\bbB_p +\bbA_p)-\bbB_p )=0$.  In fact the largest root of $\det(\lambda (\bbB_p+\bbA_p)-\bbB_p )=0$ is $\frac{\lambda_{1}}{1+\lambda_{1}}$ if $\lambda_{1}$ is the largest root of the F matrices $\bbB_p \bbA^{-1}_p$ in Theorem \ref{t1} when $0<d_1<1$.

When $d_1>1$ the largest root of $\det\Big(\lambda (\bbB_p +\bbA_p)- \bbB_p \Big)=0$ is one with multiplicity $(p-m)$. We instead consider the $(p-m+1)$th largest root of $\det\Big(\lambda (\bbB_p +\bbA_p)- \bbB_p\Big)=0$. It turns out that the $(p-m+1)$th largest root of $\det\Big(\lambda (\bbB_p+\bbA_p)-\bbB_p \Big)=0$ is $\frac{\lambda_{1}}{1+\lambda_{1}}$ if $\lambda_{1}$ is the largest root of $det(\lambda \bbA_p-\bbB_p )=0$.

%It then suffices to consider the largest root of $\det(\frac{\lambda}{1-\lambda} \frac{\bbY\bbY^*}{\breve{m}}-\frac{\bbX\bbX^*}{\breve{n}})=0$ because $\frac{\lambda}{1-\lambda}$ is an increasing function.
Moreover note  the equality
$$(\bbB_p+\bbA_p)^{-1} \bbB_p+(\bbB_p+\bbA_p)^{-1} \bbA_p=I.$$
If $\bbY$ matches $\bbX^0$ to order 3, then the smallest positive root of $\det(\lambda (\bbB_p+\bbA_p)- \bbB_p)=0$ also tends to type-1 Tracy-Widom distribution after appropriate centralizing and rescaling by Theorem \ref{t1} when $d_1>1$ and $d_2>1$.
\end{rmk}
%\begin{rmk}\label{r1} The exact expressions of $\mu_{p}$ and $\sigma_{p}$ are given as follows.
%Let $0< \alpha_{p} < \frac{\pi}{2} $ and $0< \beta_{p} < \frac{\pi}{2}$ satisfy
%\begin{equation}\label{2.18}
%\sin^2(\alpha_{p})=\frac{p}{m+n},\quad \sin^2(\beta_{p})=\frac{n}{m+n}.
%\end{equation}
%Then
%\begin{equation}\label{2.19}
%\mu_{p}=\frac{m}{n} \tan^2(\alpha_{p}+\beta_{p})
%\end{equation}
%and
%\begin{equation}\label{2.20}
%\frac{1}{\sigma_{p}^3}=\mu_{p}^3\frac{16n^2}{(m+n)^2}\frac{1}{\sin(2\beta_{p})\sin(2\alpha_{p})\sin^2(2\beta_{p}+2\alpha_{p})}.
%\end{equation}
%%We will prove it in Section 5.
%\end{rmk}
%We will prove it in section 5.

We would like to point out that Johnstone \cite{J08} proved part (i) of Theorem (\ref{t1}) when $p$ is even, $\bbA_p$ and $\bbB_p$ are both Wishart matrices. Part (ii) of Theorem (\ref{t1}) is new even if $\bbA_p$ and $\bbB_p$ are both Wishart matrices.  When proving Theorem \ref{t1} we have indeed obtained different asymptotic mean and variance. Precisely we have proved that
\begin{equation}\label{2.17}
\lim_{p \rightarrow \infty} P(\sigma_{p} \breve{n}^{2/3} (\lambda_{1}-\mu_{p}) \leq s)=F_1(s)
\end{equation}
%\begin{rmk}\label{r1}
and that
\begin{equation}\label{a40}|\frac{\breve{m}}{\breve{n}}\mu_{J,p}-\mu_{p}|=O(p^{-1}),\quad \lim_{p \rightarrow \infty }\sigma_{p} \frac{\breve{m}}{\breve{n}^{1/3}}\sigma_{J,p}=1.
 \end{equation}
 (\ref{2.17}) and (\ref{a40}) imply Theorem \ref{t1}. %Then we can find that (\ref{5.04}) is equivalent to

\section{Application and Simulations}

This section is to discuss some applications of our universality results in high-dimensional statistical inference and
conduct simulations to check the quality of the approximations of our limiting law.

\subsection{Equality of two covariance matrices}
%\subsubsection{}
 Consider the model of the following form
 $$\bbZ_1=\mathbf{\Sigma}_1^{\frac{1}{2}}\bbX, \ \ \bbZ_2=\mathbf{\Sigma}_2^{\frac{1}{2}}\bbY,$$
 where \bbX \ and \bbY \ are $p\times n$ and $p \times m$ random matrices satisfying the conditions of Theorem \ref{t1}, $\mathbf{\Sigma}_1$ and $\mathbf{\Sigma}_2$ are $p\times p$ invertible population covariance matrices.
 %Suppose that $\bbX_i$(i=1,...,n) and $\bbY_j$(j=1,...,m) are two p dimensional independent samples from two distributions with covariance matrices $\Sigma_1$  and $\Sigma_2$ respectively.
 We are interested in testing whether $\mathbf{\Sigma}_1=\mathbf{\Sigma}_2$. Formally, we focus on the following hypothesis testing problem
   $$\bbH_0:\mathbf{\Sigma}_1=\mathbf{\Sigma}_2 \ \ vs. \ \ \bbH_1:\mathbf{\Sigma}_1\neq \mathbf{\Sigma}_2.$$
   Under the null hypothesis we have
    $$det(\lambda \frac{\bbZ_2\bbZ_2^*}{\breve{m}}-\frac{\bbZ_1\bbZ_1^*}{\breve{n}})=0\Longleftrightarrow det(\lambda \frac{\bbY\bbY^*}{\breve{m}}-\frac{\bbX\bbX^*}{\breve{n}})=0,$$
     which implies that we can apply our theoretical result to the largest root of $det(\lambda \frac{\bbZ_2\bbZ_2^*}{\breve{m}}-\frac{\bbZ_1\bbZ_1^*}{\breve{n}})=0$ under the null hypothesis. By Theorem \ref{t1} we see that $\lambda_1$ tends to Tracy-Widom's distribution after centralizing and rescalling.
 %largest(smaller than 1) eigenvalue of $(\bbX\bbX^*+\bbY\bbY^*)^{-1}\bbX\bbX^*$, where $\bbX=(\bbX_1,...,\bbX_n)$ and $\bbY=(\bbY_1,...,\bbY_m)$.  If $\bbX_i$ and $\bbY_j$ follow normal distributions $N(\mu_1,\Sigma_1)$ and $N(\mu_2,\Sigma_2)$, then under the null hypothesis the largest eigenvalue $\theta$ follow $\theta(p,n,m)$ distribution, one can refer to \cite{M1982}(page 332) for instance.

\subsection{Simulations}

We conduct some numerical simulations to check the accuracy of the distributional
approximations in Theorem \ref{t1} under various settings of $(p,m,,n)$ and the distribution of $\bbX$. We also study the power for the testing of equality of two covariance matrices.

As in \cite{J08} we below use $\ln(\lambda_1)$ to run simulations. To do so we first give its distribution. By \cite{J08} and (\ref{2.17}) we can find that
\begin{equation}\label{2.17s1}
\lambda_{1}= \mu_{p}+\frac{Z}{\sigma_{p}\breve{n}^{2/3}}+o_p(\breve{n}^{-2/3}),
\end{equation}
where $Z=F_1^{-1}(U)$ and $U$ is a $U(0,1)$ random variable. By Taylor's expansion we then have
\begin{equation}\label{2.17s2}
\ln(\lambda_{1})= \ln(\mu_{p})+\frac{Z}{\mu_{p}\sigma_{p}\breve{n}^{2/3}}+o_p(\breve{n}^{-2/3}).
\end{equation}
Recall $|\frac{m}{n}\mu_{J,p}-\mu_{p}|=O(p^{-1})$ and $\lim_{p \rightarrow \infty }\sigma_{p} \frac{m}{n^{1/3}}\sigma_{J,p}=1$ in Section 2. Summarizing the above we can find
\begin{equation}\label{2.17sln}
\lim_{p \rightarrow \infty} P(\sigma_{pln}  (\ln(\lambda_{1})-\mu_{pln}) \leq s)=F_1(s),
\end{equation}
where
\begin{equation}\label{2.19sln}
\mu_{pln}= \ln(\frac{\breve{m}}{\breve{n}}\mu_{J,p}),\quad
%\end{equation}
%and
%\begin{equation}\label{2.20sln}
\sigma_{pln}=\frac{\mu_{J,p}}{\sigma_{J,p}}.
\end{equation}

%\newpage
\subsubsection{Accuracy of approximations for TW laws and size}

We conduct some numerical simulations to check the accuracy of the distributional
approximations in Theorem \ref{t1}, which include the size of the test as well.
% Table generated by Excel2LaTe\bbX from sheet 'MB'
\begin{table}[htbp]\label{table1}
  \centering
  \caption{Standard quantiles for several triples (p,m,n): Gaussian case}
    \begin{tabular}{rrrrrrrrrrr}
    \toprule
    & & \multicolumn{4}{c}{Initial triple $M_0$=(5,40,10)}  & \multicolumn{4}{c}{Initial triple $M_1$=(30,20,25)} & \\ \cmidrule(r){3-6} \cmidrule(r){7-10}
    Percentile & TW   &  $M_0$  &   $2M_0$    &   $3M_0$    &  $4M_0$    &    $M_1$   &  $2M_1$      &   $3M_1$     &   $4M_1$     & 2*SE \\
    \midrule
    -3.9  & 0.01  & 0.0208 & 0.0133 & 0.0124 & 0.0115 & 0.0017 & 0.0035 & 0.0048 & 0.0060 & 0.002 \\
    -3.18 & 0.05  & 0.0680 & 0.0601 & 0.0562 & 0.0582 & 0.0210 & 0.0276 & 0.0327 & 0.0370 & 0.004 \\
    -2.78 & 0.1   & 0.1176 & 0.1120 & 0.1088 & 0.1095 & 0.0608 & 0.0712 & 0.0808 & 0.0842 & 0.006 \\
    -1.91 & 0.3   & 0.3154 & 0.3030 & 0.3080 & 0.3084 & 0.2641 & 0.2744 & 0.2864 & 0.2909 & 0.009 \\
    -1.27 & 0.5   & 0.5139 & 0.5070 & 0.5051 & 0.5082 & 0.4839 & 0.4904 & 0.4960 & 0.4964 & 0.01 \\
    -0.59 & 0.7   & 0.7073 & 0.7154 & 0.7012 & 0.7111 & 0.7055 & 0.7031 & 0.7019 & 0.7005 & 0.009 \\
    0.45  & 0.9   & 0.9083 & 0.9058 & 0.9047 & 0.9090 & 0.9040 & 0.9010 & 0.9016 & 0.9003 & 0.006 \\
    0.98  & 0.95  & 0.9561 & 0.9544 & 0.9517 & 0.9557 & 0.9489 & 0.9530 & 0.9504 & 0.9498 & 0.004 \\
    2.02  & 0.99  & 0.9919 & 0.9909 & 0.9913 & 0.9919 & 0.9878 & 0.9887 & 0.9897 & 0.9901 & 0.002 \\
    \bottomrule
    \end{tabular}%
  \label{tab:addlabel}%
\end{table}%

Table 1 is done by R. We set two initial triples $(p,m,n)$ of $M_0=(5,40,10)$ and $M_1=(30,20,25)$ and then consider $2M_i,3M_i$ and $4M_i$, i=1,2. The triples $M_0$ and $M_1$ correspond to invertible $\bbY\bbY^*$ and noninvertible $\bbY\bbY^*$ respectively. For each case we generate 10000 (\bbX,\bbY) whose entries follow standard normal distribution. We calculate the largest root of $det(\lambda \frac{\bbZ_2\bbZ_2^*}{\breve{m}}-\frac{\bbZ_1\bbZ_1^*}{\breve{n}})=0$ to get $\ln(\lambda_{1})$ and renormalize it with $\mu_{pln}$ and $\sigma_{pln}$. In the ``Pecentile column'', the quantiles of $TW_1$ law corresponding to the ``TW'' column are listed. We state the values of the empirical distributions of the renormalized $\lambda_1$ for various triples at the corresponding quantiles in columns 3-10 and the standard errors based on binomial sampling are listed in the last column. QQ-plots corresponding to the triples $(20,160,40)$ and $(120,80,100)$ are also stated below.

\includegraphics[width=0.5\textwidth,angle=0]{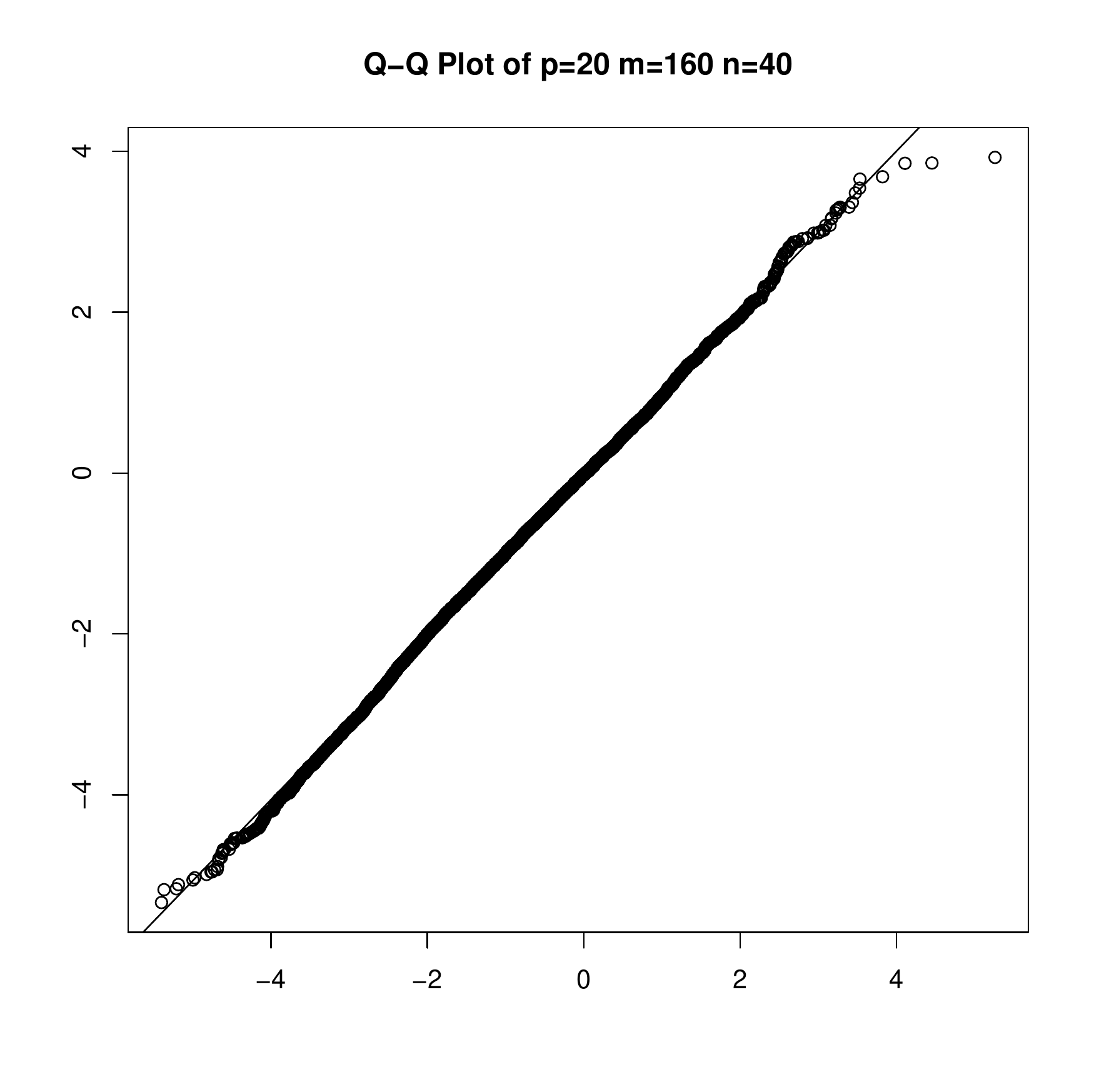}
\includegraphics[width=0.5\textwidth,angle=0]{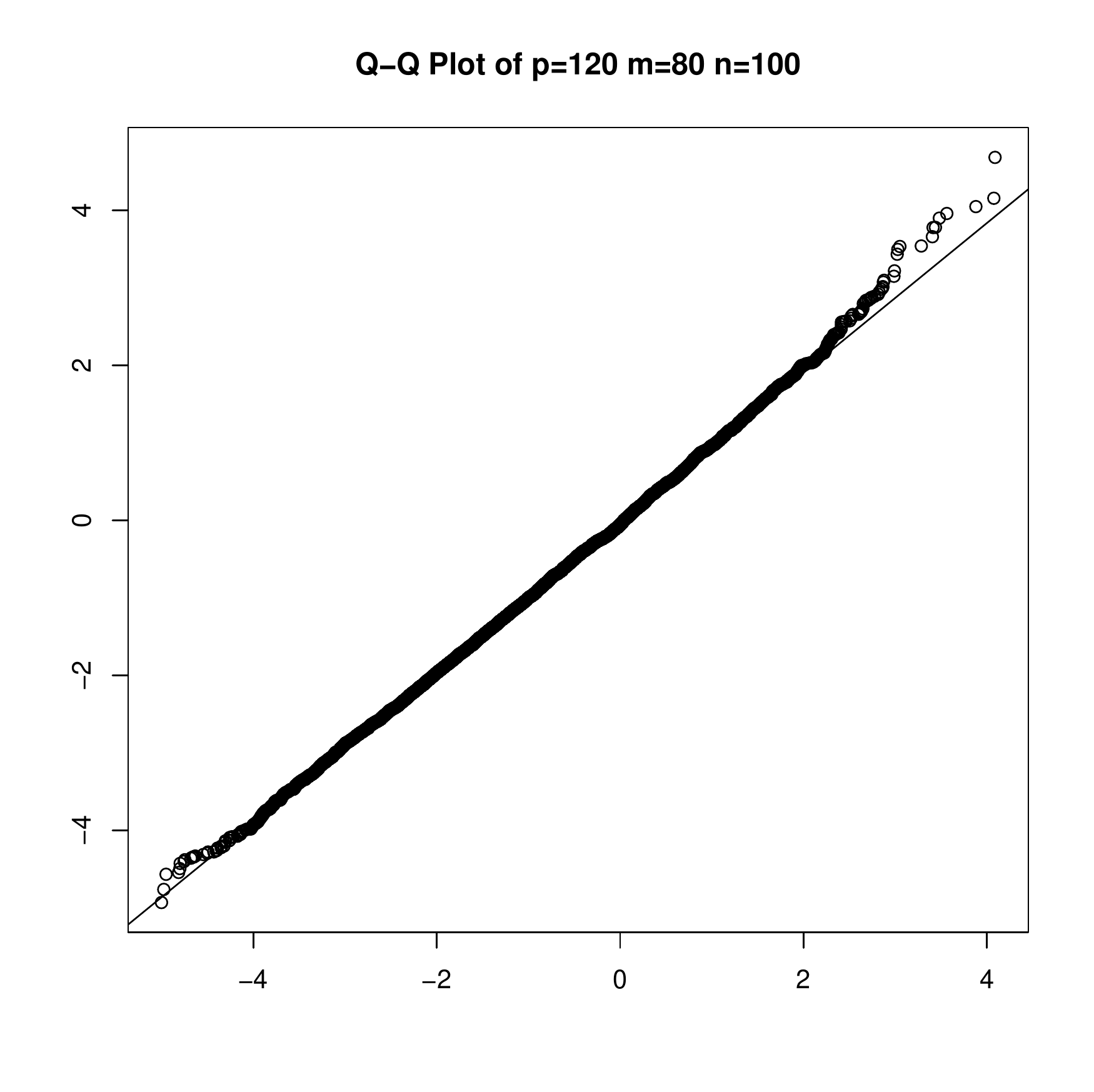}
 The next two tables and graphs are the same as table 1 and the corresponding graphs except that that we replace the gaussian distribution by the some discrete distribution and uniform distribution.
\begin{table}[htbp]
  \centering
  \caption{Standard quantiles for several triples (p,m,n): Discrete  distribution with the probability mass function P($\bbx=\sqrt{3}$)=P($\bbx=-\sqrt{3}$)=1/6 and P($\bbx$=0)=2/3. }
    \begin{tabular}{rrrrrrrrrrr}
   \toprule
    & & \multicolumn{4}{c}{Initial triple $M_0$=(5,40,10)}  & \multicolumn{4}{c}{Initial triple $M_1$=(30,20,25)} & \\ \cmidrule(r){3-6} \cmidrule(r){7-10}
    Percentile & TW   &  $M_0$  &   $2M_0$    &   $3M_0$    &  $4M_0$    &    $M_1$   &  $2M_1$      &   $3M_1$     &   $4M_1$     & 2*SE \\
    \midrule
    -3.9  & 0.01  & 0.0192 & 0.0132 & 0.0136 & 0.0123 & 0.0006 & 0.0031 & 0.0046 & 0.0047 & 0.002\\
    -3.18 & 0.05  & 0.0637 & 0.0581 & 0.0571 & 0.0573 & 0.0216 & 0.0302 & 0.0321 & 0.0356 & 0.004\\
    -2.78 & 0.1   & 0.1147 & 0.1101 & 0.1099 & 0.1088 & 0.0626 & 0.0733 & 0.0757 & 0.0824 & 0.006\\
    -1.91 & 0.3   & 0.3100  & 0.2966 & 0.3060 & 0.3029 & 0.2665 & 0.2721 & 0.2808 & 0.2827 & 0.009\\
    -1.27 & 0.5   & 0.5000   & 0.4959 & 0.4969 & 0.4996 & 0.4841 & 0.4834 & 0.4985 & 0.4899 & 0.01\\
    -0.59 & 0.7   & 0.7025 & 0.7013 & 0.7099 & 0.7018 & 0.6990 & 0.6992 & 0.7109 & 0.6975 & 0.009\\
    0.45  & 0.9   & 0.9107 & 0.9061 & 0.9071 & 0.9036 & 0.9014 & 0.9040 & 0.9059 & 0.9001 & 0.006\\
    0.98  & 0.95  & 0.9566 & 0.9546 & 0.9538 & 0.9546 & 0.9503 & 0.9527 & 0.9526 & 0.9512 & 0.004\\
    2.02  & 0.99  & 0.9929 & 0.994 & 0.9903 & 0.9914 & 0.9890 & 0.9908 & 0.9901 & 0.9894 & 0.002\\
    \bottomrule
    \end{tabular}%
  \label{tab:addlabel}%
\end{table}%

\includegraphics[width=0.5\textwidth,angle=0]{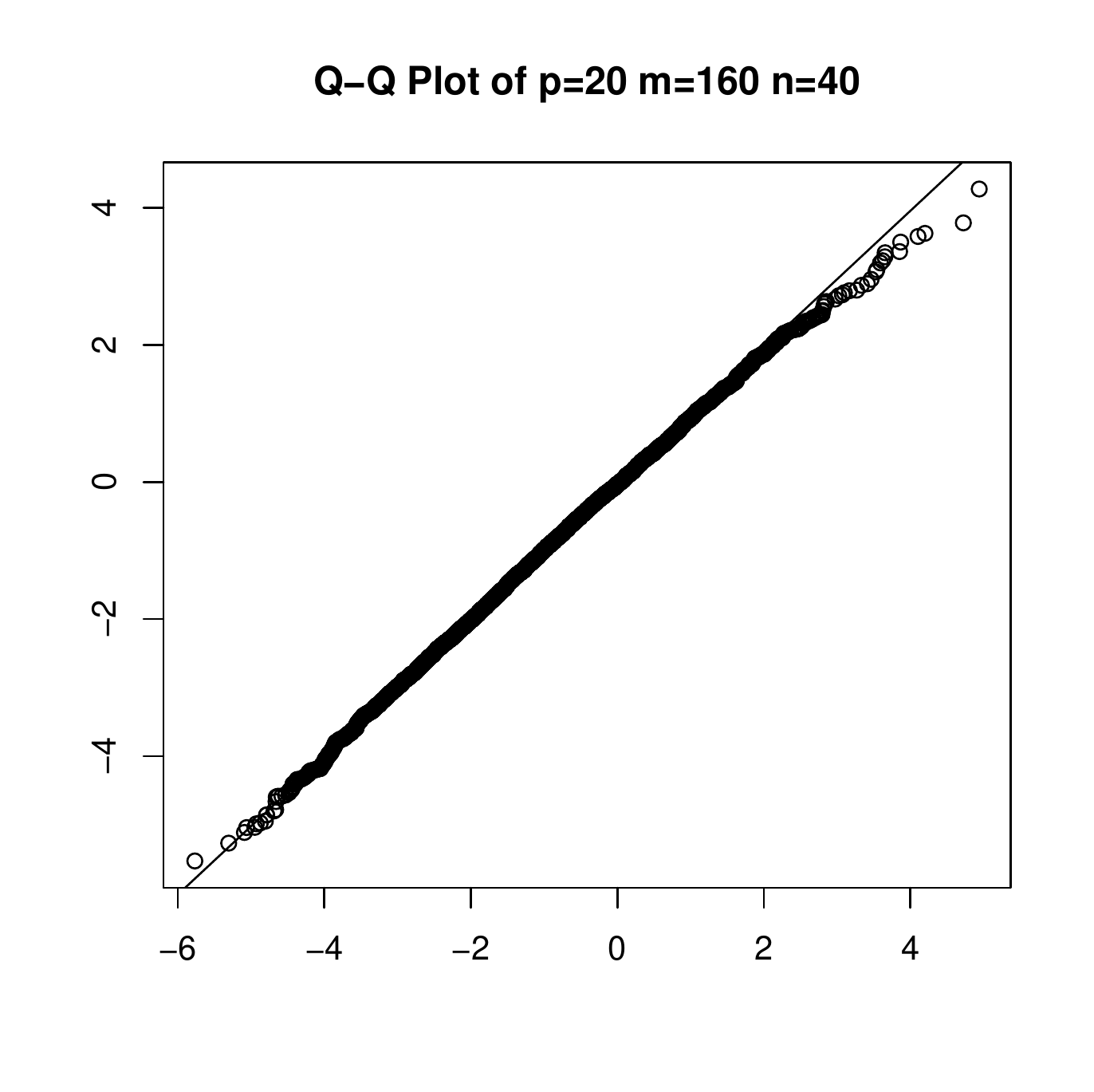}
\includegraphics[width=0.5\textwidth,angle=0]{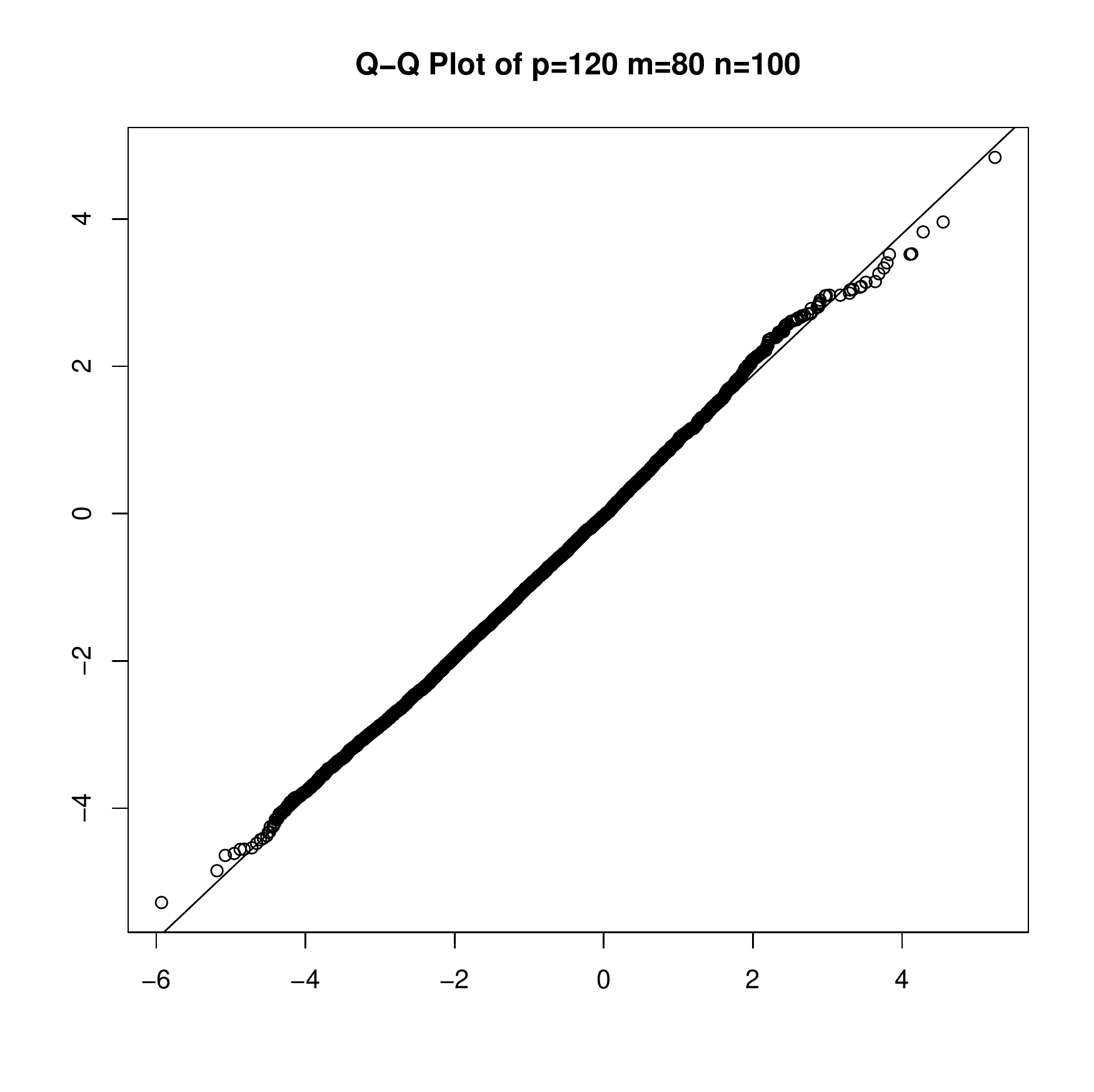}

% Table generated by Excel2LaTeX from sheet 'MB1'
\begin{table}[htbp]
  \centering
  \caption{Standard quantiles for several triples (p,m,n): Continuous uniform distribution $U(-\sqrt{3},\sqrt{3})$}
    \begin{tabular}{rrrrrrrrrrr}
    \toprule
    & & \multicolumn{4}{c}{Initial triple $M_0$=(30,80,40)}  & \multicolumn{4}{c}{Initial triple $M_1$=(80,40,50)} & \\ \cmidrule(r){3-6} \cmidrule(r){7-10}
    Percentile & TW   &  $M_0$  &   $2M_0$    &   $3M_0$    &  $4M_0$    &    $M_1$   &  $2M_1$      &   $3M_1$     &   $4M_1$     & 2*SE \\
    \midrule
    -3.9  & 0.01  & 0.0098 & 0.0117 & 0.0122 & 0.0120 & 0.0101 & 0.0087 & 0.0092 & 0.0096 & 0.002 \\
    -3.18 & 0.05  & 0.0612 & 0.0632 & 0.0606 & 0.0592 & 0.0514 & 0.0462 & 0.0492 & 0.0482 & 0.004 \\
    -2.78 & 0.1   & 0.1205 & 0.1243 & 0.1208 & 0.1197 & 0.1023 & 0.0942 & 0.1033 & 0.0992 & 0.006 \\
    -1.91 & 0.3   & 0.3644 & 0.3542 & 0.351 & 0.3432 & 0.3132 & 0.2946 & 0.3101 & 0.3017 & 0.009 \\
    -1.27 & 0.5   & 0.5767 & 0.5575 & 0.5563 & 0.5496 & 0.516 & 0.5073 & 0.5151 & 0.5069 & 0.01 \\
    -0.59 & 0.7   & 0.7728 & 0.7540 & 0.7443 & 0.7440 & 0.7182 & 0.7123 & 0.714 & 0.7171 & 0.009 \\
    0.45  & 0.9   & 0.9397 & 0.9243 & 0.9181 & 0.9202 & 0.9141 & 0.9068 & 0.9071 & 0.9059 & 0.006 \\
    0.98  & 0.95  & 0.9722 & 0.9672 & 0.9599 & 0.9614 & 0.9584 & 0.9538 & 0.9556 & 0.9534 & 0.004 \\
    2.02  & 0.99  & 0.9959 & 0.9941 & 0.993 & 0.9922 & 0.9932 & 0.9912 & 0.9919 & 0.9916 & 0.002 \\
    \bottomrule
    \end{tabular}%
  \label{tab:addlabel}%
\end{table}%
\includegraphics[width=0.5\textwidth,angle=0]{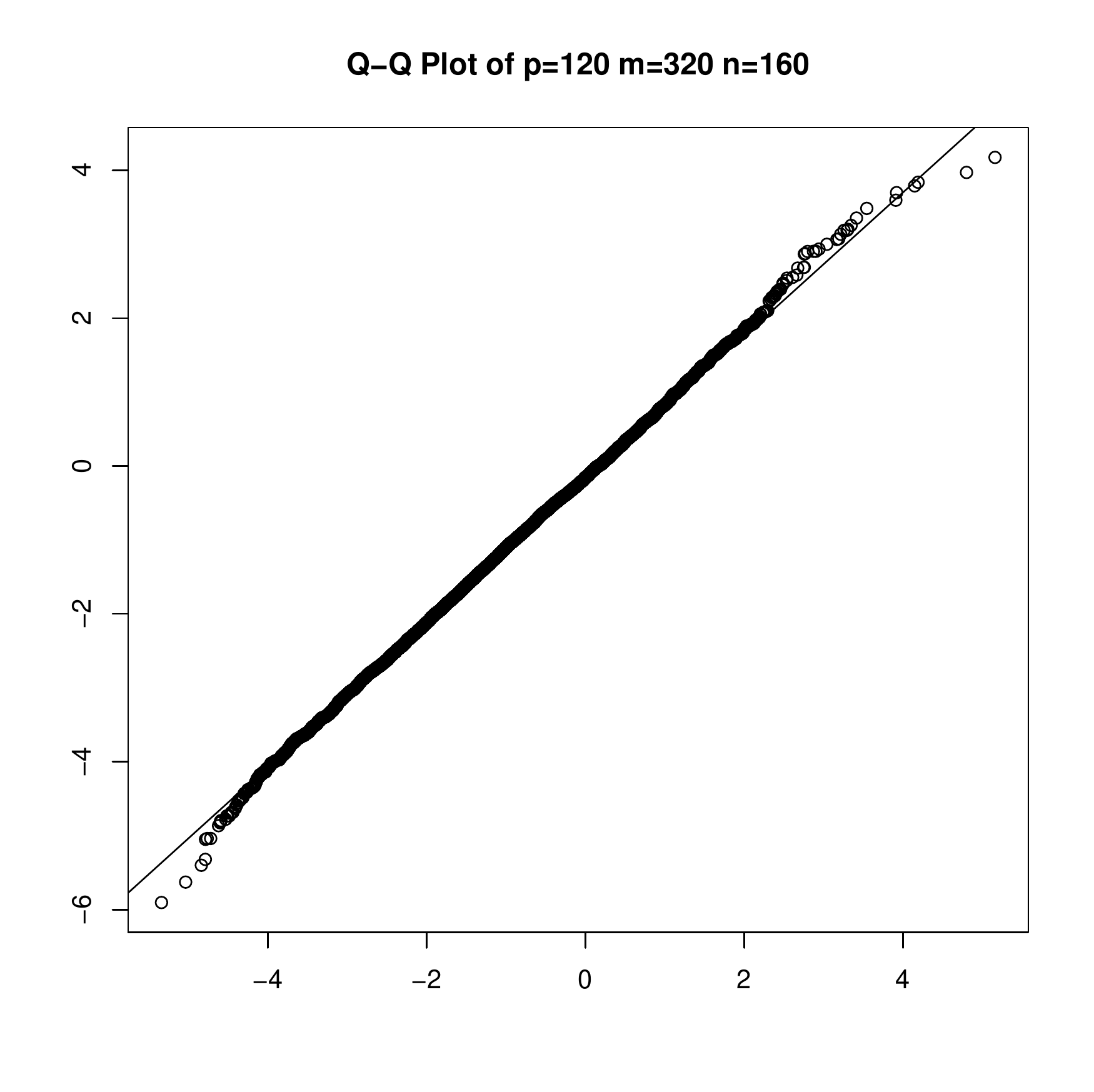}
\includegraphics[width=0.5\textwidth,angle=0]{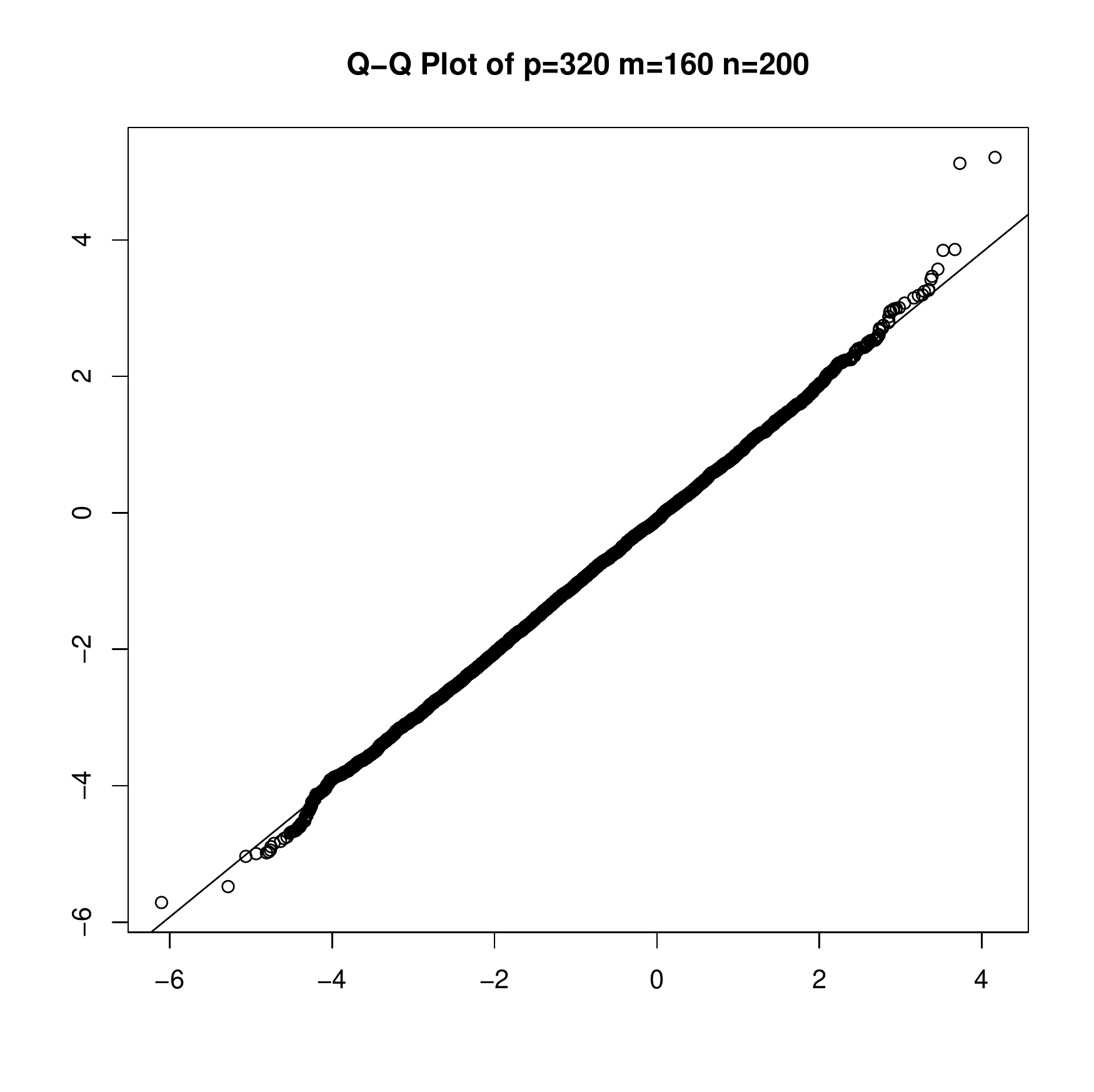}

When considering the test of equality of two population covariance matrices since $\mathbf{\Sigma}_1$ is assumed to be invertible in the null case $\mathbf{\Sigma}_1=\mathbf{\Sigma}_2$, without loss of generality, we may assume that $\mathbf{\Sigma}_1=\mathbf{\Sigma}_2=\bbI$.
 Therefore one may refer to Table one as well for the size of the test for the nominal significant levels.

\subsubsection{Power}

We study the power of the test and consider the alternative case
$$\bbZ_1=\mathbf{\Sigma}\bbX, \ \ \bbZ_2=\bbY,$$
where $\mathbf{\Sigma}\neq \bbI$.

When $\bbY\bbY^*$ is invertible we choose
$\mathbf{\Sigma}=\bbI+\tau\frac{\frac{p}{m}-r}{1-\frac{p}{m}}\bbe_1\bbe_1^T,$ where $r=\sqrt{\frac{p}{m}+\frac{p}{n}-\frac{p^2}{mn}}$.
The reason why we choose the factor $\frac{\frac{p}{m}-r}{1-\frac{p}{m}}$ is that when $\tau>1$ it is a spiked F matrix and the largest eigenvalue converges to normal distribution weakly by Proposition 11 of \cite{DJO15}.

When $\bbY\bbY^*$ is not invertible by Theorem 1.2 of \cite{BJ06} we can find out that the smallest non-zero eigenvalue of $\frac{1}{m}\mathbf{\Sigma}^{-1/2} \bbY\bbY^* \mathbf{\Sigma}^{-1/2}$ is not spiked for the above $\mathbf{\Sigma}$. So it is hard to get a spiked F matrix. Therefore we use another matrix
$$\mathbf{\Sigma}=\left( {\begin{array}{*{20}c}
  {1 } & {} & {} & {} & {} & {} & {}  \\
  {} & {\omega } & {} & {} & {} & {} & {}  \\
  {} & {} &  1  & {} & {} & {} & {}  \\
  {} & {} & {} & {\omega } & {} & {} & {}  \\
  {} & {} & {} & {} & \ddots & {} & {}  \\
  {} & {} & {} & {} & {} &  1  & {}  \\
  {} & {} & {} & {} & {} & {} & \omega  \\
\end{array} } \right).$$
%which implies that $\mathbb{P}(\lambda(\mathbf{\Sigma})=1)=\mathbb{P}(\lambda(\mathbf{\Sigma})=\omega)=\frac{1}{2}$.

In Tables 4-6 the data \bbX \ and \bbY \  are generated as in Tables 1-3 and the nominal significant level of our test is 5\%.% which means the rejection region of the test is 0.98 by TW law.

% Table generated by Excel2LaTeX from sheet 'normh1keni1'
\begin{table}[htbp]\label{ta4}
  \centering
  \caption{Power of several triples(p,m,n): Gaussian distribution}
    \begin{tabular}{rrrrrrrrrr}
    \toprule
    &\multicolumn{4}{c}{Initial triple $M_0$=(5,40,10)}  & &\multicolumn{4}{c}{Initial triple $M_1$=(30,20,25)} \\ \cmidrule(r){2-5} \cmidrule(r){7-10}
   $\tau$&$M_0$  &   $2M_0$    &   $3M_0$    &  $4M_0$&  $\omega$  &    $M_1$   &  $2M_1$      &   $3M_1$     &   $4M_1$      \\
    \midrule
    0.5 &0.0672 & 0.0585 & 0.0563 & 0.0593 & 0.3&0.2178 & 0.4934 & 0.7071 & 0.8419 \\
    2 &0.2763 & 0.3801 & 0.4551 & 0.5067 & 0.6&0.0574 & 0.1332 & 0.2241 & 0.3106 \\
    4 &0.6291 & 0.816 & 0.9072 & 0.9567 & 2&0.1037 & 0.2166 & 0.3463 & 0.5029 \\
    6 &0.8162 & 0.9543 & 0.988 & 0.9967 & 3&0.2242 & 0.5521 & 0.8156 & 0.9537 \\
    \bottomrule
    \end{tabular}%
  \label{tab:addlabel}%
\end{table}%
\newpage
\begin{table}[htbp]\label{ta5}
  \centering
    \centering
  \caption{Power of several triples(p,m,n): Discrete distribution }
    \begin{tabular}{rrrrrrrrrr}
    \toprule
    &\multicolumn{4}{c}{Initial triple $M_0$=(5,40,10)}  & &\multicolumn{4}{c}{Initial triple $M_1$=(30,20,25)} \\ \cmidrule(r){2-5} \cmidrule(r){7-10}
   $\tau$&$M_0$  &   $2M_0$    &   $3M_0$    &  $4M_0$&  $\omega$  &    $M_1$   &  $2M_1$      &   $3M_1$     &   $4M_1$      \\
    \midrule
     0.5 &0.0674 & 0.0573 & 0.0576 & 0.0595 &0.3& 0.2101 & 0.4883 & 0.7024 & 0.8425 \\
    2 &0.3045 & 0.397 & 0.4561 & 0.5171 & 0.6&0.057 & 0.1382 & 0.2176 & 0.3078 \\
    4&0.647 & 0.8137 & 0.8984 & 0.9478 & 2&0.1055 & 0.2232 & 0.3504 & 0.4974 \\
    6 &0.8147 & 0.943 & 0.9813 & 0.9936 & 3&0.2254 & 0.5487 & 0.8211 & 0.9529 \\
    \bottomrule
    \end{tabular}%
  \label{tab:addlabel}%
\end{table}%
\begin{table}[htbp]\label{ta6}
  \centering
  \caption{Power of several triples(p,m,n): Continuous uniform distribution $U(-\sqrt{3},\sqrt{3})$}
    \begin{tabular}{rrrrrrrrrr}
    \toprule
    &\multicolumn{4}{c}{Initial triple $M_0$=(30,80,40)}  & &\multicolumn{4}{c}{Initial triple $M_1$=(80,40,50)} \\ \cmidrule(r){2-5} \cmidrule(r){7-10}
   $\tau$&$M_0$  &   $2M_0$    &   $3M_0$    &  $4M_0$&  $\omega$  &    $M_1$   &  $2M_1$      &   $3M_1$     &   $4M_1$      \\
    \midrule
    0.5 &0.2283 & 0.3188 & 0.3977 & 0.4662 & 0.3&0.9965 & 1.0000     & 1.0000     & 1.0000 \\
  2&1.0000 & 1.0000& 1.0000  & 1.0000 & 0.6&0.7112 & 0.9623 & 0.9964 & 0.9999 \\
   4&1.0000 & 1.0000 & 1.0000  & 1.0000 & 2&0.9257 & 1.0000    & 1.0000    & 1.0000 \\
    6&1.0000& 1.0000 &1.0000 & 1.0000 & 3&1.0000 &1.0000 & 1.0000     & 1.0000 \\
    \bottomrule
    \end{tabular}%
  \label{tab:addlabel}%
\end{table}%
\newpage
In Tables 4-6 we can find that when $\tau=0.5<1$ $(\mathbf{\Sigma}^{-1/2} \bbY\bbY^* \mathbf{\Sigma}^{-1/2})\bbX\bbX^*$ is not a spiked F matrix and the power is poor. When $\tau>1$ it is a spiked F matrix and the power increases
with the dimension and $\tau$. This phenomenon is due to the fact that it may not cause significant change to the largest eigenvalue of $F$ matrix when finite rank perturbation is weak enough.  This phenomenon has been widely discussed for sample covariance matrices, see \cite{FP2009} and \cite{BPZ2014a}. For the spiked F matrix one can refer to \cite{DJO15} and \cite{WY}. For the non-invertible case when $\mathbf{\Sigma}$ is far away from \bbI ($\omega=0.3$ or 3) the power becomes better. This is because when the empirical spectral distribution (ESD) of $\mathbf{\Sigma}$ is very different from the M-P law $\lambda_1$ may tend to another point $\mu_{\mathbf{\Sigma}}$ instead of $\mu_p$. Then we may gain good power because $n^{2/3}(\mu_{\mathbf{\Sigma}}-\mu_p)$ may tend to infinity.

\section{Proof of Part(i) of Theorem \ref{t1}}
\subsection{ Two key Lemmas }

%We will give some definitions, lemmas and tools.

%(D4) $\varphi:=(\log p)^{\log \log p}$.
%
%\begin{deff} [Matrix Norms]

%\end{rmk}

%\begin{deff}
This subsection is to first prove two key lemmas for proving part(i) of Theorem \ref{t1}. We begin with some notation and definitions. Throughout the paper we use $M,M_0,M_0',M_0'', M_{1},M_1''$ to denote some
generic positive constants whose values may differ from line to line. We also use $D$ to denote sufficiently large positive constants whose values may differ from line to line. We say that an event $\Lambda$ holds with high probability if for any big positive constant $D$
\begin{equation*}
P(\Lambda^c) \leq n^{-D},
\end{equation*}
 for sufficiently large n. Recall the definition of $\gamma_j$ in (\ref{2.6}). Let $c_{p,0}\in [0, a_{p})$ satisfy
\begin{equation}\label{2.10}
\frac{1}{p} \sum_{j=1}^p (\frac{c_{p,0}}{\gamma_{j}-c_{p,0}})^2=\frac{n}{p}.
\end{equation}
%\end{deff}
Existence of $c_{p,0}$ will be verified in Lemma \ref{0312-1} below. Moreover define
%We will prove the existence of $c_{p,0}$ in Lemma \ref{0312-1}.
\begin{equation}\label{2.11}
\mu_{p,0}=\frac{1}{c_{p,0}}(1+\frac{1}{n}\sum_{j=1}^p (\frac{c_{p,0}}{\gamma_{j}-c_{p,0}})),\quad
%\end{equation}
%and
%\begin{equation}\label{2.12}
\frac{1}{\sigma_{p,0}^3}=\frac{1}{c_{p,0}^3}(1+\frac{1}{n}\sum_{j=1}^p (\frac{c_{p,0}}{\gamma_{j}-c_{p,0}})^3).
\end{equation}

Set $\bbA_p=\frac{1}{m} \bbY\bbY^*$ and $\bbB_p=\frac{1}{n} \bbX\bbX^*$.
%\begin{deff}\label{0313-1}
Rank the eigenvalues of the matrix $\bbA_{p}$ as $\hat{\gamma}_{1} \geq \hat{\gamma}_{2} \geq \cdots \geq \hat{\gamma}_{p}$. Let $\hat{c}_{p} \in [0, \hat{\gamma}_{p})$ satisfy
\begin{equation}\label{2.13}
\frac{1}{p} \sum_{j=1}^p (\frac{\hat{c}_{p}}{\hat{\gamma}_{j}-\hat{c}_{p}})^2=\frac{n}{p}.
\end{equation}
%Moreover we will prove
The existence of $\hat{c}_{p}$ with high probability will be given in Lemma \ref{0313-2} below. Moreover set
\begin{equation}\label{2.14}
\hat{\mu}_{p}=\frac{1}{\hat{c}_{p}}(1+\frac{1}{n}\sum_{j=1}^p (\frac{\hat{c}_{p}}{\hat{\gamma}_{j}-\hat{c}_{p}})),\quad
%\end{equation}
%and
%\begin{equation}\label{2.15}
\frac{1}{\hat{\sigma}_{p}^3}=\frac{1}{\hat{c}_{p}^3}(1+\frac{1}{n}\sum_{j=1}^p (\frac{\hat{c}_{p}}{\hat{\gamma}_{j}-\hat{c}_{p}})^3).
\end{equation}
We now discuss the properties of $c_{p},c_{p,0},\hat{c}_{p},\mu_{p},\mu_{p,0},\hat{\mu}_{p},\sigma_p,\sigma_{p,0}$ defined (\ref{2.7})-(\ref{2.9}), (\ref{2.10})- (\ref{2.14}) in the next two lemmas. These lemmas are crucial to the proof strategy which transforms $F$ matrices into an appropriate sample covariance matrix. %in lemma \ref{0313-3}.

\begin{lem}\label{0312-1}
Under the conditions in Theorem \ref{t1}, there exists a constant $M_0$ such that
\begin{equation}\label{3.4}
\sup_p\{\frac{c_{p}}{a_{p}-c_{p}} \}\leq M_0,\quad \sup_p\{\frac{c_{p,0}}{a_{p}-c_{p,0}} \} \leq M_0,
\end{equation}
\begin{equation}\label{3.5}
\lim_{p \rightarrow \infty} n^{2/3}|\mu_{p}-\mu_{p,0}|=0,
\end{equation}
\begin{equation}\label{3.6}
\lim_{p \rightarrow \infty} \frac{\sigma_{p}}{\sigma_{p,0}}=1,\quad
%\end{equation}
%and
%\begin{equation}\label{3.7}
\limsup_p  \frac{c_{p,0}}{a_{p}}<1.
\end{equation}

\end{lem}

\begin{proof}
The exact expression of $c_p$ in (\ref{2.7}) can be figured out under the conditions in Theorem \ref{t1} (see Section 5). In fact, when $n=p$, from (\ref{4.6}) below we have
$$c_p=\frac{(m-p)^2}{2(m+p)m}.$$
Recall the definition of $a_p$ in (\ref{2.5}). It follows that
$$\frac{c_p}{a_p}=\frac{(\sqrt{m}+\sqrt{p})^2}{2(m+p)},$$
which further implies that
\begin{equation}\label{3.9}
\limsup_p  \frac{c_{p}}{a_{p}}<1.
\end{equation}
In view of this, there are two constants $M_0>0$ and $M_0''>0$ such that
\begin{equation}\label{3.10}
\sup_p\{\frac{c_{p}}{a_{p}-c_{p}} \} \leq M_0,\quad
%\begin{equation}\label{3.11}
\inf_p\{c_{p} \} \geq M_0''.
\end{equation}

%\end{equation}

%\begin{equation*}
%\frac{c_p^2}{v^2} \geq \frac{n}{p} \geq \frac{1}{d_3'} >0.
%\end{equation*}
%
%So we can  get a constant  $M_0''>0$ such that

When $n \neq p$, from (\ref{5.3}) below we have
$$c_p=\frac{n(m+p)(m+n-p)-(m+2n-p) \sqrt{mnp(m+n-p)}}{m(n-p)(m+n)}.$$ Using the above expression for $c_p$ one may
similarly obtain (\ref{3.9})-(\ref{3.10}) as well but with tedious calculations and we ignore details here.

Now we define a function $f_1(x)$ by
\begin{equation}\label{3.13}
f_1(x)=\frac{1}{p} \sum_{j=1}^p (\frac{x}{\gamma_{j}-x})^2.
\end{equation}
We claim that there exists $c_{p,0} \in (0, a_p)$ so that
\begin{equation}\label{a1}
f_1(c_{p,0}) = \frac{n}{p}.
\end{equation}
Indeed, due to (\ref{3.9}) we obtain
\begin{equation*}
f_1(c_{p})=\frac{1}{p} \sum_{j=1}^p (\frac{c_{p}}{\gamma_{j}-c_{p}})^2 \geq  \sum_{j=1}^p\int_{\gamma_{j}}^{\gamma_{j-1}} (\frac{c_{p}}{x-c_{p}})^2 \varrho_{p}(x)dx=\int_{\gamma_{p}}^{\gamma_{0}} (\frac{c_{p}}{x-c_{p}})^2 \varrho_{p}(x)dx.
\end{equation*}
This, together with (\ref{2.5}) and (\ref{2.7}), implies that
\begin{equation}\label{3.14}
f_1(c_{p}) \geq \frac{n}{p}
\end{equation}
and
\begin{equation}\label{3.14*}
 \frac{n}{p}=\int_{\gamma_{p}}^{\gamma_{0}} (\frac{c_{p}}{x-c_{p}})^2 \varrho_{p}(x)dx \geq \frac{1}{p} \sum_{j=1}^p (\frac{c_{p}}{\gamma_{j-1}-c_{p}})^2.
\end{equation}
Note that $f_1(x)$ is a continuous function on $(0,a_p)$ and $f_1(0)=0$. These, together with (\ref{3.14}), ensure that there exists $c_{p,0} \in (0,c_p]$ so that (\ref{a1}) holds, as claimed.

We next develop an upper bound for the difference between $c_{p,0} $ and $c_p$. It follows from (\ref{3.14}) and (\ref{3.14*}) that
$$
|f_1(c_{p})-\frac{n}{p}|=\Big|\frac{1}{p} \sum_{j=1}^p (\frac{c_{p}}{\gamma_{j}-c_{p}})^2-\int_{\gamma_{p}}^{\gamma_{0}} (\frac{c_{p}}{x-c_{p}})^2 \varrho_{p}(x)dx\Big|
$$$$
 \leq \frac{1}{p} \Big|\sum_{j=1}^p \Big((\frac{c_{p}}{\gamma_{j}-c_{p}})^2-(\frac{c_{p}}{\gamma_{j-1}-c_{p}})^2\Big)\Big| \leq \frac{2c_{p}^2(b_{p}-c_{p})}{p(a_{p}-c_{p})^4} \sum_{j=1}^p |\gamma_{j}-\gamma_{j-1}|\leq \frac{2(M_0)^4b_{p}(b_{p}-a_{p})}{(M_0'')^2 p},
$$
where the last inequality uses (\ref{3.9})-(\ref{3.10}). With $M_1'=\frac{2(M_0)^4b_{p}(b_{p}-a_{p})}{(M_0'')^2}$ the above inequality becomes
\begin{eqnarray}\label{3.15}
&&|f_1(c_{p})-\frac{n}{p}| \leq \frac{M_1'}{p}.
\end{eqnarray}

%\begin{eqnarray*}
%&&|f_1(c_{p})-\frac{n}{p}|\\
%&& \leq \frac{2c_{p}^2(b_{p}-c_{p})}{p(a_{p}-c_{p})^4} \sum_{j=1}^p |\gamma_{j}-\gamma_{j-1}| \\
%&& \leq \frac{2(M_0)^4b_{p}(b_{p}-a_{p})}{(M_0'')^2 p}.
%\end{eqnarray*}

Moreover taking derivative of $f(x)$ in (\ref{3.13}) yields
\begin{eqnarray}\label{3.16}
f_1'(x)=\frac{1}{p} \sum_{j=1}^p (\frac{2x^2}{(\gamma_{j}-x)^3}+\frac{2x}{(\gamma_{j}-x)^2}).
\end{eqnarray}
When $0 < x < c_p$ (smaller than $a_{p}$) and $f_1(x) \geq \frac{n}{2p}$,
\begin{eqnarray}\label{3.19}
f_1'(x) > \frac{1}{p} \sum_{j=1}^p \frac{2x}{(\gamma_{j}-x)^2} \geq \frac{n}{p x} \geq \frac{n}{p a_{p}}.
\end{eqnarray}
%When $0 < x < a_{p}$,
%\begin{eqnarray}\label{3.17}
%$f_1'(x) > 0$ evidently.
%%\end{eqnarray}
%It also follows from (\ref{3.10}) that for $0< x \leq c_{p}$
%\begin{equation}\label{3.18}
%\sup_p\{\frac{x}{a_{p}-x} \} \leq M_0.
%\end{equation}
%\begin{eqnarray*}
%f_1(c_{p,0}) = \frac{n}{p}.
%\end{eqnarray*}
When $c_{p,0} < x \leq c_p$ we always have $f_1(x)\geq \frac{n}{p}$ via (\ref{a1}) because $f_1'(x)> 0$ by (\ref{3.16}).  Via (\ref{3.15}) and (\ref{3.19}) we then obtain from the mean value theorem that
\begin{equation}\label{3.20}
|c_{p,0}-c_{p}| \leq \frac{M_1' a_{p}}{n}.
\end{equation}
This, together with (\ref{3.10}), implies that there is a constant $M_1>0$ such that when $p$ is big enough,
\begin{equation}\label{3.21}
M_1 < c_{p,0} \leq c_{p}.
\end{equation}

We conclude from (\ref{2.8}), (\ref{2.11}), (\ref{3.10}), (\ref{3.20}) and (\ref{3.21}) that
$$
|\mu_{p}-\mu_{p,0}| \leq |\frac{1}{c_{p}}-\frac{1}{c_{p,0}}|+\frac{1}{n}\sum_{j=1}^p \max\{(|\frac{1}{\gamma_{j}-c_{p,0}}-\frac{1}{\gamma_{j}-c_{p}}|,|\frac{1}{\gamma_{j}-c_{p,0}}-\frac{1}{\gamma_{j-1}-c_{p}}|\}
$$$$
 \leq \frac{|c_{p}-c_{p,0}|}{M_1^2}+\frac{1}{n}\sum_{j=1}^p \frac{(|\gamma_{j}-\gamma_{j-1}|+|c_{p}-c_{p,0}|)M_0^2}{M_1^2}
 $$
 $$
\leq \frac{M_1' a_{p}}{n M_1^2}+\frac{b_{p}-a_{p}}{n M_1^2}+ \frac{p M_0^2 M_1' a_{p}}{n^2 M_1^2}=O(\frac{1}{p}).
$$
%\begin{eqnarray}\label{3.22}
%&&|\mu_{p}-\mu_{p,0}|=O(\frac{1}{p}).
%\end{eqnarray}
Similarly one can prove that
\begin{equation}\label{3.23}
|\frac{1}{\sigma_{p}^3}-\frac{1}{\sigma_{p,0}^3}|=O(\frac{1}{p}).
\end{equation}
(\ref{3.5}) and the first result in (\ref{3.6}) then follow. From (\ref{3.9}) and (\ref{3.20}) one can also obtain (\ref{3.4}) and the second result in (\ref{3.6}).
\end{proof}

\begin{lem}\label{0313-2}
Under the conditions in Theorem \ref{t1}, for any $\zeta>0$ there exists a constant $M_{\zeta} \geq M_0$ such that
\begin{equation}\label{3.24}
\sup_p\{\frac{\hat{c}_{p}}{\hat{\gamma}_{p}-\hat{c}_{p}} \}\leq M_{\zeta},\quad
%\end{equation}
%
%\begin{equation}\label{3.25}
\limsup_p  \frac{\hat{c}_{p}}{\hat{\gamma}_p}<1,
\end{equation}
and
\begin{equation}\label{3.26}
\lim_{p \rightarrow \infty} n^{2/3}|\hat{\mu}_{p}-\mu_{p,0}|=0,\quad
%\end{equation}
%
%\begin{equation}\label{3.27}
\lim_{p \rightarrow \infty} \frac{\hat{\sigma}_{p}}{\sigma_{p,0}}=1.
\end{equation}
hold with high probability. Indeed (\ref{3.24}) and (\ref{3.26}) hold on the event $S_{\zeta}$ defined by
%\begin{equation}\label{3.1}
%and an event  by
\begin{eqnarray}\label{3.30}
S_{\zeta}=\{\forall j, 1 \leq j \leq p, |\hat{\gamma}_j-\gamma_j| \leq p^{\zeta} p^{-2/3} \tilde{j}^{-1/3} \},
\end{eqnarray}
where $\zeta$ is a sufficiently small positive constant and $
\tilde{j}=\min \{ \min\{m,p\}+1-j,j \}.
%\end{equation}
$
\end{lem}

\begin{proof}
Define a function $\hat{f}(x)$ by
\begin{equation}\label{3.28}
\hat{f}(x)=\frac{1}{p} \sum_{j=1}^p (\frac{x}{\hat{\gamma}_j-x})^2.
\end{equation}
From (\ref{2.10}) and (\ref{3.13}) we have
\begin{equation}\label{3.29}
f_1(c_{p,0}) = \frac{n}{p}.
\end{equation}

The first aim is to find $\hat{c}_{p} \in [0, \hat{\gamma}_p)$ to satisfy
\begin{equation}\label{a2}
\hat{f}(\hat{c}_{p}) = \frac{n}{p}.
\end{equation}
When $\zeta$ is small enough we conclude from (\ref{3.28}), (\ref{3.29}) and (\ref{3.31}) that on the event $S_{\zeta}$
$$
|\hat{f}(c_{p,0})- \frac{n}{p}|=|\hat{f}(c_{p,0})-f_1(c_{p,0})|
=|\frac{1}{p} \sum_{j=1}^p ((\frac{c_{p,0}}{\hat{\gamma}_j-c_{p,0}})^2-(\frac{c_{p,0}}{\gamma_j-c_{p,0}})^2)|
$$$$ \leq \frac{c_{p,0}^2}{p} \max_j \{ \frac{|\hat{\gamma}_j+\gamma_j-2c_{p,0}|}{(\hat{\gamma}_j-c_{p,0})^2(\gamma_j-c_{p,0})^2} \}\sum_{j=1}^p |\hat{\gamma}_j-\gamma_j|
$$
\begin{equation}\label{3.32}
\leq\frac{c_{p,0}^2}{p}  \max_j \{ \frac{|p^{\zeta} p^{-2/3}\tilde{j}^{-1/3}|+2|\gamma_j-c_{p,0}|}{(-p^{\zeta} p^{-2/3}\tilde{j}^{-1/3}+\gamma_j-c_{p,0})^2(\gamma_j-c_{p,0})^2} \} \sum_{j=1}^p p^{\zeta} p^{-2/3} \tilde{j}^{-1/3}
=O(p^{\zeta-1}),
\end{equation}
%\label{1223-1}
%Suppose that $m$,$n$ and $p$ satisfy the condition in Theorem \ref{t1} and E is a real number. For any $1 \leq j \leq p$, let
%
%\begin{equation}\label{3.1}
%\tilde{j}=\min \{ \min\{m,p\}+1-j,j \}.
%\end{equation}
%With Lemma \ref{1223-1},

%When $\zeta$ is small enough, on the event $S_{\zeta}$,
%\begin{eqnarray*}
%&&|\hat{f}(c_{p,0})- \frac{n}{p}| \leq \\
%&&\frac{c_{p,0}^2}{p}  \max_j \{ \frac{|p^{\zeta} p^{-2/3}\tilde{j}^{-1/3}+2\gamma_j-2c_{p,0}|}{(-p^{\zeta} p^{-2/3}\tilde{j}^{-1/3}+\gamma_j-c_{p,0})^2(\gamma_j-c_{p,0})^2} \} \sum_{j=1}^p p^{\zeta} p^{-2/3} \tilde{j}^{-1/3}.
%\end{eqnarray*}
where the last step uses the fact that via (\ref{3.4}) and (\ref{3.21})
\begin{eqnarray*}
\max_j \{ \frac{|p^{\zeta} p^{-2/3}\tilde{j}^{-1/3}+2\gamma_j-2c_{p,0}|}{(-p^{\zeta} p^{-2/3}\tilde{j}^{-1/3}+\gamma_j-c_{p,0})^2(\gamma_j-c_{p,0})^2} \} \leq M.
\end{eqnarray*}
%It implies that
%\begin{eqnarray}\label{3.32}
%|\hat{f}(c_{p,0})- \frac{n}{p}| =O(p^{\zeta-1}).
%\end{eqnarray}

Taking derivative of (\ref{3.28}) yields
\begin{eqnarray}\label{3.33}
&&\hat{f}'(x)=\frac{1}{p} \sum_{j=1}^p (\frac{2x^2}{(\hat{\gamma}_j-x)^3}+\frac{2x}{(\hat{\gamma}_j-x)^2}).
\end{eqnarray}
%When $0 < x < \hat{\gamma}_p$,
%\begin{eqnarray}\label{3.34}
%&&\hat{f}'(x)>0.
%\end{eqnarray}
When $0< c_{p,0}-p^{-1/2}<x<c_{p,0}+p^{-1/2}$ from (\ref{3.4}) and (\ref{3.21}) we have on the event $S_{\zeta}$
\begin{eqnarray}
&&|\hat{f}(c_{p,0})-\hat{f}(x)|=\frac{1}{p} |\sum_{j=1}^p \frac{c_{p,0}^2(\hat{\gamma}_j-x)^2-x^2(\hat{\gamma}_j-c_{p,0})^2}{(\hat{\gamma}_j-x)^2(\hat{\gamma}_j-c_{p,0})^2}|\nonumber\\
&& = \frac{1}{p} |\sum_{j=1}^p \frac{(c_{p,0}-x)\hat{\gamma}_j[c_{p,0}(\hat{\gamma}_j-x)+x(\hat{\gamma}_j-c_{p,0})]}{(\hat{\gamma}_j-x)^2(\hat{\gamma}_j-c_{p,0})^2}|=O(p^{-1/2}).\label{a3}
\end{eqnarray}
%From the property of the event $S_{\zeta}$ (\ref{3.30}),
%
%\begin{eqnarray*}
%&&|\hat{f}(c_{p,0})-\hat{f}(x)|=O(p^{-1/2}).
%\end{eqnarray*}
When $0 < x < \hat{\gamma}_p$ we have
\begin{eqnarray*}
&&\hat{f}'(x)>\frac{1}{p} \sum_{j=1}^p \frac{2x}{(\hat{\gamma}_j-x)^2}= \frac{2}{x} \hat{f}(x)> \frac{2}{(c_{p,0}+p^{-1/2})} \hat{f}(x).
\end{eqnarray*}
In view of this, (\ref{3.32}) and (\ref{a3}) there exists $M_2>0$ so that
\begin{eqnarray}\label{3.35}
\hat{f}'(x)>M_2,
\end{eqnarray}
for sufficiently large $p$ when $0 < x < \hat{\gamma}_p$.
On the event $S_{\zeta}$, applying the mean value theorem yields
\begin{equation*}
\hat{f}(c_{p,0}-p^{-1/2}) < \hat{f}(c_{p,0})-M_2p^{-1/2}
\end{equation*}
and
\begin{equation*}
\hat{f}(c_{p,0}+p^{-1/2}) >\hat{f}(c_{p,0})+M_2p^{-1/2}.
\end{equation*}
It follows from (\ref{3.32}) that when $p$ is large enough,
\begin{equation*}
\hat{f}(c_{p,0}-p^{-1/2}) < \frac{n}{p} < \hat{f}(c_{p,0}+p^{-1/2}).
\end{equation*}
Since $\hat{f}(x)$ is continuous on $(0,\hat{\gamma}_p)$ there is $\hat{c}_{p} \in [0, \hat{\gamma}_p)$ ($c_{p,0}\leq c_p<a_p=\gamma_p$ by Lemma \ref{0312-1}) so that (\ref{a2}) holds
%\begin{equation*}
%\hat{f}( \hat{c}_{p})  =\frac{n}{p}
%\end{equation*}
and
\begin{equation*}
c_{p,0}-p^{-1/2} < \hat{c}_{p} < c_{p,0}+p^{-1/2}.
\end{equation*}

From (\ref{3.32}), (\ref{a2}) and (\ref{3.35}) we have
\begin{eqnarray}\label{3.36}
&&|c_{p,0}-\hat{c}_{p}|=O(p^{\zeta-1}).
\end{eqnarray}
Recall $a_p=\gamma_p$. The second inequality in (\ref{3.24}) holds on the event $S_{\zeta}$ due to (\ref{3.6}), (\ref{3.30}) and (\ref{3.36}).
Likewise on the event $S_{\zeta}$ in view of (\ref{3.4}) and (\ref{3.36}) there exists a constant
$M_{\zeta} \geq M_0$ such that
\begin{equation}\label{3.37}
\sup_p\{\frac{\hat{c}_{p}}{\hat{\gamma}_p-\hat{c}_{p}} \}\leq M_{\zeta},
\end{equation}
the first inequality in (\ref{3.24}).

Due to $\hat{c}_{p}<\hat{\gamma}_p$ and the definition of $\hat f(x)$ in (\ref{3.28}) we have
\begin{equation*}
(\frac{\hat{c}_{p}}{\hat{\gamma}_p-\hat{c}_{p}})^2 \geq \frac{n}{p},
\end{equation*}
which implies that
\begin{equation}\label{3.38}
\hat{c}_{p} \geq \frac{\sqrt{\frac{n}{p}} \hat{\gamma}_p}{1+\sqrt{\frac{n}{p}}}.
\end{equation}
It follows from (\ref{2.11}) and (\ref{2.14}) that
\begin{eqnarray*}
|\mu_{p,0}-\hat{\mu}_{p}|&\leq &|\frac{1}{c_{p,0}}-\frac{1}{\hat{c}_{p}}|+\frac{1}{n}\sum_{j=1}^p |\frac{1}{\gamma_j-c_{p,0}}-\frac{1}{\hat{\gamma}_j-\hat{c}_{p}}| \\
&& \leq \frac{|c_{p,0}-\hat{c}_{p}|}{c_{p,0}\hat{c}_{p}}+ \frac{1}{n} \frac{\sum_{j=1}^p(|\gamma_j-\hat{\gamma}_j|+|c_{p,0}-\hat{c}_{p}|)}{(\gamma_p-c_{p,0})(\hat{\gamma}_p-\hat{c}_{p})}.
\end{eqnarray*}
We then conclude from (\ref{3.36})-(\ref{3.38}) that on the event $S_{\zeta}$
\begin{eqnarray}\label{3.39}
|\mu_{p,0}-\hat{\mu}_{p}| =O(p^{\zeta-1}).
\end{eqnarray}
It's similar to prove that
\begin{equation}\label{3.40}
|\frac{1}{\hat{\sigma}_{p}^3}-\frac{1}{\sigma_{p,0}^3}|=O(p^{\zeta-1}).
\end{equation}
(\ref{3.26}) then holds on the event $S_{\zeta}$.  Moreover, by Theorem 3.3 of \cite{PY11}, for any small $\zeta>0$ and any $D>0$,
\begin{eqnarray}\label{3.31}
P(S_{\zeta}^c)\leq p^{-D}.
\end{eqnarray}
The proof is therefore complete.

\end{proof}

%
%This lemma can be found in .

\subsection{Proof of Part (i) of Theorem \ref{t1}}

\begin{proof}
%When $0<\lim_{p \rightarrow \infty} \frac{p}{m}<1$, we have $\lim_{p \rightarrow \infty} P(\sigma_{min}(\bbY\bbY^*)>0)=1$ by Lemma \ref{1223-1}.
%So we just need to consider the case that $\bbY\bbY^*$ is invertible.
 Recall the definition of the matrices $\bbA_p$ and $\bbB_p$ above (\ref{2.13}). Define a F matrix $\bbF=\bbA_p^{-1}\bbB_p$ whose largest eigenvalue is $\lambda_1$ according to the definition of $\lambda_1$ in Theorem \ref{t1}. It then suffices to find the asymptotic distribution of $\lambda_1$ to prove Theorem \ref{t1}.

%By Lemma \ref{1223-1}, for any $\zeta>0$ and $\theta>0$, there exists an event $S_{\zeta}$ such that
%
%\begin{eqnarray}\label{3.49}
%S_{\zeta}=\{\forall j,1 \leq j \leq p, |\hat{\gamma}_j-\gamma_j| \leq  p^{\zeta-2/3} \tilde{j}^{-1/3} \}
%\end{eqnarray}
%and
%\begin{eqnarray}\label{3.50}
%P(S_{\zeta}^c)\leq p^{-\theta}.
%\end{eqnarray}

Recalling the definition of the event $S_{\zeta}$ in (\ref{3.30}) we may write
\begin{eqnarray*}
P(\sigma_{p} n^{2/3} (\lambda_{1}-\mu_{p}) \leq s)=P\Big(\big(\sigma_{p} n^{2/3} (\lambda_{1}-\mu_{p}) \leq s\big) \bigcap \ S_{\zeta}\Big)+P\Big(\big(\sigma_{p} n^{2/3} (\lambda_{1}-\mu_{p}) \leq s\big) \bigcap S_{\zeta}^c\Big).
\end{eqnarray*}
This, together with (\ref{3.31}), implies that (\ref{2.17}) is equivalent to
\begin{equation}\label{3.51}
\lim_{p \rightarrow \infty} P\Big(\big(\sigma_{p} n^{2/3} (\lambda_{1}-\mu_{p}) \leq s\big) \bigcap \ S_{\zeta}\Big)=F_1(s).
\end{equation}
%So we just need to prove (\ref{3.51}).

%For any given $A_p \in S_{\zeta}$,
Write
\begin{eqnarray}\label{3.52}
&&\sigma_{p} n^{2/3} (\lambda_{1}-\mu_{p}) =\frac{\sigma_{p}}{\hat{\sigma}_{p}} \hat{\sigma}_{p} n^{2/3} (\lambda_{1}-\hat{\mu}_{p})+\sigma_{p}n^{2/3}(\hat{\mu}_{p}-\mu_{p}).
\end{eqnarray}
(see (\ref{2.13}) and (\ref{2.14}) for $\hat{\sigma}_{p}$ and $\hat{\mu}_{p}$). Note that the eigenvalues of $\bbA_p^{-1}$ are  $\frac{1}{\hat{\gamma}_1} \leq \frac{1}{\hat{\gamma}_2} \leq \cdots \leq \frac{1}{\hat{\gamma}_p}$.
Rewrite (\ref{2.13}) as %we can also get $\hat{c}_{p} \in [0, \hat{\gamma}_p)$ is the solution of the equation
\begin{equation}\label{3.53}
\frac{1}{p} \sum_{j=1}^p (\frac{\frac{1}{\hat{\gamma}_j}\hat{c}_{p}}{1-\frac{1}{\hat{\gamma}_p}\hat{c}_{p}})^2=\frac{n}{p}.
\end{equation}
Also recast (\ref{2.14}) as
\begin{equation}\label{3.54}
\hat{\mu}_{p}=\frac{1}{\hat{c}_{p}}(1+\frac{p}{n}\frac{1}{p}\sum_{j=1}^p\frac{\frac{1}{\hat{\gamma}_j}\hat{c}_{p}}{1-\frac{1}{\hat{\gamma}_p}\hat{c}_{p}}),\quad
%\end{equation}
%and
%\begin{equation}\label{3.55}
\frac{1}{\hat{\sigma}_{p}^3}=\frac{1}{\hat{c}_{p}^3}(1+\frac{p}{n}\frac{1}{p}\sum_{j=1}^p\frac{\frac{1}{\hat{\gamma}_j}\hat{c}_{p}}{1-\frac{1}{\hat{\gamma}_p}\hat{c}_{p}})^3.
\end{equation}

Up to this stage the result about the largest eigenvalue of the sample covariance matrices $\bbZ\bbZ^*\bold\Sigma$ with $\bold\Sigma$ being the population covariance matrix comes into play where
$\bbZ$ is of size $p\times n$ satisfying Condition 1 and $\bold\Sigma$ is of size $p\times p$. A key condition to ensure Tracy-Widom's law for the largest eigenvalue is that
if $\rho \in (0,1/\sigma_1)$ is the solution to the equation
\begin{equation}\label{3.44}
\int(\frac{t\rho}{1-t\rho})^2dF^{\bold\Sigma}(t)= \frac{n}{p}
\end{equation}
then
\begin{equation}\label{3.45}
\lim \sup_p \rho\sigma_1  <1,
\end{equation}
(one may see \cite{K2007}, Conditions 1.2 and 1.4 and Theorem 1.3 \cite{BPZ2014a}, Conditions 2.21 and 2.22 and Theorem 2.18 of \cite{KY14}).
Here $F^{\bold\Sigma}(t)$ denotes the empirical spectral distribution of $\Sigma$ and $\sigma_1$ means the largest eigenvalue of $\bold\Sigma$.
Now given $\bbA_p$, if we treat $\bbA_p^{-1}$ as $\bold\Sigma$, then (\ref{3.45}) is satisfied on the event $S_{\zeta}$ due to (\ref{2.13}) and (\ref{3.24}) in Lemma \ref{0313-2}.
%\begin{lem}[Theorem 2.18 of \cite{KY14}]\label{0313-3}
%Let $Q=TZZ^*T^*$ be an $p \times p$ matrix where $T$ is a $p \times p$ deterministic matrix and $Z$ is a $p \times n$ random matrix. The entries $\{z_{ij}\}$ of $Z$ are independent real-valued random variables such that
%
%\begin{equation}\label{3.41}
%E z_{ij}=0, E z_{ij}^2=\frac{1}{n}
%\end{equation}
%
%for all $i$ and $j$. We also assume that, for all $k \in N$ there is a constant $C_k$ such that
%
%\begin{equation}\label{3.42}
%E|\sqrt{n}z_{ij}|^k \leq C_k
%\end{equation}
%for all $i$ and $j$.
%
%We assume that $\Sigma=TT^*$ is positive defined and denote the eigenvalues of $\Sigma$ by
%
%\begin{equation*}
%\sigma_1 \geq \sigma_2 \geq \cdots \geq \sigma_p.
%\end{equation*}
%Let
%\begin{equation}\label{3.43}
%\hat{\rho}:= \frac{1}{p} \sum_{j=1}^p \delta_{\sigma_j}
%\end{equation}
%denote the empirical spectral density of $\Sigma$.
%
%
%
%Then we can get the largest eigenvalue $\mu_1$ of $Q$ satisfies
%
%
%\begin{equation}\label{3.46}
%\lim _{p \rightarrow \infty} P(\gamma_0 n^{2/3} (\mu_1-E_+)\leq s)= F_1(s).
%\end{equation}
%
%
%
%\begin{equation}\label{3.47}
%E_+=\frac{1}{\xi_+}(1+ \frac{p}{n} \int \frac{t\xi_+}{1-t\xi_+} d\hat{\rho}(t))
%\end{equation}
%
%and
%
%\begin{equation}\label{3.48}
%\frac{1}{\gamma_0^3}=\frac{1}{\xi_+^3}(1+ \frac{p}{n} \int (\frac{t\xi_+}{1-t\xi_+})^3 d\hat{\rho}(t)).
%\end{equation}
%(\ref{3.53})-(\ref{3.55}) can be compared with (\ref{3.44}), (\ref{3.47}) and (\ref{3.48}).
%
%We can also compare Condition \ref{cond1} and (\ref{3.25}) with (\ref{3.41}), (\ref{3.42}) and (\ref{3.45}).
%
%So by Lemma \ref{0313-3} we can get
 It follows from Theorem 1.3 of \cite{BPZ2014a} and Theorem 2.18 of \cite{KY14} that
\begin{equation}\label{3.56}
\lim_{p \rightarrow \infty} P\Big(\big(\hat{\sigma}_{p} n^{2/3} (\lambda_{1}-\hat{\mu}_{p}) \leq s\big) \bigcap S_{\zeta}|\bbA_p\Big)=F_1(s),
\end{equation}
which implies that
\begin{equation}\label{3.56*}
\lim_{p \rightarrow \infty} P\Big(\big(\hat{\sigma}_{p} n^{2/3} (\lambda_{1}-\hat{\mu}_{p}) \leq s\big) \bigcap S_{\zeta}\Big)=F_1(s).
\end{equation}
Moreover by Lemmas \ref{0312-1} and \ref{0313-2} we obtain on the event $S_{\zeta}$
\begin{equation}\label{3.57}
\lim_{p \rightarrow \infty} \frac{\sigma_{p}}{\hat{\sigma}_{p}}=1
\end{equation}
and
\begin{equation}\label{3.58}
\lim_{p \rightarrow \infty} \sigma_{p}n^{2/3}(\hat{\mu}_{p}-\mu_{p})=0.
\end{equation}
(\ref{3.51}) then follows from (\ref{3.52}), (\ref{3.56})-(\ref{3.58}) and Slutsky's theorem. The proof is complete. %that
%\begin{equation}\label{3.59}
%\lim_{p \rightarrow \infty} P(\sigma_{p} n^{2/3} (\lambda_{1}-\mu_{p}) \leq s,S_{\zeta} )=F_1(s),
%\end{equation}
%
%With (\ref{3.59}) we can prove (\ref{3.51}) and (\ref{2.17}).
\end{proof}

\section{Proof of (\ref{a40})}

\begin{proof}
This section is to verify (\ref{a40}) and give an exact expressions of $c_p, \mu_{p}$ and $\sigma_{p}$ in (\ref{2.7})-(\ref{2.9}) at the mean time. %the equivalence between (\ref{5.2}) and . % Moreover for the purpose of verifying the formulas of asymptotic mean and variance in Theorem \ref{1222-2},
We first introduce the following notation. Let $\breve{m}=\max\{m,p\}$, $\breve{n}=\min\{n,m+n-p\}$ and $\breve{p}=\min\{m,p\}$.
Choose $0< \alpha_{p} < \frac{\pi}{2} $ and $0< \beta_{p} < \frac{\pi}{2}$ to satisfy
\begin{equation}\label{2.18s}
\sin^2(\alpha_{p})=\frac{\breve{p}}{\breve{m}+\breve{n}},\quad \sin^2(\beta_{p})=\frac{\breve{n}}{\breve{m}+\breve{n}}.
\end{equation}
Define
\begin{equation}\label{2.19s}
\mu_{p}=\frac{\breve{m}}{\breve{n}} \tan^2(\alpha_{p}+\beta_{p})
\end{equation}
and
\begin{equation}\label{2.20s}
\frac{1}{\sigma_{p}^3}=\mu_{p}^3\frac{16\breve{n}^2}{(\breve{m}+\breve{n})^2}\frac{1}{\sin(2\beta_{p})\sin(2\alpha_{p})\sin^2(2\beta_{p}+2\alpha_{p})}.
\end{equation}

We below first verify the equivalence between (\ref{2.8})-(\ref{2.9}) and (\ref{2.19s})-(\ref{2.20s}). For definiteness, consider $p<m$ in what follows and the case $p>m$ can be discussed similarly.
Denote by $s(z)$ the Stieltjes transform of the MP law $\rho_{p}(x)$
$$
s(z)=\int\frac{\rho_{p}(x)}{x-z}dx, \quad Im(z)>0
$$
and set
\begin{equation}\label{4.1}
g_{p}(x)=\frac{1-\frac{p}{m}-x-\sqrt{(x-1-\frac{p}{m})^2-4\frac{p}{m}}}{2\frac{p}{m}x},
\end{equation}
which is the function obtained from $s(z)$ by replacing $z$ with $x$ (one may see (3.3.2) of \cite{BS06}).
Evidently, the derivative of $s(z)$ is
$$ s'(z)=\int\frac{\rho_{p}(x)}{(x-z)^2}dx.$$
Note that $c_p$ is outside the support of the MP law (see Lemma \ref{0312-1}). In view of the above and (\ref{2.7}) we obtain
\begin{equation}\label{4.2}
c^2_{p}g_{p}'(c_{p})=\frac{n}{p},
\end{equation}
which further implies that
\begin{equation}\label{4.3}
\sqrt{(c_{p}-1-\frac{p}{m})^2-4\frac{p}{m}}=\frac{(1-\frac{p}{m})^2-(1+\frac{p}{m})c_{p}}{\frac{2n}{m}+1-\frac{p}{m}}.
\end{equation}

When $n \neq p$, solving (\ref{4.3}) and disregarding one of the solutions bigger than $a_p$ we have
$$c_{p}= \frac{(\frac{m+p}{m})(\frac{p}{m}+\frac{p}{n}-\frac{p^2}{mn})-\sqrt{(1+\frac{p}{m})^2(\frac{p}{m}+\frac{p}{n}-\frac{p^2}{mn})^2+(1-\frac{p}{m})^2
(\frac{p}{m}+\frac{p}{n}-\frac{p^2}{mn})(\frac{p}{n}-1)(\frac{p}{m}+\frac{p}{n})}}{(1-\frac{p}{n})(\frac{p}{m}+\frac{p}{n})}
$$
\begin{equation}
=\frac{n(m+p)(m+n-p)-(m+2n-p) \sqrt{mnp(m+n-p)}}{m(n-p)(m+n)}.\label{5.3}
\end{equation}
This, together with (\ref{2.8}), yields
$$\mu_{p}=\frac{1}{c_{p}}+\frac{p}{n} g_{p}(c_{p}) =\frac{1}{c_{p}}\frac{2(m+n-p)}{m+2n-p}-\frac{(n-p)m}{(m+2n-p)n}
$$$$=\frac{m}{n} \frac{(n-p)(n(m+n-p)+\sqrt{mnp(m+n-p)})}{n(m+p)(m+n-p)-(m+2n-p)\sqrt{mnp(m+n-p)}}
$$$$=\frac{m}{n} \frac{(\sqrt{(m+n-p)n}+\sqrt{mp})^2}{(\sqrt{m(m+n-p)}-\sqrt{np})^2}.
$$
By (\ref{2.18s}) one may obtain
$$\frac{\sqrt{(m+n-p)n}+\sqrt{mp}}{\sqrt{m(m+n-p)}-\sqrt{np}}=\frac{\frac{\sqrt{(m+n-p)n}+\sqrt{mp}}{m+n}}{\frac{\sqrt{m(m+n-p)}-\sqrt{np}}{m+n}}
$$$$=\frac{\cos \alpha_{p} \sin \beta_{p}+\sin \alpha_{p} \cos \beta_{p}}{\cos \alpha_{p} \cos \beta_{p}-\sin \alpha_{p} \sin \beta_{p}}
=\tan (\alpha_{p}+\beta_{p}).
 $$
It follows that \begin{equation}\label{4.4}
\mu_{p}=\frac{m}{n} \tan^2(\alpha_{p}+\beta_{p}),
\end{equation}
which is (\ref{2.19s}).

Using (\ref{2.18s}), (\ref{2.19s}) and the second derivative of $g''_{p}(x)$ at $c_p$ (\ref{2.9}) can be rewritten as
\begin{eqnarray*}
&&\frac{1}{\sigma_{p}^3}=\frac{1}{c_{p}^3}+\frac{p}{2n} g''_{p}(c_{p}) \\
&&=\frac{1}{c_{p}^3}+ \frac{p}{2n}(\frac{1-\frac{p}{m}}{\frac{p}{m}c_{p}^3}-\frac{\sqrt{(c_{p}-1-\frac{p}{m})^2-4\frac{p}{m}}}{\frac{p}{m}c_{p}^3}-\frac{1+\frac{p}{m}}{\frac{p}{m}c_{p}^2\sqrt{(c_{p}-1-\frac{p}{m})^2-4\frac{p}{m}}}+\\
&&\frac{1}{2\frac{p}{m}c_{p}\sqrt{(c_{p}-1-\frac{p}{m})^2-4\frac{p}{m}}}+\frac{(c_{p}-\frac{p}{m}-1)^2}{2\frac{p}{m}c_{p}((c_{p}-1-\frac{p}{m})^2-4\frac{p}{m})^{3/2}})\\
&&=\cos^2(\beta_{p})\cot^3(\beta_{p})\csc(\alpha_{p})\sec(\alpha_{p})\sec^4(\beta_{p}+\alpha_{p})\tan^4(\beta_{p}+\alpha_{p})\\
&&=16\cos^4(\beta_{p})\cot^2(\beta_{p})\csc(2\beta_{p})\csc(2\alpha_{p})\csc^2(2\beta_{p}+2\alpha_{p})\tan^6(\beta_{p}+\alpha_{p})\\
&&=16\frac{m^2}{(m+n)^2} \frac{m}{n} \frac{1}{\sin(2\beta_{p})\sin(2\alpha_{p})\sin^2(2\beta_{p}+2\alpha_{p})}\tan^6(\beta_{p}+\alpha_{p})\\
&&=\mu_{p}^3\frac{16n^2}{(m+n)^2}\frac{1}{\sin(2\beta_{p})\sin(2\alpha_{p})\sin^2(2\beta_{p}+2\alpha_{p})},
\end{eqnarray*}
which is (\ref{5.3}).

%\begin{equation}\label{4.5}
%\frac{1}{\sigma_{p}^3}=\mu_{p}^3\frac{16n^2}{(m+n)^2}\frac{1}{\sin(2\beta_{p})\sin(2\alpha_{p})\sin^2(2\beta_{p}+2\alpha_{p})}.
%\end{equation}

When $n=p$, solving (\ref{4.3}) yields
\begin{eqnarray}\label{4.6}
&&c_{p}= \frac{(1-\frac{p}{m})^2}{2(1+\frac{p}{m})}=\frac{(m-p)^2}{2(m+p)m}.
\end{eqnarray}
From (\ref{2.18s}) one may conclude that
$
\alpha_{p}=\beta_p.
$
Since
\begin{eqnarray*}
&&\mu_{p}=\frac{1}{c_{p}}+\frac{p}{n} g_{p}(c_{p})=\frac{1}{c_{p}}\frac{2(m+n-p)}{m+2n-p} =\frac{4m^2}{(m-p)^2}
\end{eqnarray*}
and
\begin{eqnarray*}
\frac{2\sqrt{mp}}{(m-p)}=\frac{2\frac{\sqrt{mp}}{m+p}}{\frac{m-p}{m+p}}= \frac{\sin(2\alpha_{p})}{\cos(2\alpha_{p})}=\tan(\alpha_{p}+\beta_{p})
\end{eqnarray*}
we have
\begin{eqnarray}\label{4.7}
&&\mu_{p}=\frac{m}{n} \tan^2(\alpha_{p}+\beta_{p}).
\end{eqnarray}
It's similar to prove
\begin{equation}\label{4.8}
\frac{1}{\sigma_{p}^3}=\mu_{p}^3\frac{16n^2}{(m+n)^2}\frac{1}{\sin(2\beta_{p})\sin(2\alpha_{p})\sin^2(2\beta_{p}+2\alpha_{p})}.
\end{equation}
The above implies the equivalence between (\ref{2.8})-(\ref{2.9}) and (\ref{2.19s})-(\ref{2.20s}).

It is straightforward to verify that $|\frac{m}{n}\mu_{J,p}-\mu_{p}|=O(p^{-1})$ and $\lim_{p \rightarrow \infty }\sigma_{p} \frac{m}{n^{1/3}}\sigma_{J,p}=1$ according to (\ref{2.18s})-(\ref{2.20s}) and (\ref{5.2}).
\end{proof}

%\iffalse{

\section{Proof of Part (ii) of Theorem \ref{t1}: Standard Gaussian Distribution}

This section is to consider the case when $\{X_{ij}\}$ follow normal distribution with mean zero and variance one. We below first introduce more notation. Let $\bbA=(A_{ij})$ be a matrix. We define the following norms
 $$\|\bbA\|= \max_{|\bbx|=1}|\bbA\bbx|,\ \ \|\bbA\|_{\infty}=\max_{i,j}|A_{ij}|, \ \ \|\bbA\|_F=\sqrt{\sum_{ij}|A_{ij}|^2},$$
 where $|\bbx|$ represents the Euclidean norm of a vector $\bbx$. Notice that we have a simple relationship among these norms %ol $\|\cdot\|_{\infty}$ by $\|\cdot\|$ or $\|\cdot\|_F$:
 $$\|\bbA\|_{\infty}\le \|\bbA\|\le \|\bbA\|_F.$$
 %\end{deff}
%\begin{rmk}
%One should also notice that the triangle inequality doesn't hold for the norm $\|\cdot\|_{\infty}$.
We also need the following commonly used definition about stochastic domination to simplify the statements.
\begin{deff} (Stochastic domination) Let
$$\xi=\{\xi^{(n)}(u):n\in \mathbb{N}, u\in U^{(n)}\}, \ \ \zeta=\{\zeta^{(n)}(u):n\in \mathbb{N}, u\in U^{(n)}\}$$
be two families of random variables, where $U^{(n)}$ is a n-dependent parameter set (or independent of n). If for sufficiently small positive $\ep$ and sufficiently large $\sigma$,
$$\sup_{u\in U^{(n)}}\mathbb{P}\left[|\xi^{(n)}(u)|>n^{\ep}|\zeta^{(n)}(u)|\right]\le n^{-\sigma}$$
for large enough $n\geq n(\ep,\sigma)$, then we say that $\zeta$ stochastically dominates $\xi$ uniformly in u. We denote this relationship by $|\xi| \prec \zeta$ and also write it as $\xi =O_{\prec}(\zeta)$. Furthermore
we also write it as $|x| \prec y$ if $x$ and $y$ are both nonrandom and $|x|\leq n^{\ep}|y|$ for sufficiently small positive $\ep$.
%Furthermore, if $\xi$ is a complex(or real) random variable and $|\xi|\prec \zeta$, we  can also write $\xi =O_{\prec}(\zeta)$.
\end{deff}

\begin{proof}
We start the proof by reminding readers that $m<p$ and $m+n>p$. Since $m<p$ the limit of the empirical distribution function of $\frac{1}{p}\bbY^*\bbY$ is the MP law and we denote its density by $\rho_{pm}(x)$.  We define $\gamma_{m,1} \geq \gamma_{m,2} \geq \cdots \geq \gamma_{m,m}$ to satisfy
\begin{equation}\label{6.1}
\int_{\gamma_{m,j}}^{+\infty}\rho_{pm}dx=\frac{j}{m},
\end{equation}
with $\gamma_{m,0} =(1+\sqrt{\frac{m}{p}})^2$,$\gamma_{m,m} =(1-\sqrt{\frac{m}{p}})^2$.
Correspondingly denote the eigenvalues of $\frac{1}{p}\bbY^*\bbY$ by $\hat{\gamma}_{m,1} \geq \hat{\gamma}_{m,2} \geq \cdots \geq \hat{\gamma}_{m,m}$.
Here we would remind the readers that $\rho_{pm}(x),\gamma_{m,j}, \hat{\gamma}_{m,1}$ are similar to those in (\ref{2.5}), below (\ref{2.5}) and above (\ref{2.13}) except that we are interchanging the role of p and m because we are considering $\frac{1}{p}\bbY^*\bbY$ rather than $\frac{1}{m}\bbY\bbY^*$.
Moreover as in (\ref{3.31}) and (\ref{3.30}) for any sufficiently small $\zeta>0$ and big $D>0$ there exists an event $S_{\zeta}$ (here with a bit abuse of notion $S_{\zeta}$) such that
\begin{equation}\label{6.2}
S_{\zeta}=\{\forall j,1 \leq j \leq m, |\hat{\gamma}_{m,j}-\gamma_{m,j}| \leq  p^{\zeta-2/3} \tilde{j}^{-1/3} \}
\end{equation}
and
\begin{equation}\label{6.3}
P(S_{\zeta}^c)\leq p^{-D} .
\end{equation}

Note that $\frac{1}{p}\bbY\bbY^*$ and $\frac{1}{p}\bbY^*\bbY$ have the same nonzero eigenvalues. To simplify notation let $m_p=m+n-p$. Write
\begin{eqnarray}\label{6.4}
\frac{1}{p}\bbY\bbY^*= \bbU^* \left( \begin{array}{cc}
\bbD & 0 \\
0 & 0
\end{array}
\right) \bbU,
\end{eqnarray}
with
$
\bbD=diag \{\hat{\gamma}_{m,1}, \hat{\gamma}_{m,2}, \cdots , \hat{\gamma}_{m,m}\}
$
and $\bbU$ is an orthogonal matrix. Then $\det(\lambda \frac{\bbY\bbY^*}{p}-\frac{\bbX\bbX^*}{m_p})=0$ is equivalent to
\begin{eqnarray*}
\det \left(\lambda \left( \begin{array}{cc}
\bbD & 0 \\
0 & 0
\end{array}
\right)-\frac{1}{m_p}U^*\bbX\bbX^*U \right)=0.
\end{eqnarray*}
Moreover, since $\{\bbX_{ij}\}$ are independent standard normal random variables and $\bbU$ is an orthogonal matrix we have $U\bbX\overset{d}{=}\bbX$ so that it suffices to consider the following determinant \begin{eqnarray}\label{6.5}
\det\left(\lambda  \left( \begin{array}{cc}
\bbD & 0 \\
0 & 0
\end{array}
\right)-\frac{1}{m_p}\bbX\bbX^*\right)=0.
\end{eqnarray}
Here $\overset{d}{=}$ means having the identical distribution.

Now rewrite $\bbX$ as
%
%\begin{eqnarray}\label{6.6}
$
\bbX=\left( \begin{array}{cc}
\bbX_1 \\
\bbX_2
\end{array}
\right),
$
%\end{eqnarray}
where $\bbX_1$ is a $m \times n$ matrix and $\bbX_2$ is a $(p-m) \times n$ matrix. It follows that
\begin{eqnarray}\label{6.7}
\bbX\bbX^*=\left( \begin{array}{cc}
\bbX_1 \bbX_1^* & \bbX_1 \bbX_2^*\\
\bbX_2 \bbX_1^* & \bbX_2 \bbX_2^*
\end{array}\right)
\overset{\Delta}{=}\left( \begin{array}{cc}
\bbX_{11}  & \bbX_{12}\\
\bbX_{21}  & \bbX_{22}
\end{array}
\right).
\end{eqnarray}
(\ref{6.5}) can be rewritten as
\begin{eqnarray*}
\det \left( \begin{array}{cc}
\frac{1}{m_p}\bbX_{11}-\lambda \bbD & \frac{1}{m_p}\bbX_{12} \\
\frac{1}{m_p}\bbX_{21} & \frac{1}{m_p}\bbX_{22}
\end{array}
\right)=0.
\end{eqnarray*}
Since $m+n>p$, $\bbX_{22}$ is invertible. (\ref{6.5}) is further equivalent to
\begin{eqnarray}\label{0318.1}
\det (\frac{1}{m_p}\bbX_{11}-\lambda \bbD-\frac{1}{m_p}\bbX_{12} \bbX^{-1}_{22} \bbX_{21}) =0.
\end{eqnarray}
Moreover,
\begin{eqnarray*}
\bbX_{11}-\bbX_{12} \bbX_{22}^{-1} \bbX_{21} =\bbX_1\bbX_1^*-\bbX_1 \bbX_2^*(\bbX_2 \bbX_2^*)^{-1}\bbX_2 \bbX_1^*=\bbX_1 (I_n-\bbX_2^*(\bbX_2 \bbX_2^*)^{-1}\bbX_2) \bbX_1^*.
\end{eqnarray*}
Since $rank(I_n-\bbX_2^*(\bbX_2 \bbX_2^*)^{-1}\bbX_2)=m+n-p=m_p$ we can write
\begin{eqnarray*}
\bbI_n-\bbX_2^*(\bbX_2 \bbX_2^*)^{-1}\bbX_2= \bbV  \left( \begin{array}{cc}
\bbI_{m_p} & 0 \\
0 & 0
\end{array} \right) \bbV^*.
\end{eqnarray*}
where $\bbV$ is an orthogonal matrix. %Note that $X_{ij}$  are independent standard normal random variables so that
In view of the above we can construct a $m \times m_p$ matrix $\bbZ=(Z_{ij})_{m,m_p}$ consisting of independent standard normal random variables so that
\begin{eqnarray}\label{1129.3}
\bbX_{11}-\bbX_{12} \bbX_{22}^{-1} \bbX_{21} \overset{d}{=}\bbZ\bbZ^*.
\end{eqnarray}
It follows that (\ref{0318.1}) and hence (\ref{6.5}) are equivalent to
\begin{equation}\label{b1}
\det (\frac{1}{m_p}\bbZ\bbZ^*-\lambda \bbD) =0.
\end{equation}

It then suffices to consider the largest eigenvalue of $\frac{1}{m_p}\bbD^{-1}\bbZ\bbZ^*$. Denote by $\lambda_{1}$ the largest eigenvalue of $\frac{1}{m_p}D^{-1}ZZ^*$.
%
%Now we define a matrix $F=\frac{1}{m+n-p}D^{-1}ZZ^*$ and we can find that the largest eigenvalue of $F$ is $\lambda_{p,1}$.
%By (\ref{6a}), we can get
%
%\begin{eqnarray*}
%0<  \lim_{p \rightarrow \infty} \frac{m}{m+n-p} < \infty,0< \lim_{p \rightarrow \infty} \frac{m}{p} < 1.
%\end{eqnarray*}
As in (\ref{2.13}) and (\ref{2.14}) define $\hat{c}_{m} \in [0,\hat{\gamma}_{m,m})$ to satisfy
\begin{equation}\label{a22}
\frac{1}{m} \sum_{j=1}^m (\frac{\hat{c}_{m}}{\hat{\gamma}_{m,j}-\hat{c}_{m}})^2=\frac{m_p}{m}
\end{equation}
and  $\hat{\mu}_{p}$ and $\hat{\sigma}_{p}$ by
\begin{equation*}
\hat{\mu}_{m}=\frac{1}{\hat{c}_{m}}(1+\frac{1}{m_p}\sum_{j=1}^m (\frac{\hat{c}_{m}}{\hat{\gamma}_{m,j}-\hat{c}_{m}})),\quad
%\end{equation*}
%and
%\begin{equation*}
\frac{1}{\hat{\sigma}_{m}^3}=\frac{1}{\hat{c}_{m}^3}(1+\frac{1}{m_p}\sum_{j=1}^m (\frac{\hat{c}_{m}}{\hat{\gamma}_{m,j}-\hat{c}_{m}})^3).
\end{equation*}
From Lemma \ref{0313-2} we have on the event $S_{\zeta}$
\begin{equation}\label{a23}
\limsup_{p}  \frac{\hat{c}_{m}}{\hat{\gamma}_{m,m}}<1,
\end{equation}
which implies condition (\ref{3.45}). % the s of lemma \ref{0313-3} (it's similar to section 4.2) and get the result
It follows from Theorem 1.3 of \cite{BPZ2014a} and Theorem 2.18 of \cite{KY14} that
\begin{equation}\label{0529.1}
\lim_{p \rightarrow \infty} P(\hat{\sigma}_{m} (m+n-p)^{2/3} (\lambda_{1}-\hat{\mu}_{m}) \leq s)=F_1(s).
\end{equation}
As in the proof of Theorem \ref{t1}, by Lemmas \ref{0312-1} and \ref{0313-2}  one may further conclude that
\begin{equation}\label{0529.2}
\lim_{p \rightarrow \infty} P(\sigma_{p} (m+n-p)^{2/3} (\lambda_{1}-\mu_{p}) \leq s)=F_1(s).
\end{equation}

%With (\ref{6c})-(\ref{6e}), lemma \ref{0312-1}, lemma \ref{0313-2} and Remark \ref{r1}, we can get
%\begin{equation*}
%\lim_{p \rightarrow \infty} P(\sigma_{p} (m+n-p)^{2/3} (\lambda_{1}-\mu_{p}) \leq s)=F_1(s)
%\end{equation*}
%
%for any given $\bbY \in S_{\zeta}$.
%
%That means
%
%\begin{equation*}
%\lim_{p \rightarrow \infty} P(\sigma_{p} (m+n-p)^{2/3} (\lambda_{1}-\mu_{p}) \leq s|S_{\zeta})=F_1(s).
%\end{equation*}
%
%
%In fact,
%\begin{eqnarray*}
%&&P(\sigma_{p} (m+n-p)^{2/3} (\lambda_{1}-\mu_{p}) \leq s)=P(\sigma_{p} (m+n-p)^{2/3} (\lambda_{1}-\mu_{p}) \leq s|S_{\zeta})P(S_{\zeta})\\
%&&+P(\sigma_{p} (m+n-p)^{2/3} (\lambda_{1}-\mu_{p}) \leq s|S_{\zeta}^c)P(S_{\zeta}^c).
%\end{eqnarray*}
%
%By (\ref{6.3}) we then conclude that
%

\end{proof}

\section{Proof of Part (ii) of Theorem \ref{t1}: General distributions}

The aim of this section is to relax the gaussian assumption on $\bbX$.  We below assume that $\bbX$ and $\bbY$ are real matrices. The complex case can be handled similarly and hence we omit it here. %is easier. Actually, by the simple fact that $\frac{\partial \overline{x}}{\partial x}=0$, we can find many expressions become simpler. One can repeat the following steps for complex case with some minor modifications.
In the sequel, we absorb $\frac{1}{\sqrt{m+n-p}}$ and $\frac{1}{\sqrt{p}}$ into $\bbX$ and $\bbY$ respectively( i.e. $Var(X_{ij})=\frac{1}{m+n-p}$, $Var(Y_{st})=\frac{1}{p}$) for convenience. %i.e. we consider the matrix

In terms of the notation in this section ($Var(\bbY_{st})=\frac{1}{p}$), (\ref{6.4}) can be rewritten as
\begin{eqnarray*}
\bbY\bbY^*= \bbU^* \left( \begin{array}{cc}
\bbD & 0 \\
0 & 0
\end{array}
\right) \bbU.
\end{eqnarray*}
Break  $\bbU$ as $\left( \begin{array}{c}
\bbU_1 \\
\bbU_2
\end{array}\right)$ where $\bbU_1$ and $\bbU_2$ are $m\times p$ and $(p-m)\times p$ respectively. By (\ref{6.4})-(\ref{0318.1}) (note that here we can not omit $\bbU$ by $\bbU\bbX\overset{d}{=}\bbX$), the maximum eigenvalue of
$\det(\lambda \bbY\bbY^*-\bbX\bbX^*)$ is equivalent to that of the following matrix
\begin{eqnarray}\label{0313.1}
&&\bbA=\bbD^{-\frac{1}{2}}\bbU_1\bbX(I-\bbX^T\bbU^T_2(\bbU_2\bbX\bbX^TU^T_2)^{-1}U_2\bbX)\bbX^T\bbU^T_1\bbD^{-\frac{1}{2}}\non
&\overset{\Delta}{=}&\bbD^{-\frac{1}{2}}\bbU_1\bbX(I-P_{\bbX^{T}U_2^T})\bbX^T\bbU^T_1\bbD^{-\frac{1}{2}},
\end{eqnarray}
where $\bbP_{\bbX^{T}\bbU_2^T}$ is the projection matrix. It is not necessary to assume that $\bbU_2\bbX\bbX^T\bbU^T_2$ is invertible since $P_{\bbX^{T}U_2^T}$ is unique even if $(\bbU_2\bbX\bbX^T\bbU^T_2)^{-}$ is the generalized inverse matrix of $\bbU_2\bbX\bbX^T\bbU^T_2$. %This fact allows us to fix a special case $(U_2\bbX\bbX^TU^T_2)^{+}$ which is the Moore--Penrose pseudoinverse. However,
Moreover we indeed have the following lemma to control the smallest eigenvalue of $\bbU_2\bbX\bbX^T\bbU^T_2$. %Before presenting this Lemma, we introduce several matrix norms first.
% \begin{deff}(Matrix Norms) Let $A=(A_{ij})$ be a matrix. We define the following norms
% $$\|A\|= \max_{|\bbx|=1}|A\bbx|,\ \ \|A\|_{\infty}=\max_{i,j}|A_{ij}|, \ \ \|A\|_F=\sqrt{\sum_{ij}|A_{ij}|^2},$$
% where $|\bbx|$ represent the Euclidean norm of the vector $\bbx$.
%
% Notice that we have the simple inequality to control $\|.\|_{\infty}$ by $\|.\|$ or $\|.\|_F$:
% $$\|A\|_{\infty}\le \|A\|\le \|A\|_F.$$
% \end{deff}
%\begin{rmk}
%One should notice that the triangle inequality doesn't hold for the norm $\|.\|_{\infty}$.
%\end{rmk}
 \begin{lem}\label{1121-1}
Suppose that $(m+n-p)^{\frac{1}{2}}\bbX$ satisfies Condition \ref{cond1}. Then $\bbU_2\bbX\bbX^T\bbU^T_2$ is invertible and
\begin{equation}\label{a39}\|(\bbU_2\bbX\bbX^T\bbU^T_2)^{-1}\|\le M
\end{equation}for a large constant M with high probability. Moreover,
\begin{equation}\label{b5}\|\bbX\bbX^*\|\leq M
\end{equation} with high probability under conditions in Theorem \ref{t1}.
\end{lem}
\begin{proof}
One may check that the conditions in Theorem 3.12 in \cite{KY14} are satisfied when considering $\bbU_2\bbX\bbX^T\bbU^T_2$. Applying Theorem 3.12 in \cite{KY14} then yields
$$
|\lambda_{\min}\Big(\bbU_2\bbX\bbX^T\bbU^T_2\Big)-(1-\sqrt{\frac{n}{p-m}})^2|\prec n^{-2/3},
$$
where $(1-\sqrt{\frac{n}{p-m}})^2$ can be obtained when considering the special case when the entries of $\bbX$ are Gaussian.
As for (\ref{b5}) see Lemma 3.9 in [7].
\end{proof}
%\begin{rmk}
%Referring to the standard procedure of  \cite{PY11} and \cite{BPWZ2014},
%We believe that the rigidity and universality hold for the matrix $U_2\bbX\bbX^TU^T_2$. However we will not pursue them in this paper and a weaker result in Lemma \ref{1121-1} is enough.
%\end{rmk}
%Lemma \ref{1121-1} is to be shown in the appendix. We also need to control the largest eigenvalue of $\bbX\bbX^*$ as follows

%\begin{eqnarray}\label{0313.1}
%D^{-\frac{1}{2}}U_1\bbX(I-P_{\bbX^{T}U_2^T})\bbX^TU^T_1D^{-\frac{1}{2}},
%\end{eqnarray}
%where $Var(\bbX_{ij})=\frac{1}{m+n-p}$.
%We use the following notation below introduced in \cite{LAY2013} for convenience. It provides a simple way to describe the relationship of two random variables $\xi$ and $\zeta$.
%$\frac{1}{m_p}\bbD^{-1}\bbZ\bbZ^*$
Since the matrix in (\ref{0313.1}) is quite complicated we construct a linearization matrix for it
\begin{eqnarray}\label{1205.1}
\bbH=\bbH(\bbX)= \left(
  \begin{array}{ccc}
    -z\bbI & 0 & \bbD^{-1/2}\bbU_1\bbX\\
    0 & 0 & \bbU_2\bbX\\
  \bbX^T\bbU^T_1\bbD^{-1/2} & \bbX^T \bbU^T_2& -\bbI\\
  \end{array}
\right).
\end{eqnarray}
The connection between $\bbH$ and the matrix in (\ref{0313.1}) is that the upper left block of the $3\times 3$ block matrix $\bbH^{-1}$ is the Stieltjes transform of (\ref{0313.1}) by simple calculations,. We next
give the limit of the Stieltjes transform of (\ref{0313.1}) and need the following well-known result (see \cite{BS06}).
%\begin{lem}
There exists a unique solution $m(z):\mathcal{C}^+\rightarrow \mathcal{C}$ such that
\begin{equation}\label{b8}\frac{1}{m(z)}=-z+\frac{m}{m+n-p}\int \frac{t}{1+tm(z)}dH_n(t),\end{equation}
where $H_n$ is the empirical distribution function of $\bbD^{-1}$. Moreover, we set
$$\underline{m}(z)=-Tr(z(1+m(z)\bbD^{-1}))^{-1},\quad \rho(x)=\lim_{z\in \mathcal{C}^+\rightarrow x}\Im m(z).$$
%\end{lem}
From the end of the last section we see that under the gaussian case
$(\ref{0313.1}) \xlongequal{d} \bbD^{-1/2}\bbZ\bbZ^*\bbD^{-1/2}.$  Hence it is easy to see that $\hat \mu_m$ defined above (\ref{a23}) is the right most end point of  the support of $\rho(x)$.  %Hence we see that $\underline m(z)$ is the limit of the Stieltjes transform of the ESD of (\ref{0313.1}) and $m(z)$ is the limit of $m_n(z) = \frac{1}{m+n-p}\sum_{i=1}^m (H^{-1})_{ii}+\frac{p-n}{z(m+n-p)}$ under the gaussian case.

For any small positive constant $\tau$ we define the domains
\begin{eqnarray}\label{1121.1}
E(\tau, n)= \{z=E+i\eta \in \mathbb{C}^+: |z|\ge \tau, |E|\le \tau^{-1}, n^{-1+\tau}\le \eta\le \tau^{-1} \},
\end{eqnarray}
\begin{eqnarray}\label{1121.2}
E_+=E_+(\tau,\tau', n)= \{z\in E(\tau, n): E \ge \hat \mu_m-\tau'\},
\end{eqnarray}
where  $\tau'$ is a sufficiently small positive constant.

Set
\begin{equation}\label{a17}\Psi=\Psi(z)= \sqrt{\frac{\Im m(z)}{n\eta}}+\frac{1}{n\eta},\ \bbG(z)= \bbH^{-1},\ \bold\Sigma=\Sigma(z)= z^{-1}(1+m(z)\bbD^{-1})^{-1}.
\end{equation}
To calculate an explicit expression of $\bbG(z)$ we need the following well-known formula:
\begin{eqnarray}\label{0202.1}
\left(
  \begin{array}{cc}
    \bbK & \bbB\\
    \bbC & \bbD\\
  \end{array}
\right)^{-1}
=\left(\begin{array}{cc}
    0 & 0\\
    0 & \bbD^{-1}\\
  \end{array}\right)
  +\left(\begin{array}{cc}
    \bbI\\
    -\bbD^{-1}\bbC  \\
  \end{array}
  \right)
  (\bbK-\bbB\bbD^{-1}\bbC)^{-1}
 \left( \begin{array}{cc}
    \bbI & -\bbB\bbD^{-1}\\
  \end{array}\right).
  %=\left(
%  \begin{array}{cc}
%    (A-BD^{-1}C)^{-1} & -(A-BD^{-1}C)^{-1}BD^{-1}\\
%    -D^{-1}C(A-BD^{-1}C)^{-1} & D^{-1}+D^{-1}C(A-BD^{-1}C)^{-1}BD^{-1}\\
%  \end{array}
%\right)
\end{eqnarray}

%The explicit expression of G(z) will be given below in (\ref{1219.1}). %is
%\begin{eqnarray}\label{0318.2}
%&&G(z)\non
%&&=\left(
%  \begin{array}{ccc}
%    G_m(z) & -G_m(z) D^{-1/2}U_1\bbX\bbX^TU^T_2\Gamma & G_m(z)D^{-1/2}U_1\bbX(I-P_{\bbX^TU^T_2})\\
%   -\Gamma U_2\bbX\bbX^TU^T_1 D^{-1/2} G_m(z)  & \Gamma-\Gamma U_2\bbX\bbX^*U^*_1 D^{-1/2} \Sigma D^{-1/2}U_1\bbX\bbX^TU^T_2\Gamma & \Gamma U_2\bbX \\
%  0 & \bbX^TU^T_2\Gamma & (zm(z)+1)(I-P_{\bbX^TU^T_2})\\
%  \end{array}
%\right). \non
%\end{eqnarray}
We next develop the explicit expression of $\bbG(z)$.  Denote the spectral decomposition of $\bbA_1= \bbD^{-1/2}\bbU_1\bbX(I-\bbP_{\bbX^T\bbU^T_2})\bbX^T\bbU^T_1\bbD^{-1/2}$ by
$$\bbA_1=\bbV\bold\Lambda \bbV^T=\sum_{k=1}^{m}\lambda_k \bbv_k\bbv^T_k,$$
where
$$\lambda_1\ge ... \ge\lambda_{m+n-p}> 0=\lambda_{m+n-p+1}=...=\lambda_m.$$
It follows that
\begin{equation}\label{1119.14}G_{ij}=\sum_{k=1}^{m}\frac{\bbv_k(i)\bbv_k(j)}{\lambda_k-z}, \ \ 1\le i,j\le m,\end{equation}
where $G_{ij}$ denotes the $(i,j)$th entry of the matrix $\bbG(z)$ and $\bbv_k(i)$ means the $i$th component of the vector $\bbv_k$.
%and
%\begin{eqnarray}\label{1219.6}G_{(p+\mu) (p+\nu)}=z\sum_{k=1}^{n}\frac{\bbw_k(\mu)\bbw^*_k(\nu)}{\lambda_k-z}+(P_{\bbX^*U_2^T})_{\mu,\nu}, \ \ 1\le \mu,\nu\le n,\end{eqnarray}
We denote $(G_{ij})_{1\le i,j \le m}$ by $\bbG_m$, which is the green function of (\ref{0313.1}). %, $(G_{\mu\nu})_{p\le \mu,\nu \le p+n}$ by $G_n$.
Moreover, let $$\bbA_2= %\left(\begin{array}{cc} I & A_4\end{array}\right)
\left(
  \begin{array}{ccc}
    \bbI & -\bbD^{-1/2}\bbU_1\bbX\bbX^T\bbU^T_2\Gamma & \bbD^{-1/2}\bbU_1\bbX(\bbI-\bbP_{\bbX^T\bbU^T_2})\\
  \end{array}
\right)^T,$$ and  $$\bbA_3=%\left(
%  \begin{array}{ccc}
%    0 & 0\\
%    0 & A_5\\
%  \end{array}
%\right)=
\left(
  \begin{array}{ccc}
    0 & 0 &0\\
    0 & \bold\Gamma & \Gamma \bbU_2\bbX\\
    0 & \bbX^T\bbU^T_2\Gamma & -I+P_{\bbX^T\bbU^T_2}
  \end{array}
\right),$$
%and
%$$A_4=\left(
%  \begin{array}{ccc}
%     -D^{-1/2}U_1\bbX\bbX^TU^T_2\Gamma & -D^{-1/2}U_1\bbX(I-P_{\bbX^TU^T_2})\\
%  \end{array}
%\right)^T.$$
where $\bold\Gamma=(\bbU_2\bbX\bbX^T\bbU^T_2)^{-1} $. Applying (\ref{0202.1}) twice implies that
\begin{eqnarray}\label{1219.1}
\bbG(z)=\bbA_3+\sum_{k=1}^{m}\frac{\bbA_2\bbv_k\bbv^T_k\bbA^T_2}{\lambda_k-z}=\bbA_3+\bbA_2\bbG_m\bbA^T_2.\
\end{eqnarray}

To control the inverse of a matrix in the projection matrix we introduce the following smooth cutoff function
$$\mathcal{X}(x)=\begin{cases}
   1 &\mbox{if $|x|\le M_1 n^{-2}$}\\
   0 &\mbox{if $|x|\ge 2M_1 n^{-2}$},
\end{cases}$$
whose derivatives satisfy $|\mathcal{X}^{(k)}|\le M n^{2k}$, k=1,2,... and $M_1$ is some positive constant. Let $\tilde \lambda_1 \ge...\ge \tilde \lambda_{p-m}$ be the eigenvalues of $\bbU_2\bbX\bbX^T\bbU^T_2$  and $\underline{s}(z)$ be the Stieltjes transform of its ESD. Since
\begin{eqnarray}\label{1209.1}
\Im (\underline{s}(i n^{-2}))=(p-m)^{-1}\sum_{i=1}^{p-m}\frac{n^{-2}}{\tilde \lambda_i^2+n^{-4}},
\end{eqnarray}
we conclude that
\begin{equation}\label{a27}
\text{if}\ |\Im (\underline{s}(i n^{-2}))|\le M_1 n^{-2},\  \text{then}\ \tilde \lambda_{p-m} \ge \frac{M_2}{n}
 \end{equation}
 for some positive constant $M_2$, which allows us to control the maximum eigenvalue of $(\bbU_2\bbX\bbX^T\bbU^T_2)^{-1}$ outside the event $\{\tilde \lambda_{p-m}\ge c\}$. Moreover, consider the event $\{\tilde \lambda_{p-m}\ge c\}$. By Lemma \ref{1121-1}, choosing a sufficient small constant c, we have
\begin{eqnarray}\label{0318.3}
1-o(n^{-l})=\mathbb{P}(\tilde \lambda_{p-m}\ge c)\le \mathbb{P}(\Im (\underline{s}(i n^{-2}))\le M_1n^{-2}), \ \ \text{for any positive integer } l.
\end{eqnarray}
Therefore, by Lemma \ref{1121-1} we have
\begin{equation}\label{b2}
\mathbb{P}(\mathcal{X}(\Im (\underline{s}(i n^{-2})))\neq 1)\le o(n^{-l}), \ \ \text{for any positive integer } l.
\end{equation}
Similarly, by Lemma \ref{1121-1}, for $\|\bbX\|^2_F$, we have
\begin{equation}\label{b3}\mathbb{P}(\mathcal{X}(n^{-3} \|\bbX\|^2_F)\neq 1)\le o(n^{-l}), \ \ \text{for any positive integer } l.
\end{equation}
Set $\mathcal{T}_n(X)= \mathcal{X}(\Im ( \underline{s}(i n^{-2}))\mathcal{X}(n^{-3} \|\bbX\|^2_F)$, and
\begin{eqnarray}\label{1205.2}
&&\bbF(z)=\non
&&\left(
  \begin{array}{ccc}
    -\bold\Sigma & \bold\Sigma \bbD^{-1/2}\bbU_1\bbX\bbX^T\bbU^T_2\bold\Gamma & 0\\
   \bold\Gamma \bbU_2\bbX\bbX^T\bbU^T_1 \bbD^{-1/2} \Sigma  & \bold\Gamma-\bold\Gamma \bbU_2\bbX\bbX^T\bbU^T_1 \bbD^{-1/2} \bold\Sigma \bbD^{-1/2}\bbU_1\bbX\bbX^T\bbU^T_2\bold\Gamma & \bold\Gamma \bbU_2\bbX \\
  0 & \bbX^T\bbU^T_2\Gamma & (zm(z)+1)(\bbI-\bbP_{\bbX^T\bbU^T_2})\\
  \end{array}
\right). \non
\end{eqnarray}
In fact, $\bbF(z)$ is close to $\bbG(z)$ with high probability. In view of (\ref{b2}) and (\ref{b3}) it is straight forward to see that
\begin{eqnarray}\label{0304.7}
\mathcal{T}_n(X)=1
\end{eqnarray} with high probability and we will use it frequently without mention.

We are now in a position to state our main result about the local law near $\hat \mu_m$, the right end point of the support of the limit of the ESD of $\bbA$ in (\ref{0313.1}).
%We call a subset $S\subseteq E(\tau,n)$ a spectral domain if $\{\omega\in E(\tau,n): \Re \omega=\Re z, \Im \omega\ge\Im  z\}\subseteq S$ for any $z\in S$.
 \begin{thm} (Strong local law)\label{1119-3}
Suppose that $(m+n-p)^{\frac{1}{2}}\bbX$ and $p^{\frac{1}{2}}\bbY$ satisfy the conditions of Theorem \ref{t1}. Then
%If the anisotropic local law holds with the parameters $(\bbX^{Gauss}, D, S)$, then the anisotropic local law holds with the parameters $(\bbX, D, S)$.
\begin{itemize}
\item[(i)]  For any deterministic unit vectors $\bbv$, $\bbw\in \mathbb{R}^{p+n}$ %.(Anisotropic local law)
\begin{eqnarray}\label{0310.1}\langle\bbv,(\bbG(z)-\bbF(z))\bbw\rangle\prec \Psi\end{eqnarray}
uniformly $z\in E_+$ and  %with parameters $(\bbX,E,S)$.
\item[(ii)]%.(Averaged local law)
\begin{eqnarray}\label{0310.2}|\underline{m}_n(z)-\underline{m}(z)|\prec \frac{1}{n\eta}\end{eqnarray}
uniformly in $z\in E_+$, where $\underline{m}_n(z)=\frac{1}{m}\sum_{i=1}^mG_{ii}$. %with parameters $(\bbX,E,S)$.
\end{itemize}
\end{thm}

\subsection{Local law (\ref{0310.1})}

The aim of this subsection is to prove (\ref{0310.1}).  Before proving (\ref{0310.1}) we first collect some frequently used bounds below. Recall the definition of $m(z)$ in (\ref{b8}). For $z\in E(\tau,n)$ one may verify that
\begin{equation}\label{a18}
M_2\leq |m(z)| \leq M_1
\end{equation}
and
\begin{equation}\label{a19}
Im(m(z))\geq M \eta.
\end{equation}
(see Lemma 2.3 in \cite{BPWZ2014b} or Lemma 3.1 and Lemma 3.2 in \cite{s1}). Order the eigenvalues of $\bbD^{-1}$ as $d_1\geq d_2\geq\cdots\geq d_m$. From (\ref{a22}) and (\ref{a23}) we conclude that on the event $S_{\xi}$ defined in (\ref{6.2})
\begin{equation}\label{a24}
\limsup\limits_{p}\hat{c}_md_1<1.
\end{equation}
Here we remind the readers that $d_1$ corresponds to $\frac{1}{\gamma_{m,m}}$ there, validity of (\ref{a24}) does not depend on the Gaussian assumption there and we do not assume the entries of $\bbY$ to be Gaussian in the last section. In addition, with probability one
\begin{equation}\label{a25}
\hat{c}_m=-\lim\limits_{z\in \mathcal{C}^+\rightarrow \hat \mu_m}m(z),
\end{equation}
(one may see below (1.8) in \cite{BPWZ2014b} or \cite{s1}). It follows from (\ref{a24}) and (\ref{a25}) that for $z\in E_{+}$ on the event $S_{\xi}$
\begin{equation}\label{a21}
 |1+dm(z)|\geq \tau_2,\quad d\in[d_m,d_1]
\end{equation}
for some positive constant $\tau_2$ (one may also see (iv) of Lemma 2.3 of \cite{BPWZ2014b}). We then conclude (\ref{a18}) and (\ref{a21}) that on the event $S_{\xi}$
\begin{equation}\label{a26}
 \|\bold\Sigma\|=\|\bold\Sigma(z)\|\leq M,\quad d\in[d_m,d_1],
\end{equation}
where $\bold\Sigma=\bold\Sigma(z)$ is defined in (\ref{a17}). Moreover, for $z\in E_{+}$ it follows from Lemma \ref{1121-1}, (\ref{b5}), and (\ref{a18})-(\ref{a21}) that
\begin{equation}\label{0304.5}
\|\bbF(z) \|\prec 1,\quad \|\bbA_2\|\prec 1, \quad \|\bbA_3\|\prec 1.
\end{equation}
We further introduce more notations with bold lower index
$$\bbG_{\bbv s}= \langle\bbv,\bbG\bbe_s\rangle,\ \  \bbG_{\bbv \bbw}= \langle\bbv,\bbG\bbw\rangle, \ \ \text{and} \ \  \bbG_{s\bbv }= \langle\bbe_s,\bbG\bbv\rangle,$$
where $\bbe_s$ is the unit vector with the s-th coordinate equal to 1. In the sequel, if the lower index of a matrix is bold, then it represents the inner product above and otherwise it means one entry of the corresponding matrix.
%\begin{lem}\label{1119-1}
Fix $\tau>0$. For any $z\in E(\tau,n)$ we claim that
\begin{eqnarray}\label{1119.12}
\|\bbG(z)\mathcal{T}_n(\bbX)\|\le Cn^{10}\eta^{-1}, \ \ \|\partial_z\bbG(z)\mathcal{T}_n(\bbX)\|\le Cn^{10}\eta^{-2},
\end{eqnarray}
\begin{eqnarray}\label{1219.2}
\|\bbG(z)\|\prec \eta^{-1}, \ \ \|\partial_z\bbG(z)\|\prec \eta^{-2},
\end{eqnarray}
%Moreover, suppose that $\bbv$ is a deterministic unit vector, then
\begin{eqnarray}\label{1124.13}
\sum_{i=1}^m|\bbG_{\mathbf{v}i}|^2=\frac{\Im \bbG_{\mathbf{v}\mathbf{v}}}{\eta}, \quad \|\bbF(z)\mathcal{T}_n(\bbX)I(S_{\xi})\|\le Cn^{4}\eta^{-1}
\end{eqnarray}
%,
%\begin{eqnarray}\label{1219.7}
%\sum_{\mu=1}^n|G_{\mathbf{v}\mu}|^2\prec\frac{\Im G_{\mathbf{v}\mathbf{v}}}{\eta}.
%\end{eqnarray}
and
\begin{eqnarray}\label{1230.1}
|G_{\mathbf{v}\mathbf{v}}|^2\prec\frac{\Im G_{\mathbf{v}\mathbf{v}}}{\eta}+1,
\end{eqnarray}
where and in what follows $I(\cdot)$ denotes an indicator function.
%and
%\begin{eqnarray}\label{1219.7}
%|G_{\mathbf{v}\mathbf{v}}|^2\prec\frac{\Im G_{\mathbf{v_1}\mathbf{v_1}}}{\eta}+\frac{\Im G_{\mathbf{v_2}\mathbf{v_2}}}{\eta}+1.
%\end{eqnarray}
%where
%$$\mathbf{v}=(\mathbf{v},0,0)^T.$$
%\end{lem}
%\begin{proof}
Indeed, the estimates (\ref{1119.12}) % except for (\ref{1219.7})
 follow from  (\ref{1219.1}) and the definition of $\mathcal{T}_n(\bbX)$ directly. (\ref{1219.2}) and (\ref{1230.1}) about the partial order follow from Lemma \ref{1121-1}, (\ref{b5}) and (\ref{1219.1}). The first equality in (\ref{1124.13}) is straightforward and the second one is from the definition of $\mathcal{T}_n(\bbX)$ directly. %. In order to prove (\ref{1219.7}), we expand $|G_{\mathbf{v}\mathbf{v}}|^2$ as follows
%\begin{eqnarray}\label{1219.8}
%|G_{\mathbf{v}\mathbf{v}}|^2\prec&&\bbv_1^TG_mG_m^*\bbv_1+\bbv_2^TA_4^TG_mG_m^*A_4\bbv_2+\bbv_1^TG_mA_4A_4^TG_m^*\bbv_1+\bbv_2^TA_4^TG_mA_4A_4^TG_m^*A_4\bbv_2\non
%&&+\bbv_2^TA_5\bbv_2-2\bbv_1^TG_mG_m^*A_4\bbv_2-2\bbv_1^TG_mA_4A_4^TG_m^*A_4\bbv_2-\bbv_2^TA_5A_4^T(G_m+G^*_m)\bbv_1\non
%&&-\bbv_2^TA_5A_4^T(G_m+G^*_m)A_4\bbv_2+1.
%\end{eqnarray}
%Combining with  $\|A_3\|,\|A_4\|\prec 1$ by Lemma \ref{1121-1}, we can get (\ref{1219.7}) immediately.
%\end{proof}

%Similar to (\ref{1129.3}), we can see that when $(m+n-p)\bbX\bbX^T$ is a central Wishart matrix, then
%\begin{eqnarray*}
%(\ref{0313.1})\xlongequal{d} D^{-\frac{1}{2}}ZZ^TD^{-\frac{1}{2}},
%\end{eqnarray*}
%where $(m+n-p)ZZ^T$ is a central Wishart matrix. Then (\ref{0313.1}) is translated to a covariance matrix.
In the Gaussian case Theorem \ref{1119-3} can be obtained from by Theorem 2.10 of [7]. Indeed, from (\ref{1219.1}) one can see a key observation that each block of $\bbG(z)$ can be represented as a linear combination of the blocks of (3.3) in \cite{KY14}
in the Gaussian case. We now demonstrate such an observation by looking at two block matrices of $\bbG(z)$ and other blocks can checked similarly. For example, $\bbG(z)$ has a block matrix $\bbG_m \bbD^{-1/2}\bbU_1\bbX\bbX^T\bbU^T_2\bold\Gamma$. Note that $\bbU_1\bbX\bbX^T\bbU^T_2\bold\Gamma$ is independent of $\bbG_m$ given $\bbU_2\bbX$ due to $(\bbI-\bbP_{\bbX^T\bbU^T_2})\bbX^T\bbU^T_2=0$ while from the end of the last section we see that
\begin{equation}\label{b4}
\bbA \xlongequal{d} \bbD^{-1/2}\bbZ\bbZ^*\bbD^{-1/2}
\end{equation}
given $\bbU_2\bbX$ under the gaussian case (see (\ref{0313.1}) for the definition of $\bbA$). It follows that this block can be regarded as the product of random $\bbG_m$ and a non-random matrix given $\bbU_2\bbX$. So the local law holds for this block from Theorem 2.10 of [7] by absorbing the nonrandom matrix into the fixed vector $\bbv$ or $\bbw$ (note that (\ref{a21}) is required in the conditions of Theorem 2.10 of [7]). A second block matrix of $\bbG_m$ is $\bbG_m \bbT$ with $\bbT=\bbD^{-1/2}\bbU_1\bbX(\bbI-\bbP_{\bbX^T\bbU^T_2})$. From the end of the last section and (\ref{b4}) we see that $\bbG_m =(\bbT\bbT^*-z\bbI)^{-1}$ due to $(\bbI-\bbP_{\bbX^T\bbU^T_2})$ is a projection matrix so that this block is just one of the block in (3.3) in \cite{KY14}. %By Theorem 2.10 of \cite{KY14} we conclude that Theorem \ref{1119-3} holds under the Gaussian case.%it follows the anisotropic law and the averaged local law, i.e.
%\begin{eqnarray}\label{1124.12}
%\langle \bbv,(G(z=)-\Pi(z))\bbw\rangle\mathcal{T}_n(\bbX)\prec \Psi,
%\end{eqnarray}
%and
%\begin{eqnarray}\label{1119.3}
%|m_n(z)-m(z)|\prec \frac{1}{n\eta}
%\end{eqnarray}
%hold uniformly in $z\in S$ and any deterministic unit vectors $\bbv$, $\bbw\in \mathbb{R}^{n}$.

\subsubsection{Proving (\ref{0310.1}) for general distributions}

We next prove (\ref{0310.1}) for general distributions by fixing $\bbY$ first since $\bbX$ and $\bbY$ are independent (the dominated convergence theorem then ensures (\ref{0310.1})). However to simplify notations we drop the statements about conditioning on $\bbY$ as well as the event $S_{\xi}$. In other words, whenever we come across expectations they should be understood as conditional expectations and involve $I(S_\xi)$. For example, (\ref{a29}) below should be understood as follows
$$\mathbb{E}\Big(|F_{ab}(\bbX,z)|^{2q}I(S_{\xi})\Big|\bbY\Big)\leq (n^{24\delta}\Psi)^{2q}.$$

In order to prove Theorem \ref{1119-3}, it suffices to show that for any deterministic orthogonal matrices $\bbV_1$ and $\bbV_2$, we have
\begin{equation}\label{1116}
\|\bbV_1(\bbG(z)-\bbF(z))\bbV_2^T\|_{\infty}\prec \Psi,
\end{equation}
for all $z\in E_{+}$. %In fact, we will prove it in a discret subset of S instead of S itself.
We define $S$ to be a $\ep$-net of $E(\tau,n)$ with $\ep=n^{-10}$ and the cardinality of $S$, $|S|$, not bigger than $n^{30}$. Note that the function $\bbD^{1/2}(\bbG(z)-\bbF(z))\bbD^{1/2}$ is Lipschitz continuous with
respect to the operation norm in $E_{+}$ and the Lipschitz constant is $Mn^2\|\bbX\bbX^*\|+M n^2\|1/\lambda_{\min}(\bbU_2\bbX)\|$. By (\ref{b5}) it then suffices to focus on $S$ to prove Theorem \ref{1119-3} by Lemma \ref{1121-1}.

Following [7] the main idea of the proof is an induction argument from bigger imaginary parts to smaller imaginary parts.  Set $\delta$ to be a sufficient small positive constant such that $n^{24\delta}\Psi\ll1$. For any given $\eta\ge \frac{1}{n}$, we define a sequence of numbers $\eta_0\le \eta_1\le \eta_2...\le \eta_L$ with
\begin{equation}\label{a20}\eta_l=\eta n^{l\delta}, \ \ (l=0,1,...,L-1), \ \ \eta_L= 1,
\end{equation}
where
$$L\equiv L(\eta)=\max\{l\in \mathbb{N}:\eta n^{l\delta}< n^{-\delta}\}.$$
One can see that $L\leq \delta^{-1}+1$ by the definition.
%In view of  the previous argument, we can claim that
From now on we will work on the net $S$ containing the points %satisfying the following condition:
%\begin{deff}\label{1122.1}
%Let $\hat S$ be an $n^{-10}$ net of S with $|\hat S|\le n^{30}$. $\hat S$ satisfy the condition:
%for any $E+i\eta\in  S$, we have
$E+i\eta_l\in S, \ \ l=0,...,L.$
%\end{deff}
%We will focus on the set $\hat S$ instead of S in the sequel.
Moreover define $ S_k= \{z\in  S:\Im z\ge n^{-\delta k}\}$ and sequence of properties
%($A_k$)  For any z $\in S_k$ we have
\begin{equation}\label{1116.4}
B_k=\{\|\bbV_1(\bbG(z)-\bbF(z))\bbV_2^T\|_{\infty}\prec 1,\quad \text{for any z}\  \in S_k \}
\end{equation}
%($B_k$)  For any z $\in S_k$ we have
\begin{equation}\label{1116.5}
C_k=\{\|\bbV_1(\bbG(z)-\bbF(z))\bbV _2^T\|_{\infty}\prec n^{24\delta} \Psi, \quad \text{for any z}\  \in S_k \}.
\end{equation}

We start the induction by considering property $B_0$. We claim that %, we have the following Lemma immediately.
%\begin{lem}
the property $B_0$ holds.
%\end{lem}
%\begin{proof}
Indeed we conclude from (\ref{1219.1}) and(\ref{0304.5}) that
$$\|\bbV_1(\bbG(z)-\bbF(z))\bbV_2^T\|_{\infty}\prec \|\bbG_m(z)\|+\|\bbF(z)\|+1\prec 1,$$
as claimed.
%which is ensured by % Lemma \ref{1119-1} and
% (\ref{1219.1}) and (\ref{0304.5}).
%\end{proof}
Moreover it's easy to see that property $C_k$ implies property $B_k$ by the choice of $\delta$ such that $n^{24\delta}\Psi\ll 1$.  We next prove that
%The main step of our proof is the following lemma.
%\begin{lem}\label{1119-2}
%Under the conditions in Theorem \ref{1119-3},
property $B_{k-1}$ implies property $C_k$ for any $1\le k \le \delta^{-1}$. If this is true then the induction is complete and (\ref{1116}) holds for all $z\in S$.
%\end{lem}
%It's easy to see that property $(C_k)$ implies property $(A_k)$ by the choice of $\delta$ such that $n^{24\delta}\Psi\ll 1$, together with
%By Lemma \ref{1119-2}, we can prove (\ref{1116}) holds for all $z\in \hat S$ iteratively, so we focus on proving Lemma \ref{1119-2}.

To this end, we calculate the higher moments of the following function
\begin{equation}\label{a5}F_{ab}(\bbX,z)= \left((\bbJ_1\bbG(z)\bbJ_2^T)_{ab}-(\bbJ_1\bbF(z)\bbJ_2^T)_{ab}\right)\mathcal{T}_n(\bbX),\end{equation}
where
$\bbJ_1, \bbJ_2 \in \mathcal{L}= \{1,\bold\Delta, \bbV\}$, $\bold\Delta$ is defined in (\ref{a6}) below and $\bbV$ is any deterministic orthogonal matrix. Lemma \ref{1216-3} below, Markov's inequality and (\ref{0304.7}) then ensure that property $B_{k-1}$ implies property $C_k$. %\ \ \text{and their corresponding first m columns and last p+n-m columns}\}$.

\begin{lem}\label{1216-3}
Let q be a positive constant and $k\le \delta^{-1}$. Suppose that property $B_{k-1}$ in (\ref{1116.4}) holds. %for all $z\in \hat S_{k-1}$.
 Then
\begin{equation}\label{a29}\mathbb{E}\Big(|F_{ab}(\bbX,z)|^{2q}\Big)\leq (n^{24\delta}\Psi)^{2q},\end{equation}
for all $1\le a,b\le n+p$ and $z\in S_k$.
\end{lem}

The proof will be complete if we prove Lemma \ref{1216-3}. Before proceeding, we present a simple but frequently used lemma which can help us transfer the partial order of
two random variables to the partial order of the expectations.
%\begin{equation}\label{b6} If $\zeta \prec \psi$ and $|\zeta| n^M_0$ for some positive constant $M_0$, then
%\mathbb{E}\zeta \prec \psi.
%\end{equation}
\begin{lem}\label{1123-3}
 Let $\zeta$ be a random variable satisfying $\zeta \prec \nu$ where positive $\nu$ may be random or deterministic. Suppose $|\zeta|\leq n^{M_0}$ for some positive constant $M_0$.
  Then \begin{equation}\label{b6} \mathbb{E}\zeta \prec (E\nu+n^{M_0-D}),\end{equation}
  where $D$ is a sufficiently large positive constant.
\end{lem}
\begin{proof}
Since $\zeta \prec \nu$ there exists a sufficiently small positive $\ep$ and sufficiently large $D$ so that
$$P(\zeta \ge n^{\ep}\nu)\le n^{-D}.$$
Define the event $A_{\ep}=\{ \zeta \le n^{\ep}\nu\}$. Write
$$|\mathbb{E}\zeta|=\Big|\mathbb{E}\zeta I(A_{\ep})+\mathbb{E}\zeta I(A^c_{\ep})\Big|\le n^{\ep}\mathbb{E}\nu+n^{M_0}P(A^c_{\ep})\le n^{\ep}\mathbb{E}\nu+n^{M_0-D}.$$
\end{proof}

We now claim that
 \begin{equation}\label{a31}\mathbb{E}\Big(|F_{ab}(\bbX^0,z)|^{2q}\Big)\leq (n^{24\delta}\Psi)^{2q},\end{equation}
 if $\bbX$ in Lemma \ref{1216-3} is replaced by the corresponding Gaussian random matrix $\bbX^0=(X_{i\mu}^0)= \bbX^{Gauss}$ consisting of Gaussian random variables with mean zero and variance one. Indeed, one can see that $|F_{ab}(\bbX^0,z)|^{2q}\prec \Psi^{2q}$ from the paragraph containing (\ref{b4}). To apply (\ref{b6}) to conclude the claim we need $|F_{ab}(\bbX^0,z)|\leq n^{M_0}$, which follows immediately from the first estimate in (\ref{1119.12}) and the second estimate in (\ref{1124.13}).

\subsubsection{Proving Lemma \ref{1216-3} by the interpolation method}

We next finish Lemma \ref{1216-3} for the general distributions by the interpolation method developed by \cite{KY14}. To this end we need to define the interpolation matrix $\bbX^t$ between $\bbX^1=(X_{i\mu}^1)= \bbX$ and $\bbX^0$. %To show ($C_k$) from ($A_{k-1}$), we use the interpolation method developed by \cite{KY14}. We introduce the following definition of the interpolation matrix $\bbX^u$ between $\bbX^1= \bbX$ and $\bbX^0= \bbX^{Gauss}$.
%\begin{deff}\label{1119-4}
For $1\le i\le p $ and $1\le \mu \le n$, denote the distribution function of the random variables $X_{i\mu}^{u}$ by $F_{i\mu}^u$ for $u=0,1$. For $t \in [0,1]$, we define the interpolated distribution function by
\begin{equation}\label{a4}F_{i\mu}^{t}= t F^1_{i\mu}+(1-t)F_{i\mu}^0.
\end{equation}
%where $F_{i\mu}^{u}$ is the cumulative distribution function(CDF) of $\bbX_{i\mu}^{u}$.
%We will work on the matrix consisting of the entries of $\bbX^0, \bbX^{u}$ and $\bbX^1$ together, i.e. the space consisting of $(\bbX^0,\bbX^{u},\bbX^1)$, where $\bbX^0,\bbX^{u}$ and $\bbX^1$ are independent $M\times n$ matrices.
Define the interpolation matrix $\bbX^t=(X_{i\mu}^{t})$ with $F_{i\mu}^{t}$ being the distribution of $X_{i\mu}^{t}$ and $\{X_{i\mu}^{t}\}$ are
independent for $i,\mu$. We furthermore introduce the matrix
\begin{equation}\label{a7}\bbX_{(i\mu)}^{t, \lambda}= \bbX^t+(\lambda-X_{i\mu}^t) \bbe_i\bbe_{\mu}^T,\end{equation}
which differs from $\bbX^t$ at the $(i,\mu)$ position only.
%where  $E_{i\mu}=\bbe_i\bbe_{\mu}^T$
We also define $\bbG^t(z)= \bbG(\bbX^t,z)$ and $\bbG^{t,\lambda}_{(i\mu)}(z)= \bbG(\bbX^{t,\lambda}_{(i\mu)},z)$, the analogues of $\bbG(z)$ defined above (\ref{0202.1}), by replacing the random matrix $\bbX$ in $\bbG(z)$ with $\bbX^t$ and $\bbX^{t,\lambda}_{(i\mu)}$ respectively.

%\end{deff}
We now need the following interpolation formula and one may see Lemma 6.9 of \cite{KY14}.
\begin{lem}\label{1124.14*}
For any function $F:\mathbb{R}^{p\times n}\rightarrow \mathbb{C}$, we have
\begin{eqnarray}\label{1124.14}
\mathbb{E}F(\bbX^1)-\mathbb{E}F(\bbX^0)=\int_0^1dt\sum_{i=1}^p\sum_{\mu=1}^n\left[\mathbb{E}F(\bbX^{t, \bbX_{(i\mu)}^1}_{(i\mu)})-\mathbb{E}F(\bbX^{t, \bbX_{(i\mu)}^0}_{(i\mu)})\right].
\end{eqnarray}
\end{lem}

%Because $\zeta \prec \psi$, we have for any positive number $\ep$ and D,
%$$P(\zeta \ge n^{\ep}\psi)\le n^{-D},$$
%for sufficiently large n. We define the event $A_{\ep}= \{ \zeta \le n^{\ep}\psi\}$, then
%
%Choosing sufficiently large constant D and by the condition $\psi\ge n^{-C}$, we conclude that $\mathbb{E}\zeta\prec \psi$.

%\end{proof}

%In this paper, we always choose $\theta=2$. First, the following Lemma  follows from Lemma \ref{1123-3} immediately:
%\begin{lem}\label{1216-2}
%Lemma \ref{1216-3} holds when \bbX is replaced by $\bbX^0$.
%\end{lem}
To handle the right hand side of (\ref{1124.14}) we establish the following Lemma.
\begin{lem}\label{1119-5}
Fix an positive integer q and $k\le \delta^{-1}$. Suppose that property $B_{k-1}$ holds.
Then there exists some function $g_{ab}(.,z)$ such that for $t\in [0,1],u\in \{0,1\}, z\in S_k$
\begin{equation}\label{1119.13}
\sum_{i=1}^p\sum_{\mu=1}^n\left[\mathbb{E}\Big(|F_{ab}(\bbX^{t, X_{i\mu}^u}_{(i\mu)},z)|^{2q})\Big)-\mathbb{E}|g_{ab}(\bbX^{t, 0}_{(i\mu)},z)|^{2q}\right]=O((n^{24\delta}\Psi)^{2q}+\|\mathbb{E}\bbL(\bbX^{t},z)\|_{\infty}),
\end{equation}
with the matrix $\bbL(\bbX^{t},z)=\Big(|F_{ab}(\bbX^{t},z)|^{2q}\Big)_{1\le a, b\le n+p}$.
%One should notice that here we use $g_{st}(\cdot,z)$ instead of $F_{st}(\cdot,z)$ as a transition.
\end{lem}
Lemma \ref{1119-5} immediately implies that for $z\in S_k$
\begin{equation}\label{b7}
\sum_{i=1}^p\sum_{\mu=1}^n\left[\mathbb{E}\Big(|F_{ab}(\bbX^{t, X_{i\mu}^1}_{(i\mu)},z)|^{2q}\Big)-\mathbb{E}\Big(|F_{ab}(\bbX^{t, X_{i\mu}^0}_{(i\mu)},z)|^{2q}\Big)\right]=O((n^{24\delta}\Psi)^{2q}+\|\mathbb{E}L(\bbX^{t},z)\|_{\infty}).
\end{equation}
To apply the above results we need the following Gronnwall's inequality.
\begin{lem}\label{1216-2}
Suppose that $\beta(t)$ is nonnegative and continuous and $u(t)$ is continuous. If for any $t\in\mathbb{R}$, $\alpha(t)$ is nondecreasing and $u(t)$ satisfies the following equality
$$
u(t)\leq \alpha(t)+\int_0^t\beta(s)u(s)ds,
$$
then
$$
u(t)\leq \alpha(t)\exp\Big(\int_0^t\beta(s)ds\Big).
$$
\end{lem}
To apply Gronnwall's inequality it is observed that
$$\frac{\partial }{\partial t}\Big(\max\limits_{1\le s, t\le n+p}  \mathbb{E}|F_{ab}(\bbX^t,z)|^{2q}\Big)\le \max\limits_{1\le s, t\le n+p} \frac{\partial}{\partial t}  \mathbb{E}|F_{ab}(\bbX^t,z)|^{2q}.$$
From (\ref{1124.14}) and (\ref{b7}) we see that
$$\frac{\partial \mathbb{E}|F_{ab}(\bbX^t,z)|^{2q}}{\partial t}=O((n^{24\delta}\Psi)^{2q}+\|\mathbb{E}\bbL(\bbX^{t},z)\|_{\infty}),$$
if $F$ in (\ref{1124.14}) is taken as $|F_{ab}(\cdot,z)|^{2q}$. Gronnwall's inequality and (\ref{a31}) imply that
$$
\frac{\partial \mathbb{E}|F_{ab}(\bbX^t,z)|^{2q}}{\partial t}\leq M(n^{24\delta}\Psi)^{2q}+M\Big(\max\limits_{1\le s, t\le n+p}  \mathbb{E}|F_{ab}(\bbX^0,z)|^{2q}\Big)\leq M(n^{24\delta}\Psi)^{2q}
$$
This, together with Lemma \ref{1124.14*} and (\ref{a31}), implies that Lemma \ref{1216-3} holds. Similarly for future use we would point out that if $n^{24\delta}\Psi$ in (\ref{1119.13}) is replaced by $n^{\delta}\Psi^2$ and (\ref{a31}) is strengthened to \begin{equation}\label{a32}\mathbb{E}\Big(|F_{ab}(\bbX^0,z)|^{2q}\Big)\leq (n^\delta\Psi^2)^{2q}\end{equation}
  then
\begin{equation}\label{a33}
\mathbb{E}\Big(|F_{ab}(\bbX,z)|^{2q}\Big)\leq (n^\delta\Psi^2)^{2q},
\end{equation}
if the real part of $z$ is outside the support.

What remains is to prove Lemma \ref{1119-5} and we below consider the case $u=1$ only ($u=0$ is similar).  We first develop a crude bound below so that we may use property $B_{k-1}$ in (\ref{1116.4}), which is the assumption of Lemma \ref{1119-5}.
\begin{lem}\label{1116-3}
Suppose that property $B_{k-1}$ holds. Then for any unit vector $\bbv$ and $\bbw$
$$\langle\bbv,(\bbG(z)-\bbF(z))\bbw\rangle=O_{\prec}(n^{2\delta})$$
for all $z\in S_k$.
\end{lem}
\begin{proof}
Recall the definition of $\eta_l$ in (\ref{a20}). Note that $z_l=E+i\eta_l\in S_{k}$ for l=1,2,...,L when $z=E+i\eta\in S_{k-1}$. Hence (\ref{1116.4}) ensures that
$$\Im G_{\bbv\bbv}(E+i\eta_l)\prec |\bbv|^2+\Im \langle\bbv,\Pi(E+i\eta_l)\bbv\rangle\prec |\bbv|^2,$$
where the last $\prec$ follows from (\ref{0304.5}). We conclude the proof by Lemma \ref{1116.1} below.
\end{proof}
\begin{lem}\label{1116.1}
For any $z\in S$ and $\bbx$, $\bby\in \mathbb{R}^{p+n}$, we have
$$\langle\bbx,(\bbG(z)-\bbF(z))\bby\rangle\prec n^{2\delta}\sum_{l=1}^{L(\eta)}(\Im G_{\bbx\bbx}(E+i\eta_l)+\Im G_{\bby\bby}(E+i\eta_l))+|\bbx||\bby|.$$
\end{lem}
\begin{proof} The proof of this lemma follows that of Lemma 6.12 in \cite{KY14} closely.
%Recalling (\ref{1219.1}), we have
%$$\bbG(z)-\bbF(z)=\bbA_3+\sum_{k=1}^{m}\frac{\bbA_2\bbv_k\bbv^*_k\bbA^T_2}{\lambda_k-z}-\bbF(z).$$
It follows from (\ref{1219.1}) and (\ref{0304.5}) that
$$\langle\bbx,(\bbG(z)-\bbF(z))\bby\rangle\prec \sum_{k=1}^{m}\frac{\langle\bbx,A_2\bbv_k\rangle^2}{|\lambda_k-z|}+\sum_{k=1}^m\frac{\langle\bby,A_2\bbv_k\rangle^2}{|\lambda_k-z|}+|\bbx||\bby|.$$
We evaluate the first term below and the second term can be handled similarly. We introduce the indices subsets
$$\mathcal{C}_l= \left\{k: \eta_{l-1}\le |\lambda_k-E|<\eta_l\right\},\ \ (l=0,1,...,L+1),$$
where $\eta_{-1}=0$ and $\eta_{L+1}=\infty$ so that we can rewrite the first term as follows.
$$\sum_{k=1}\frac{\langle\bbx,A_2\bbv_k\rangle^2}{|\lambda_k-z|}=\sum_{l=0}^{L+1}\sum_{k\in \mathcal{C}_l}\frac{\langle\bbx,A_2\bbv_k\rangle^2}{|\lambda_k-z|}.$$
Consider the inner sum for $l\in \left \{1,2,...,L\right\}$,
\begin{eqnarray}\label{1116.2}
&&\sum_{k\in U_l}\frac{\langle\bbx,A_2\bbv_k\rangle^2}{|\lambda_k-z|}\le\sum_{k\in \mathcal{C}_l}\frac{\langle\bbx,A_2\bbv_k\rangle^2\eta_l}{(\lambda_k-E)^2}\le2\sum_{k\in \mathcal{C}_l}\frac{\langle\bbx,A_2\bbv_k\rangle^2\eta_l}{(\lambda_k-E)^2+\eta_{l-1}^2}
\non &&\le2\frac{\eta_l}{\eta_{l-1}}\Im G_{\bbx\bbx}(E+i\eta_{l-1})\le 2n^{\delta}\Im G_{\bbx\bbx}(E+i\eta_{l-1}).
\end{eqnarray}
Combining with the fact that $y\Im G_{\bbx\bbx}(E+iy)$ is nondecreasing function of y, we have
$$\sum_{k\in U_l}\frac{\langle\bbx,A_2\bbv_k\rangle^2}{|\lambda_k-z|}\le 2n^{2\delta}\Im G_{\bbx\bbx}(E+i\eta_{l\vee 1}).$$

Next, we consider the cases l=0 and l=L+1.
$$\sum_{k\in \mathcal{C}_0}\frac{\langle\bbx,A_2\bbv_k\rangle^2}{|\lambda_k-z|}\le2\sum_{k\in \mathcal{C}_0}\frac{\langle\bbx,A_2\bbv_k\rangle^2\eta}{(\lambda_k-E)^2+\eta^2}\le 2\Im G_{\bbx\bbx}(E+i\eta)\le 2n^{\delta}\Im G_{\bbx\bbx}(E+i\eta_1).$$
$$\sum_{k\in \mathcal{C}_{L+1}}\frac{\langle\bbx,A_2\bbv_k\rangle^2}{|\lambda_k-z|}\le2\sum_{k\in \mathcal{C}_{l+1}}\frac{\langle\bbx,A_2\bbv_k\rangle^2|\lambda_k-E|\eta_L}{(\lambda_k-E)^2+\eta^2_L}\prec \sum_{k\in \mathcal{C}_{L+1}}\frac{\langle\bbx,A_2\bbv_k\rangle^2\eta_L}{(\lambda_k-E)^2+\eta^2_L}\le \Im G_{\bbx\bbx}(E+i\eta_L),$$
where we also use (\ref{b5}).
\end{proof}

It is observed that Lemma \ref{1116-3} holds for the interpolation random matrix $\bbX^t$ as well because from (\ref{a4}) one can see that the entries of $\bbX^t$ are independent random variables with mean zero, variance one and finite moment.
Recall the definitions of $\bbJ_i,i=1,2$ in (\ref{a5}). It follows that
%According to the assumption of Lemma \ref{1119-5}, property $B_{k-1}$ holds and hence bec, therefore
\begin{eqnarray}\label{1116.8}
\|\bbJ_1(\bbG^t(z)-\bbF(z))\bbJ_2^T\|_{\infty}\prec n^{2\delta},\quad \text{for z in}\  S_k.
\end{eqnarray}

Below we further generalize it so that (\ref{1116.8}) still holds even if any entry $X_{i\mu}^t$ of $\bbG^t(z)$ is replaced by any other random variable of size not bigger than $n^{-1/2}$. From (\ref{a7}) write
$$
\bbX_{(i\mu)}^{t, \lambda_1}-\bbX_{(i\mu)}^{t, \lambda_2}=(\lambda_1-\lambda_2)\bbe_i\bbe_\mu^T.
$$
This, together with (\ref{1205.1}), yields that
\begin{equation}\label{a12}
\bbH(\bbX_{(i\mu)}^{t, \lambda_1})-\bbH(\bbX_{(i\mu)}^{t, \lambda_2})=\bold\Delta_{(i\mu)}^{\lambda_1-\lambda_2},
\end{equation}
where $\bbH(\bbX_{(i\mu)}^{t, \lambda_1})$ is obtained from $\bbH(\bbX)$ in (\ref{1205.1}) with $\bbX$ replaced by $\bbX_{(i\mu)}^{t, \lambda}$
%\begin{equation}\label{a6}\bold\Delta=\left(
%  \begin{array}{ccc}
%    \bbU^T_1\bbD^{-1/2}& \bbU^T_2 & 0\\
%  \end{array}
%\right)
%\end{equation}
and
\begin{equation}\label{a6}
\bold\Delta_{(i\mu)}^{\lambda}%&=& \lambda\Big[ \left(
%  \begin{array}{c}
%     0\\
%     0\\
%     \bbE_{\mu i}\\
%  \end{array}
%\right)   \left(
%  \begin{array}{ccc}
%    \bbU^T_1\bbD^{-1/2}& \bbU^T_2 & 0\\
%  \end{array}
%\right)+ \left(
%  \begin{array}{ccc}
%    \bbD^{-1/2}\bbU_1 \\
%     \bbU_2\\
%       0\\
%  \end{array}
%\right) \left(
%  \begin{array}{ccc}
%    0 & 0 & \bbE_{i\mu}\\
%  \end{array}
%\right)   \Big]\non
=\lambda \Big( \bbe_{\mu+p}\bbe_{i}^T \bold\Delta+ \bold\Delta^T \bbe_{i} \bbe^T_{\mu+p} \Big),\quad
\bold\Delta=\left(
  \begin{array}{ccc}
    \bbU^T_1\bbD^{-1/2}& \bbU^T_2 & 0\\
  \end{array}
\right),
\end{equation}
where and in the following $\bbe_{\mu+p}$ is always $(n+p)\times 1$ and $\bbe_i$ is $p\times 1$.
%\begin{eqnarray}\label{1116.6}
%\Delta_{(i\mu)}^{\lambda}&=& \lambda\Big[ D^{-1/2}U_1\bbX\bbX^TU^T_2(U_2\bbX\bbX^TU^T_2)^{-1}U_2(\bbX\bbe_{\mu}\bbe'_i+\bbe_i\bbe'_{\mu}\bbX^*)U^*_2(U_2\bbX\bbX^TU^T_2)^{-1}U_2\bbX\bbX^*U^*_1D^{-1/2}\non
%&&+ D^{-1/2}U_1\bbe_i\bbe'_{\mu}(I-\bbX^TU^T_2(U_2\bbX\bbX^TU^T_2)^{-1}U_2\bbX)\bbX^*U^*_1D^{-1/2}\non
%&&+ D^{-1/2}U_1\bbX(I-\bbX^TU^T_2(U_2\bbX\bbX^TU^T_2)^{-1}U_2\bbX)\bbe_{\mu}\bbe'_iU^*_1D^{-1/2}\non
%&&-D^{-1/2}U_1\bbX\bbe_{\mu}\bbe'_iU^*_2(U_2\bbX\bbX^TU^T_2)^{-1}U_2\bbX\bbX^*U^*_1D^{-1/2}\non
%&&-D^{-1/2}U_1\bbX\bbX^TU^T_2(U_2\bbX\bbX^TU^T_2)^{-1}U_2\bbe_i\bbe'_{\mu}\bbX^*U^*_1D^{-1/2}\Big],
%\end{eqnarray}
Applying the formula $\bbA^{-1}-\bbB^{-1}=\bbA^{-1}(\bbB-\bbA)\bbB^{-1}$ repeatedly we further obtain the following resolvent formula for any $H\in \mathbb{N}$ ,
\begin{eqnarray}\label{1123.2}
\bbG_{(i\mu)}^{t, \lambda_1}=\bbG_{(i\mu)}^{t, \lambda_2}+\sum_{h=1}^H (-1)^{h}\bbG_{(i\mu)}^{t, \lambda_2}\Big(\bold\Delta_{(i\mu)}^{\lambda_1-\lambda_2}\bbG_{(i\mu)}^{t, \lambda_2}\Big)^h+(-1)^{H+1}\bbG_{(i\mu)}^{t, \lambda_1}\Big(\bold\Delta_{(i\mu)}^{\lambda_1-\lambda_2}\bbG_{(i\mu)}^{t, \lambda_2}\Big)^{H+1},
\end{eqnarray}
recalling the definition of $\bbG_{(i\mu)}^{t, \lambda_1}$ below (\ref{a7}). Here and below we drop the variable $z$ when there is no confusion but one should keep in mind that $z\in S_k$.

\begin{lem}\label{0304-1}
Suppose that $\lambda$ is a random variable and satisfies $|\lambda|\prec n^{-1/2}$. Then
\begin{eqnarray}\label{1116.7}
\|\bbJ_1(\bbG_{(i\mu)}^{t, \lambda}-\bbF)\bbJ_2^T\|_{\infty}\prec n^{2\delta}.
\end{eqnarray}
\end{lem}
\begin{proof}
Recall (\ref{0304.5})
\begin{equation*}\|\bbF \|\prec 1.\end{equation*}
It is easy to see that
$$\|\bold\Delta \|\le  M,$$
which implies that
\begin{eqnarray}\label{0304}
\|\bold\Delta_{(i\mu)}^{\lambda}\|\le M\lambda.
\end{eqnarray}
We next apply (\ref{1123.2}) with $\lambda_1=\lambda$, $H=11$ and $\lambda_2=X_{i\mu}^t$ so that $\bbG_{(i\mu)}^{t, \lambda_2}=\bbG^t$.
We conclude from (\ref{1116.8}) that
$$\|\bbJ_1 \bbG_{(i\mu)}^{t, \lambda_2}\|+\| \bbG_{(i\mu)}^{t, \lambda_2}\bbJ_2^T\|\prec n^{2\delta}.$$
Note that $|\lambda_1-\lambda_2|\prec n^{-1/2}$. Similar to the first inequality in (\ref{1219.2}), $\bbG_{(i\mu)}^{t, \lambda_1}$ can be bounded by the imaginary part of $z$, i.e. $\bbG_{(i\mu)}^{t, \lambda_1}=O_{\prec}(n)$.   Summarizing the above we conclude Lemma \ref{0304-1}. %Substituting the upper bound into (\ref{1123.2}),  and (\ref{0304}), it is easy to control the right hand side terms of (\ref{1123.2}) except for the last one.  %We can handle the last term by this rough bound.
\end{proof}

In order to simplify the notations, recalling (\ref{a5}) we define
$$f_{(i\mu)}(\lambda)= |F_{ab}(\bbX_{(i\mu)}^{t, \lambda})|^{2q}=\left(F_{st}(\bbX_{(i\mu)}^{t, \lambda})\overline{F_{ab}(\bbX_{(i\mu)}^{t, \lambda})}\right)^{q},$$
where we omit some parameters. %which could be viewed as fixed in the following proof.
By Lemma \ref{0304-1} and (\ref{1123.2}) one can easily get the following Lemma.
\begin{lem}\label{1123-2}
Suppose that $\lambda$ is a random variable and satisfies $|\lambda|\prec n^{-1/2}$. Then for any fixed integer k we have
\begin{eqnarray}\label{1123.3}
|f^{(k)}_{(i\mu)}(\lambda)|\prec n^{2\delta(2q+k)},
\end{eqnarray}
where $f^{(k)}_{(i\mu)}(\lambda)$ denotes the $k$th derivative of $f_{(i\mu)}(\lambda)$ with respect to $\lambda$.
\end{lem}
From Taylor's expansion and (\ref{1123.3}) when $|\lambda|\prec n^{-1/2}$ we have
\begin{eqnarray}\label{1123.4}
f_{(i\mu)}(\lambda)=\sum_{k=0}^{8q}\frac{\lambda^k}{k!}f^{(k)}_{(i\mu)}(0)+O_{\prec}(\Psi^{2q}).
\end{eqnarray}
It follows from Lemma \ref{1123-2} and (\ref{b6}) that
\begin{eqnarray}\label{1123.5}
&&\mathbb{E}|F_{ab}(\bbX_{(i\mu)}^{t,X_{i\mu}^1})|^{2q}-\mathbb{E}|F_{ab}(\bbX_{(i\mu)}^{t,0})|^{2q}=\mathbb{E}f_{(i\mu)}(X_{i\mu}^1)-\mathbb{E}f_{(i\mu)}(0)\\
&&=\frac{1}{2(m+n-p)}\mathbb{E}f^{(2)}_{(i\mu)}(0)+\sum_{k=4}^{8q}\frac{1}{k!}\mathbb{E}f^{(k)}_{(i\mu)}(0)\mathbb{E}(X_{i\mu}^1)^k+O_{\prec}(\Psi^{2q}),\nonumber
\end{eqnarray}
where we use $\mathbb{E}(X_{i\mu}^1)^k=0,k=1,3$.
%here one should notice that 2q is an even number such that f is differentiable.
 To show (\ref{1119.13}), it suffices to prove that
\begin{eqnarray}\label{1123.6}
n^{-k/2}\sum_{i=1}^p\sum_{\mu=1}^n\mathbb{E}f^{(k)}_{(i\mu)}(0)=O((n^{24\delta}\Psi)^{2q}+\|\mathbb{E}|F(\bbX^{t})|^{2q}\|_{\infty}),
\end{eqnarray}
for k=4,...,8q. At this moment we would like to point out that $\mathbb{E}|g_{ab}(\bbX_{(i\mu)}^{t,0})|^{2q}$ in (\ref{1119.13}) equals
$$\mathbb{E}|F_{ab}(\bbX_{(i\mu)}^{t,0})|^{2q}+\frac{1}{2(m+n-p)}\mathbb{E}f^{(2)}_{(i\mu)}(0).$$

We will not prove (\ref{1123.6}) directly. Instead we will prove the following claim in order to obtain a self-consistent estimation of $\bbX^{t}$. We claim that
%\begin{lem}\label{1123-4}
if
\begin{eqnarray}\label{1123.7}
n^{-k/2}\sum_{i=1}^p\sum_{\mu=1}^n\mathbb{E}f^{(k)}_{(i\mu)}(X^{t}_{i\mu})=O((n^{24\delta}\Psi)^{2q}+\|\mathbb{E}|F(\bbX^{t})|^{2q}\|_{\infty}),
\end{eqnarray}
is true for k=4,...,16q, then (\ref{1123.6}) holds for k=4,...,8q.
%\end{lem}
%\begin{proof}
Indeed, in order to apply (\ref{1123.7}) to prove (\ref{1123.6}) we denote $f_{(i\mu)}$ and $X^{t}_{i\mu}$ by f and X respectively for simplicity. Similar to (\ref{1123.5}), by (\ref{1123.3}) we have
\begin{eqnarray}\label{1123.8}
\mathbb{E}f^{(l)}(0)=\mathbb{E}f^{(l)}(X)-\sum_{k=1}^{16q-l}\mathbb{E}f^{(l+k)}(0)\frac{\mathbb{E}X^k}{k!}+O_{\prec}(n^{l/2-1/2-8q+40\delta q}).
\end{eqnarray}
It follows from (\ref{1123.8}) that
\begin{eqnarray*}\label{1123.9}
\mathbb{E}f^{(k)}(0)&=&\mathbb{E}f^{(k)}(X)-\sum_{k_1\ge 1\atop{k+k_1\le 16q}}\mathbb{E}f^{(k+k_1)}(0)\frac{\mathbb{E}X^{k_1}}{k_1!}+O_{\prec}(n^{k/2-1/2-8q+40\delta q})\non
&=&\mathbb{E}f^{(k)}(X)-\sum_{k_1\ge 1\atop{k+k_1\le 16q}}\mathbb{E}f^{(k+k_1)}(X)\frac{\mathbb{E}X^{k_1}}{k_1!}\non
&&+\sum_{k_1,k_2\ge 1\atop{k+k_1+k_2\le 16q}}\mathbb{E}f^{(k+k_1+k_2)}(0)\frac{\mathbb{E}X^{k_1}}{k_1!}\frac{\mathbb{E}X^{k_2}}{k_2!}+O_{\prec}(n^{k/2-1/2-8q+40\delta q})\non
&=...=&\sum_{r=0}^{16q-k}(-1)^r\sum_{k_1,k_2,...,k_r\ge 1\atop{k+\sum k_i\le 16q}}\mathbb{E}f^{(k+\sum k_i)}(X)\prod_i\frac{\mathbb{E}X^{k_i}}{k_i!}+O_{\prec}(n^{k/2-1/2-8q+40\delta q}).
\end{eqnarray*}
This, together with (\ref{a19}) and the definition of $\Psi$ in (\ref{a17}), implies (\ref{1123.6}) immediately, as claimed.
%\end{proof}

It then suffices to prove (\ref{1123.7}). Recall that \begin{equation}\label{a8}
f^{(k)}_{(i\mu)}(X^{t}_{i\mu})=\frac{\partial^k\Big(|F_{ab}(\bbX_{(i\mu)}^{t, X^{t}_{i\mu}})|^{2q}\Big)}{\partial (X_{i\mu}^t)^k},\end{equation}
where $F_{st}(\cdot)$ is given in (\ref{a5}). Since $\bbX^t=\bbX_{(i\mu)}^{t, X^{t}_{i\mu}}$ is the only matrix we focus on we below use $\bbX=(X_{i\mu})$ instead of $\bbX^t=(X_{i\mu}^t)$ to simplify notation because the entries of both of them have bounded higher moments. To prove (\ref{1123.7}) we need to study (\ref{a8}).

\subsubsection{Estimate of higher order derivatives (\ref{a8}) in (\ref{1123.7})}

We first look at the higher order derivatives of $(\bbJ_1\bbF(z)\bbJ_2^T)_{ab}$ with respect to $\bbX_{i\mu}$. Noting that $\bbF(z)$ is a $3\times 3$ block matrix we need to analyze the derivatives of $(\bbJ_1\bbF(z)\bbJ_2^T)_{ab}$ block by block. It turns out that the higher order derivatives of $(\bbJ_1\bbF(z)\bbJ_2^T)_{ab}$ are quite complicated even if we analyze them block by block. Fortunately, as will be seen, the exact expressions of the higher order derivatives of $(\bbJ_1\bbF(z)\bbJ_2^T)_{ab}$ are not important. Moreover we claim an important fact that the higher order derivatives of $\bbF(z)$ with respect to $\bbX_{i\mu}$ can be generated by some sum or products of (part of) common matrices $\bbU_1,\bbU_2,\bold\Sigma,\bbe_i\bbe^T_\mu, \bbe_\mu\bbe^T_i, \bbX,\Gamma(\bbX)$ (we call these common matrices atoms). Indeed, recalling $\Gamma(\bbX)=(U_2\bbX\bbX^TU^T_2)^{-1}$  simple calculations indicate that \begin{equation}\label{a9}
\frac{\partial \bbX\bbX^T}{\partial \bbX_{i\mu}}=\bbX\bbe_\mu\bbe^T_i+\bbe_i\bbe^T_\mu \bbX^T,\ \frac{\partial \Gamma(\bbX)}{\partial \bbX_{i\mu}}=-\Gamma(\bbX)(U_2\bbX\bbe_\mu\bbe^T_iU^T_2+U_2\bbe_i\bbe_\mu \bbX^TU^T_2)\Gamma(\bbX).
\end{equation}
It's easy to see that the first derivative of each block of $\bbF(z)$ with respect to $\bbX_{i\mu}$ can be constructed by sum or products of these atoms. Assuming that the $k$th derivative of each block of $\bbF(z)$ is constructed by these atoms we find that %$\frac{\partial G}{\partial \bbX_{i\mu}^{u}}=-G\frac{\partial &H}{\partial \bbX_{i\mu}^{u}}G$, $\frac{\partial \bbX\bbX^T}{\partial \bbX_{i\mu}^{u}}=\bbX\bbe_\mu\bbe^T_i+\bbe_i\bbe^T_\mu \bbX^T$ and $\frac{\partial \Gamma(\bbX)}{\partial \bbX_{i\mu}^{u}}=-\Gamma(\bbX)(U_2\bbX\bbe_\mu\bbe^T_iU^T_2+U_2\bbe_i\bbe_\mu \bbX^TU^T_2)\Gamma(\bbX)$,
the $(k+1)$th derivative of each block of $\bbF(z)$ is also constructed by these atoms by (\ref{a9}). Based on the above fact we can describe the higher order derivatives of $(\bbJ_1\bbF(z)\bbJ_2^T)_{ab}$ easier.  By dropping $\bbe_i\bbe^T_\mu$ and $\bbe_\mu\bbe^T_i$ from the atoms we define the set
\begin{equation}\label{1202.1}
\mathcal{Q}(k)= \{\text{The matrices constructed from sum or product of (part of)} \ \  \bbU_1, \bbU_2, \bbX, \bold\Sigma,  \  \Gamma(\bbX)\}.
\end{equation}
%For example, for the simplest case k=0, we have
%\begin{eqnarray*}
%\mathcal{Q}(0)=&&\{\Sigma,\Gamma, \Gamma U_2\bbX\bbX^TU^T_1 D^{-1/2} \Sigma D^{-1/2}U_1\bbX\bbX^TU^T_2\Gamma, \Sigma D^{-1/2}U_1\bbX\bbX^TU^T_2\Gamma, \Gamma %U_2\bbX\bbX^*U^*_1 D^{-1/2} \Sigma, \non
%&& \Gamma U_2\bbX,  \bbX^TU_2^T\Gamma, (zm(z)+1)(I-P_{\bbX^TU^T_2}) \}.
%\end{eqnarray*}
Any $k$th order derivative of each block of $\bbF(z)$ with respect to $\bbX_{i\mu}$ belongs to some product(s) between some matrices in $\mathcal{Q}(k)$ and $\bbe_i\bbe^T_\mu$ or $\bbe_\mu\bbe^T_i$.

Lemma \ref{1121-1} and (\ref{b5}) imply that $\|\Gamma(\bbX)\|\le M$ and $\|\bbX\bbX^*\|\le M$ with high probability. Recalling (\ref{a26}), in view of the arguments above we conclude that for any $\bbQ\in \mathcal{Q}(k)$,
\begin{equation}\label{a10}\|\bbQ\|\prec 1\end{equation} and the cardinality of $\mathcal{Q}(k)$ satisfies $|\mathcal{Q}(k)| \le M(k)$, where M(k) is a constant depending on k. Moreover, for the function $\mathcal{T}_n(\bbX)$, if $\mathcal{T}_n(\bbX)$ is differentiated, then by simple and tedious calculations, from the definition of the smooth cutoff function, (\ref{b2}) and (\ref{b3}) we have
\begin{eqnarray}\label{1219.8}
\Big|D_{i \mu}^j\mathcal{T}_n(\bbX)\Big|\prec 0
\end{eqnarray}
and
\begin{eqnarray}\label{1210.1}
\Big|\mathbb{E} D_{i \mu}^j\mathcal{T}_n(\bbX)\Big|\le n^{-l}
\end{eqnarray}
 for any positive integer l and sufficient large n.
%Moreover, it is straight forward to see that
%\begin{eqnarray}\label{1210.1}
%\mathcal{T}_n(\bbX)\prec 1.
%\end{eqnarray}
The above properties about $\mathcal{T}_n(\bbX)$ and the matrices belonging to $\mathcal{Q}(k)$ are enough for our proof below and we don't need to investigate the precise expression. % More precisely, it is enough to take derivative on $\Pi(z)\mathcal{T}_n(\bbX)$.

We next look at the higher order derivatives of $(\bbJ_1\bbG(z)\bbJ_2^T)_{ab}$ with respect to $\bbX_{i\mu}$. To characterize its higher order derivative conveniently we define group $g$ of size $k$ to be the set of paired indices:
$$
 g=\{a_1b_1,a_2b_2,\cdots, a_{k+1}b_{k+1}\},
 $$
 where each of $\{a_j,b_j,j=1,\cdots,k+1\}$ equals one of four letters $s,t,i, (\mu+p)$. Here we would remind readers that the size of group $g$ is defined to be $k$ instead of (k+1) in order to simplify the argument below. Denote the size of the group $g$ by $k=k(g)$ and introduce the set $\mathfrak{G}_k=\{g:\ k(g)=k\}$ consisting of groups of size $k$.
 %the words $w\in \bbW$ with even length constructed by four letters $\{\bbs,\bbt,\bbi,\bm{\mu}\}$. We denote the length of w by 2+2k(w), so k(w)=0,1,2,..., we use the noation
%$$w=\bbs_1\bbt_1\bbs_2\bbt_2...\bbs_{k+1}\bbt_{k+1}$$
%and introduce the subset $\bbW_k= \{w\in\bbW: k(w)=k\}$.
Moreover, we require each group in $\mathfrak{G}_k$ to satisfy three conditions specified below: %as follows when we take derivative on G:
\begin{itemize}
\item[(i)] $a_1=a$ and $b_{k+1}=b$.
\item[(ii)] For $l\in [2,k+1]$ we have $a_l\in \{i, \mu+p\}$ and $b_{l-1}\in \{i,\mu+p \}$.
\item[(iii)] For $k\in [1,k]$ we have $b_{l-1}a_l\in \{i(\mu+p),(\mu+p)i\}$.
\end{itemize}
As will be seen, groups $g$ are connected with the high order derivatives of $(\bbJ_1\bbG(z)\bbJ_2^T)_{ab}$.
%Next, we set $[*]\equiv[*]_{s,t,i,\mu}$ through
%$$[\bbs]= s, \ \ [\bbt]= t, \ \ [\bbi]= i, \ \ [\bm{\mu}]= \mu+p.$$

Moreover write $\bbF(z)=\sum\limits_{j=1}^7\bbF_j(z)$ where each $\bbF_j(z)$ corresponds to a non-zero block of $\bbF(z)$. As before, to characterize the higher order derivative of each block conveniently we define groups $g^{(j)}$ of size $k$ to be the set of paired indices:
$$
g^{(j)}=\{a_{j1}b_{j1},a_{j2}b_{j2},\cdots, a_{j(k+1)}b_{j(k+1)}\},
$$
where each $s_{jm}$ and $t_{jm}$ equals $s,t,i,\mu$.
 Moreover introduce the set $\mathfrak{G}_{jk}=\{g^{(j)}:\ k(g^{(j)})=k\}$ consisting of groups of size $k$. %Similar to the block w, we define 7 words $v_j$ satisfying $k(v_j)=k(w)$, j=1,...7. If we take derivative on $\Pi_j(z)$,
We require each group in $\mathfrak{G}_{jk}$ to satisfy conditions:
\begin{itemize}
\item[(i)] $a_{j1}=a$ and $b_{j (k+1)}=b$.
\item[(ii)] For $l\in [2,k+1]$ we have $a_{jl}\in \{i,\mu\}$ and $b_{j(l-1)}\in \{i,\mu\}$.
\item[(iii)] For $k\in [1,k]$ we have $b_{j(l-1)}a_{jl}\in \{i \mu,\mu i\}$.
\end{itemize}
As will be seen groups $g^{(j)}$ are linked to the high order derivatives of $(\bbJ_1\bbF(z)\bbJ_2^T)_{ab}$.
%Noting that $\Pi(z)$ is a $3\times 3$ block matrix, the lower index of $\Pi(z)$, i.e. $(\Pi(z))_{st}$ decides which block we will work on, so the derivative on $\Pi(z)$ is taken on the block.  Correspondingly,
%The word $w_j$ about $\Pi_j(z)$ represent the positions of the  block decided by the lower index. For example, $\mathcal{R}_{i\mu}$ represent the i-th row and $\mu$-th column's entry of the non-zero block of  $\mathcal{R}$.

%Here we will separate $\Pi(z)$ in to 9 parts, each of them is one of the 9 blocks. We take derivative at the 9 parts separately, then the word w represent different positions depending on which part we take derivative at. In other words, the lower index $\bbs_l$ and $\bbt_l$ only represent the positions of the block instead of $\Pi$.

We below associate a random variable $B_{a,b,i,\mu}(g,g^{(1)},\cdots,g^{(7)})$ with each group $g,g^{(j)},j=1,\cdots,7$. When $k(g)=k(^{(j)})=0$ we define
$$B_{a,b,i,\mu}(g,g^{(1)},\cdots,g^{(7)}))= (\bbJ_1\bbG(z)\bbJ_2^T)_{ab}-(\bbJ_1\bbF(z)\bbJ_2^T)_{ab}.$$
 When $k(g)\geq1$ and $k(g^{(j)})\ge1$, define
 \begin{equation}\label{a16}B_{a,b,i,\mu,\bbR_{2,\cdots,k},\mathcal{R}_{11,\cdots,7k+1}}(g,g^{(1)},...,g^{(7)})=
C_{a,b,i,\mu,\bbR_{2,\cdots,k},\mathcal{R}_{11,\cdots,7k+1}}(g,g^{(1)},\cdots,g^{(7)}))
\end{equation}$$-\sum_{j=1}^7(\bbJ_1\mathcal{R}_{j1})_{(a_{j1}b_{j1})}(\mathcal{R}_{j2})_{(a_{j2}b_{j2})}...(\mathcal{R}_{jk})_{(a_{jk}b_{jk})}
(\mathcal{R}_{jk+1}\bbJ_2^T)_{(a_{jk+1}b_{jk+1})},
$$
with
\begin{equation}\label{0526.2}C_{a,b,i,\mu,\bbR_{2,\cdots,k},\mathcal{R}_{11,\cdots,7k+1}}(g,g^{(1)},\cdots,g^{(7)}))=(\bbJ_1G \bbA_5)_{(a_1b_1)}(\bbR_2)_{(a_2b_2)}...(\bbR_k)_{(a_kb_k)}(\bbA_4\bbG\bbJ_2^T)_{(a_{k+1}b_{k+1})},\end{equation}
where $\bbR_j (2\le j\le n)$ has the expression of $\bbR_j=\bbA_4\bbG \bbA_5$
with $\bbA_4\in \{1, \bold\Delta \}$, $\bbA_5 \in \{1,\bold\Delta^T\}$ and the non-zero block $\mathcal{R}_{jl}$ belongs to $\mathcal{Q}(k)$ in (\ref{1202.1}).
Moreover the selection of $1$ and $\bold\Delta$ in $\bbA_4$ and $\bbA_5$ is subject to the constraint that the total number of $\bold\Delta$ and $\bold\Delta^T$ contained in  $B_{a,b,i,\mu,\bbR_{2,\cdots,k},\mathcal{R}_{11,\cdots,7k+1}}(g,g^{(1)},...,g^{(7)})$ is $k$. One should also notice that if $k(g)=1$, the terms $R_j$ will disappear. % By (\ref{1116.6}), there is a new $\Delta$ or $\Delta^T$ after one differentiation, which implies that the number of $\Delta$ and $\Delta^T$ belong to $A_{s,t,i,\mu,R_2,...,R_k,\mathcal{R}_{11},...,\mathcal{R}_{1k+1},...,\mathcal{R}_{7k+1}}(w,v_1,...,v_7)$ is k(w). The lower index of $\mathcal{R}_{jl}$ may not represent the exact position of the matrix  $\mathcal{R}_{jl}$, but we also use such hazy notation to represent that each differentiation on \bbX will create a matrix $E_{i\mu}$.
It follows from (\ref{a10}) that
 \begin{equation}\label{a11}\|\mathcal{R}_{jl}\|\prec 1.
 \end{equation}

%for any $j\in \mathbb{N}$.

It is easy to see that
\begin{equation}
\label{a13}
\frac{\partial \bbG}{\partial X_{i\mu}}=-\bbG(\bbe_{\mu+p}\bbe_i^T\bold\Delta+\bold\Delta^T\bbe_i\bbe_{\mu+p}^T)\bbG,
\end{equation}
(one may see (\ref{a12}) for the derivative ). We first demonstrate how to apply the above definitions about groups $g^{(j)}$ and $B_{a,b,i,\mu,\bbR_{2,\cdots,k},\mathcal{R}_{11,\cdots,7k+1}}(g,g^{(1)},...,g^{(7)})$ and hence write
\begin{eqnarray}\label{a38}
&&\frac{\partial^k }{\partial (X_{i\mu})^k}\Big([(\bbJ_1\bbG(z)\bbJ_2^T)_{ab}-(\bbJ_1\bbF(z)\bbJ_2^T)_{ab}]\mathcal{T}_n(\bbX)\Big)\\
&&=(-1)^k\sum_{g\in \mathfrak{G}_{k},g^{(j)}\in \mathfrak{G}_{jk} \atop{\bbR_i, i=2,...,k \atop \mathcal{R}_{jl}, j=1,..7, l=1,...,k+1} }B_{a,b,i,\mu,\bbR_{2,\cdots,k},\mathcal{R}_{11,\cdots,7k+1}}(g,g^{(1)},...,g^{(7)})\mathcal{T}_n(\bbX)+O_{\prec}(0),\nonumber
\end{eqnarray}
where the term $O_{\prec}(0)$ comes from the derivative on $\mathcal{T}_n(\bbX)$ by (\ref{1219.8}), (\ref{1219.2}) and (\ref{0304.5}). %Since the exact expressions of the parameters $\bbR_2,\cdots,\bbR_k$ and $\mathcal{R}_{11},...,\mathcal{R}_{7k+1}$ do not
%play an important role below,
To simplify the notations, we furthermore omit $\bbR_{2\cdots,k},\mathcal{R}_{11,...,7k+1}, g^{(1)},...,g^{(7)}$ in the sequel and write
\begin{equation}\label{a36}B_{a,b,i,\mu}(g)=B_{a,b,i,\mu,\bbR_{2,\cdots,k},\mathcal{R}_{11,\cdots,7k+1}}(g,g^{(1)},...,g^{(7)}),\end{equation}
\begin{equation}\label{a37}C_{a,b,i,\mu}(g)=C_{a,b,i,\mu,\bbR_{2,\cdots,k},\mathcal{R}_{11,\cdots,7k+1}}(g,g^{(1)},...,g^{(7)}),\end{equation}
(here one should notice that the sizes of $g$ and $g^{(j)}$ are the same according to definition (\ref{a16})).
%for some parameters $R_2,....,R_{\tilde k},\mathcal{R}_{11},...,\mathcal{R}_{7k+1}, v_1,...,v_7$. Here we do not need to know the explicit expression of these parameters because they don't play an important role in the following proof by $\|\mathcal{R}_{jl}\|\prec 1$.
More generally we furthermore have
\begin{eqnarray}\label{a14}
&\frac{\partial^k }{\partial (X_{i\mu})^k}\Big(|F_{ab}(\bbX)|^{2q}\Big)=(-1)^k\sum\limits_{k_1,...,k_{q}, \tilde k_1,...,\tilde k_{q}\in \mathbb{N} \atop
\sum_{r}(k_r+\tilde k_r)=k}\frac{k!}{\prod_r k_r!\tilde k_r!}
\\ &\times\prod\limits_{r=1}^{q}(\sum\limits_{g_r\in\mathfrak{G}_{k_r}\cup\mathfrak{G}_{jk_r}\atop{\bbR_i, i=2,...,k \atop \mathcal{R}_{jl}, j=1,..7, l=1,...,k+1} }\sum\limits_{\tilde g_r\in \mathfrak{G}_{\tilde k_r}\cup\mathfrak{G}_{ j\tilde k_r}\atop{\bar{\bbR}_i, i=2,...,k \atop \mathcal{\bar {R}}_{jl}, j=1,..7, l=1,...,k+1}}B_{a,b,i,\mu}(g_r)\overline{B_{a,b,i,\mu}(\tilde g_r)}\mathcal{T}^2_n(\bbX))
+O_{\prec}(0),\nonumber
\end{eqnarray}
where $g_r\in \mathfrak{G}_{ k_r}\cup\mathfrak{G}_{jk_r}$ means that the groups associated with the derivatives of $\bbG(z)$ belong to $\mathfrak{G}_{k_r}$ and the groups associated with the derivatives of $\bbF(z)$ belong to $\mathfrak{G}_{ jk_r}$. In view of (\ref{a14}) and (\ref{a8}) to prove (\ref{1123.7}) it then suffices to show that
\begin{equation}\label{1117.2}
n^{-k/2}\sum_{i=1}^p\sum_{\mu=1}^n\mathbb{E}\left[\prod_{r=1}^{q}B_{a,b,i,\mu}(g_r)\overline{B_{a,b,i,\mu}(\tilde g_r)}\mathcal{T}^{2q}_n(\bbX)\right]=O((n^{24\delta}\Psi)^{2q}+\|\mathbb{E}|F(\bbX)|^{2q}\|_{\infty}),
\end{equation}
for $4\le k \le16q$ and groups $g_r\in\mathfrak{G}_{k_r}\cup\mathfrak{G}_{jk_r}$, $\tilde g_r\in\mathfrak{G}_{\tilde k_r}\cup\mathfrak{G}_{j\tilde k_r}$ satisfying $\sum_{r}(k(g_r)+\tilde k(g_r))=k$. To simplify notations, we drop complex conjugates (which will complicate the notations but the proof is the same) from the left hand side of (\ref{1117.2}). Without loss of generality, suppose there are (2q-l) terms such that $k(g_r)=0$ and denote each of them by $g_0$. (\ref{1117.2}) reduces to
\begin{eqnarray}\label{1117.3}
n^{-k/2}\sum_{i=1}^p\sum_{\mu=1}^n\mathbb{E}\left[B_{a,b,i,\mu}(g_0)^{2q-l}\prod_{r=1}^lB_{a,b,i,\mu}(g_r)\mathcal{T}^{2q}_n(\bbX)\right]=O((n^{24\delta}\Psi)^{2q}+\|\mathbb{E}|F(\bbX)|^{2q}\|_{\infty}),
\end{eqnarray}
for $4\le k \le16q$ and groups $g_r\in\mathfrak{G}_{k_r}\cup\mathfrak{G}_{jk_r}$ satisfying $\sum_{r}k(g_r)=k$ and $k(g_0)=0$.

To estimate the left hand of (\ref{1117.3}), we introduce the notations
$$\mathcal{H}_i = \mathcal{H}_{1i}+\mathcal{H}_{abi},\ \   \mathcal{H}_{1i}= |(\bbJ_1\bbG\bold\Delta)_{ai}|+|(\bold\Delta^T\bbG\bbJ_2^T)_{ib}|, \ \ \mathcal{H}_{abi}= \sum_{\mathcal{R}\in \mathcal{Q}(k)}(|(\bbJ_1\mathcal{R})_{ai}|+|(\mathcal{R}\bbJ_2^T)_{ib}|),$$
 $$\mathcal{H}_{\mu} = \mathcal{H}_{1\mu}+\mathcal{H}_{a\mu},\ \  \mathcal{H}_{1\mu}= |(\bbJ_1\bbG)_{a (\mu+p)}|+|(\bbG\bbJ_2^T)_{(\mu+p) b}|, \ \ \mathcal{H}_{a\mu}=\sum_{\mathcal{R}\in \mathcal{Q}(k)}(|(\bbJ_1\mathcal{R})_{a\mu}|+|(\mathcal{R}\bbJ_2^T)_{\mu a}|,$$
where the lower indices $i$ and $\mu$ at $\bbJ_1\mathcal{R}$ and $\mathcal{R}\bbJ_2^T$ respectively represent the index $i$, $i+p-m$, $i+p$, and $\mu$, $\mu+p-n$ or $\mu+p$ depending on which block we consider (or differentiate). %It is easy to find out that the terms in the summation is a constant only depending on q.
By (\ref{1116.7}), (\ref{0304.5}) and (\ref{a10}) we have
\begin{equation}\label{a15}\mathcal{H}_i+\mathcal{H}_{\mu}\prec n^{2\delta}.\end{equation}
 %\begin{lem}\label{1124-4}
 Moreover for $g_r\in \mathfrak{G}_{k_r}\cup\mathfrak{G}_{jk_r}$, we similarly obtain from (\ref{1116.7}), (\ref{0304.5}), (\ref{a10}) and definition (\ref{a16}) that %(to simplify the expression, we omit the lower index $R_2,....,R_{k_r}$.)
\begin{eqnarray}\label{1117.4}
|B_{a,b,i,\mu}(g_r)|\prec n^{2\delta (k(g)+1)},
\end{eqnarray}
(recall $k(g)=k(g^{(j)})$ from definition (\ref{a16})).
Likewise, for $k(g)\ge 1$, we have
\begin{eqnarray}\label{1117.5}
|B_{a,b,i,\mu}(g_r)|\prec (\mathcal{H}_i^2+\mathcal{H}_{\mu}^2)n^{2\delta(k(g_r)-1)},
\end{eqnarray}
while k(g)=1,
\begin{eqnarray}\label{1117.6}
|B_{a,b,i,\mu}(g_r)|\prec \mathcal{H}_i\mathcal{H}_{\mu}.
\end{eqnarray}
%\end{lem}

When $k\le 2l-2$ there must exist at least 2 $g_r$'s satisfying $k(g_r)=1$ because $\sum_{r=1}^l k(g_r)= k \le 2l-2$. It follows from (\ref{1117.4}) and (\ref{1117.6}) that
\begin{eqnarray}\label{1117.7}
|B_{a,b,i,\mu}(g_0)^{2q-l}\prod_{r=1}^lB_{a,b,i,\mu}(g_r)|&\prec & n^{2\delta(k+l)}F^{2q-l}_{ab}(\bbX)\Big(I(k\ge 2l-1)(\mathcal{H}_i^2+\mathcal{H}_{\mu}^2)\non
&+&
I(k\le 2l-2)\mathcal{H}_i^2\mathcal{H}_{\mu}^2\Big).
\end{eqnarray}
Recalling  the notation $\bold\Delta$ in (\ref{a6}) we have $\|\bold\Delta\bold\Delta^T\|\leq M$. % $\bold\Delta^T\bold\Delta=\left(
%  \begin{array}{ccc}
%    \bbD^{-1} &0&0\\
%    0& \bbI&0\\
%    0&0&   0\\
%  \end{array}
%\right)$.
In view of (\ref{a10}) it is easy to see that
\begin{equation}\label{g28}\sum_{i=1}^p\mathcal{H}_{1i}^2+\sum_{\mu=1}^n \mathcal{H}_{1\mu}^2\prec n\phi_a^2+n\phi_b^2,\end{equation}
\begin{equation}\label{h28}\sum_{i\ \text{or}\ a \ \text{or} \ b}^p\mathcal{H}_{abi}^2+\sum_{a\ \text{or}\ \mu}^n \mathcal{H}_{a\mu}^2\prec 1,\end{equation}
%the reason why we have additional $n^{1/2}$ is that $\Delta\Delta^*$ in the middle of $(BG\Delta\Delta^*G^*B^*)$:
%\begin{eqnarray}\label{1122.2}
%|(BG\Delta\Delta^*G^*B^*)_{ss}|&\le&\sum_{i}|(BG)_{si}||(\Delta\Delta^*G^*B^*)_{is}|\le(\sum_{i}|(BG)_{si}|^2)^{\frac{1}{2}}(\sum_{i}|(\Delta\Delta^*G^*B^*)_{is}|^2)^{\frac{1}{2}}\non
%&\prec& \sqrt n \frac{\Im(BGB^*)_{ss}+\eta}{\eta}.
%\end{eqnarray}
where $i\ \text{or}\ a \ \text{or} \ b$ means the summation over either $i$ or $a$ or $b$ and
$$\phi_a^2= \frac{\Im(\bbJ\bbG\bbJ^*)_{aa}+\eta}{n\eta},$$
with $\bbJ\in \mathcal{L}$ defined in (\ref{a5}). This implies that
\begin{equation}\label{a28}\sum_{i=1}^p\mathcal{H}_{i}^2+\sum_{\mu=1}^n \mathcal{H}_{\mu}^2\prec n\phi_a^2+n\phi_b^2.\end{equation}
From (\ref{a19}) and (\ref{a21})
$$\phi_a^2= \frac{\Im(\bbJ\bbF\bbJ^*)_{aa}+ \Im(\bbJ(\bbG-\bbF)\bbJ^*)_{aa}+\eta}{n\eta}\prec \frac{\Im m+ \Im(\bbJ(\bbG-\bbF)\bbJ^*)_{aa}}{n\eta}.$$
Recalling the definition of $\Psi$ in (\ref{a17}) we conclude that
\begin{eqnarray}\label{1120.5}
\phi_a^2\prec \Psi(\Psi+F_{aa}(\bbX)).
\end{eqnarray}
By (\ref{a27}), (\ref{a26}), (\ref{1119.12}), (\ref{1124.13}), the definition of $\mathcal{T}_n(\bbX)$ and definition (\ref{a16}) we have
\begin{eqnarray}\label{0312.1}
|B_{a,b,i,\mu}(g_0)B_{a,b,i,\mu}(g_r)\mathcal{T}_n(\bbX)|\le n^{M_0}.
\end{eqnarray}
From (\ref{1117.7}), (\ref{a28}), (\ref{0312.1}) and (\ref{b6}) the left hand side of (\ref{1117.3}) is bounded in absolute value by
\begin{eqnarray}\label{1117.8}
n^{-k/2+2}n^{3\delta(k+l)}\mathbb{E}F_{ab}^{2q-l}(\bbX)\Big[I(k\ge 2l-1)(\phi_a^2+\phi_b^2)+I(k\le 2l-2)(\phi_a^4+\phi_b^4)\Big]+n^{-D}.
\end{eqnarray}
Set
$$F^{2q}_1= F^{2q}_{aa}+F^{2q}_{ba}+F^{2q}_{ab}.$$
We conclude from (\ref{1120.5})-(\ref{0312.1}) and (\ref{b6}) that the left hand side of (\ref{1117.3}) is bounded in absolute value by
\begin{eqnarray}\label{1120.6}
  \begin{cases}
  n^{3\delta(k+l)}\left(\Psi^{k-2}\mathbb{E}F_1^{2q-l}(\bbX)+\Psi^{k-3}\mathbb{E}F_1^{2q-l+1}(\bbX)\right)+n^{-D},  &\mbox{if}\ \ k\ge2l-1, \\
   n^{3\delta(k+l)}\left(\Psi^k\mathbb{E}F_1^{2q-l}(\bbX)+\Psi^{k-2}\mathbb{E}F_1^{2q-l+2}(\bbX)\right)+n^{-D},  &\mbox{if}\ \ k\le2l-2.
   \end{cases}
\end{eqnarray}
Since $l\le k$ (\ref{1120.6}) is further bounded by
\begin{eqnarray}\label{1120.7}
  \begin{cases}
  (n^{24\delta}\Psi)^{k-2}\mathbb{E}F_1^{2q-l}(\bbX)+(n^{24\delta}\Psi)^{k-3}\mathbb{E}F_1^{2q-l+1}(\bbX)+n^{-D},  &\mbox{if}\ \ k\ge2l-1, \\
   (n^{24\delta}\Psi)^k\mathbb{E}F_1^{2q-l}(\bbX)+(n^{24\delta}\Psi)^{k-2}\mathbb{E}F_1^{2q-l+2}(\bbX)+n^{-D},  &\mbox{if}\ \ k\le2l-2.
   \end{cases}
\end{eqnarray}
This ensures that the left hand side of (\ref{1120.6}) is bounded in absolute value by
\begin{eqnarray}\label{1120.8}
(n^{24\delta}\Psi)^{l}\mathbb{E}F_1^{2q-l}(\bbX)+(n^{24\delta}\Psi)^{l-1}\mathbb{E}F_1^{2q-l+1}(\bbX)+(n^{24\delta}\Psi)^{l-2}I(l\ge 3)\mathbb{E}F_1^{2q-l+2}(\bbX)+n^{-D},
\end{eqnarray}
where we use the facts that $k\geq l+2$ when $k\geq 4$ and $k\geq 2l-1$ and that $k\geq l$ and $l\geq 3$ when $k\le2l-2$ and $k\geq 4$.
When $l\ge 2$, (\ref{1117.3}) follows from (\ref{1120.8}), the facts that $(E|X|^r)^{1/r}$ is a nondecreasing function of $r$ and that $n^{-D}\leq (n^{24\delta}\Psi)^{2q}$ for sufficiently large $D$. For example
\begin{equation}\label{a30}
(n^{24\delta}\Psi)^{l-2}\mathbb{E}F_1^{2q-l+2}(\bbX)\leq \Big(\mathbb{E}\Big(F_1^{2q-l+2}(\bbX)\Big)^{\frac{2q}{2q-l+2}} \Big)^{\frac{2q-l+2}{2q}} \Big((n^{24\delta}\Psi)^{2q}\Big)^{\frac{l-2}{2q}}
\end{equation}$$\leq \Big(\mathbb{E}\Big(F_1^{2q}(\bbX)\Big)+(n^{24\delta}\Psi)^{2q}\Big)^{\frac{2q-l+2+l-2}{2q}}.
$$ When $l=1$, the first term can be handled similarly and the second term directly implies (\ref{1117.3}).
Thus we have proved (\ref{0310.1}) in Theorem \ref{1119-3}.

\subsection{Local law (\ref{0310.2})}

This subsection is to prove (\ref{0310.2}) in Theorem \ref{1119-3}, i.e.
\begin{eqnarray}\label{1219.3}
|\underline{m}_n(z)-\underline{m}(z)|\prec \frac{1}{n\eta}.
\end{eqnarray}

As pointed out in the paragraph containing (\ref{b4}), (\ref{1219.3}) holds when the underlying distribution of $X_{ij}$ of $\bbX$ is the standard Gaussian distribution. Moreover, we need to use the interpolation method to prove (\ref{0310.2}) for the general distributions as in proving (\ref{0310.1}). However we do not need induction on the imaginary part of $z$ unlike before due to existence of (\ref{0310.1}).
 %It suffices to prove the following Proposition, which is a little stronger than (\ref{1219.3}).
%\begin{prop}\label{1119.1}
%Suppose that  (\ref{1219.3})  holds with parameters $(\bbX^{Gauss},E,S,\Phi)$, then (\ref{1219.3}) holds with parameters $(\bbX,E,S,\Phi)$, where $\Phi\ge \Psi^2+1/n$.
%\end{prop}

In order to prove (\ref{1219.3}) it suffices to show that
\begin{eqnarray}\label{1219.5}
|\underline{m}_n(z)-\underline{m}(z)|\mathcal{T}_n(\bbX)\prec \frac{1}{n\eta}.
\end{eqnarray}
As in (\ref{a5}) we introduce the notation $\hat F^{2q}(\bbX,z)$ as follows
$$\hat F^{2q}(\bbX,z)=|\underline{m}_n(z)-\underline{m}(z)|^{2q}\mathcal{T}^{2q}_n(\bbX)=|\frac{1}{m}\sum_{k}^m G_{kk}(z)-\underline{m}(z)|^{2q}\mathcal{T}^{2q}_n(\bbX).$$
Checking on Lemmas \ref{1216-3}, \ref{1124.14*}, \ref{1119-5}, (\ref{b7}) and (\ref{1123.7}) in the last section we only need to show
\begin{equation}\label{a34}n^{-k/2}\sum_{i=1}^p\sum_{\mu=1}^n\mathbb{E}\Big[(\frac{\partial}{\partial \bbX_{i\mu}})^k\hat F^{2q}(\bbX,z)\Big]=O((n^{\delta}\Psi^2)^{2q}+\|\hat
F^{2q}(\bbX,z)\|_{\infty}),\ k\geq 4\end{equation}
where $\delta$ is sufficiently small so that $n^\delta$ is smaller than $n^\varepsilon$ before $(\ref{1219.5})$ due to the definition of the partial order.
Applying the definition of $B_{a,b,i,\mu}$ in the previous section with $\bbJ_1=\bbJ_2=1$ and $a=b=k$, it suffices to show that
\begin{eqnarray}\label{1119.2}
n^{-k/2}\sum_{i=1}^p\sum_{\mu=1}^n\mathbb{E}\prod_{h=1}^{2q}\left[\frac{1}{m}\sum_{k=1}^mB_{k,k,i,\mu}(g_h)\right]=O((n^{\delta}\Psi^2)^{2q}+\|\mathbb{E}\hat F^{2q}(\bbX,z)\|_{\infty}).
\end{eqnarray}
Notice that (\ref{0310.1}) holds uniformly for any unit determinant vectors $\bbv$ , $\bbw$ and $z\in S$. This, together with (\ref{1120.5}), implies that
$$
\phi_s^2\prec \Psi^2.
$$
We then conclude from (\ref{1117.7}) and (\ref{a28}) that
\begin{eqnarray}\label{1119.4}
\frac{1}{m}\sum\limits_{k=1}^mB_{k,k,i,\mu}(g_h)\prec \Psi^2, \ \ \text{for} \ \ g(w)\ge 1.
\end{eqnarray}
For future use, recalling (\ref{0526.2}) and (\ref{a37}) we also obtain from (\ref{h28}) and (\ref{1119.4})
\begin{eqnarray}\label{0526.1}
\frac{1}{m}\sum\limits_{k=1}^m C_{k,k,i,\mu}(g_h)\prec \Psi^2, \ \ \text{for} \ \ g(w)\ge 1.
\end{eqnarray}
 As in (\ref{1117.7}) we then have
 $$
|\frac{1}{m}\sum\limits_{k=1}^mB_{k,k,i,\mu}(g_0)^{2q-l}\prod_{r=1}^l\frac{1}{m}\sum\limits_{k=1}^mB_{k,k,i,\mu}(g_r)|\prec \hat F^{2q-l}(\bbX,z)\Psi^{2l}.
 $$
(\ref{1119.2}) and hence (\ref{0310.2}) then follow via (\ref{b6}) and an argument similar to (\ref{a30}).
%(\ref{1119.4}), together with (\ref{1117.7})-(\ref{1120.8}), ensures the averaged local law, i.e.
%\begin{eqnarray*}
%|\underline{m}_n(z)-\underline{m}(z)|\prec \frac{1}{n\eta}.
%\end{eqnarray*}

\subsection{Convergence rate on the right edge and  universality}

\subsubsection{Convergence rate on the right edge}
The aim of this subsection is to prove the following Lemma.
\begin{lem}\label{0525-1} Denote by $\lambda_1$ the largest eigenvalue of $\bbA$ in (\ref{0313.1}). Under conditions of Theorem \ref{t1},
$$\lambda_1-\hat \mu_m=O_{\prec}(n^{-\frac{2}{3}}).$$
\end{lem}
\begin{proof}
The approach is similar to that in \cite{LHY2011}, (\cite{PY11}) and \cite{BPWZ2014b}.   Checking on the proof of Theorem 4.1 in \cite{BPWZ2014b} carefully,  %by (\ref{0310.2}),
we find that (ii) in Theorem 4.1 in \cite{BPWZ2014b} and hence the lower bound of $\lambda_1$ of Lemma \ref{0525-1} still hold in our case because of (\ref{a24}) and (\ref{b4}).
It then suffices to prove that for any small positive constant $\tau$
\begin{eqnarray}\label{1119.5}
\lambda_1\le \hat \mu_m+n^{-2/3+\tau}
\end{eqnarray}
holds with high probability. %By Lemma \ref{1223-1}, we have $\|\bbX\bbX^T\|\le C$ with high probability.
Note that by (\ref{b5}) and Lemma \ref{1121-1}
\begin{equation}\label{a35}\|\bbA\|\le M\end{equation}
with high probability for sufficient large positive constant $M$ (here one should notice that $\|\bbD^{-1}\|\leq M$ with high probability due to (\ref{6.2}) and (\ref{6.3})).
For a suitably small $\tau$, set $z=E+i\eta$ and $\kappa=|E-\hat \mu_m|$ where $E\in[\hat \mu_m+n^{-2/3+\tau},\hat \mu_m+\tau^{-1}]$ and $\eta=n^{-1/2-\tau/4}\kappa^{1/4}$. By Lemma 2.3 of \cite{BPWZ2014b}, we have
\begin{equation}\label{1119.7}
\Im \underline{m} \asymp \frac{\eta}{\sqrt{\kappa+\eta}}\ll \frac{1}{n\eta},
\end{equation}
where $\ll$ means much less than.

We furthermore claim that with high probability
\begin{eqnarray}\label{1119.8}
|\underline{m}_n-\underline{m}|\ll\frac{1}{n\eta}%\frac{\Psi^2}{\sqrt{\kappa+\eta}}+\frac{1}{n}.
\end{eqnarray}
Indeed, (\ref{1119.8}) holds when $\bbX$ reduces to $\bbX^0$ due to (4.6) in \cite{BPWZ2014b}, (\ref{1119.7}) and (\ref{b4}). For the general distributions, (\ref{1119.8}) follows from (\ref{a34}) and (\ref{a33}).
It follows from (\ref{1119.7}) and (\ref{1119.8}) that with high probability
$$
\Im(\underline{m}_n)\ll \frac{1}{n\eta}.
$$
Moreover note that with high probability
$$
\sum\limits_iI(E-\eta\leq\lambda_i\leq E+\eta)\leq Mn\eta \Im(\underline{m}_n)\ll 1.
$$
As a consequence there is no eigenvalue in $[E-\eta,E+\eta]$ with high probability. This, together with (\ref{a35}), ensures (\ref{1119.5}).
%To prove (\ref{1119.6}), it suffices to show that $\Im \underline{m}_n(z)\prec n^{-\tau/2}(n\eta)^{-1}$. Therefore it suffices to prove that we have with high probability,
%\begin{eqnarray}\label{1119.6}
%spec(\bbA)\cap [\hat \mu_m+n^{-2/3+\tau},\hat \mu_m+\tau^{-1}]=\emptyset.
%\end{eqnarray}

%By (\ref{1219.3}) and $\kappa+\eta\prec 1$, we have shown Lemma \ref{0525-1}.
\end{proof}

\subsubsection{Universality}
The aim of this subsection is to prove (ii) of Theorem \ref{t1}. By (\ref{0529.1}) and (\ref{0529.2}), it suffices to prove edge universality at the rightmost edge of the support $\hat \mu_m$. In other words, the asymptotic distribution of $\lambda_1$ is not affected by the distribution of the entries of $\bbX$ under the 3rd moment matching condition. Similar to theorem 6.4 of \cite{LHY2011}, we first show the following green function comparison theorem.
\begin{thm}\label{1221-1}
There exists $\varepsilon_0>0$. For any $\ep <\varepsilon_0$, set $\eta=n^{-2/3-\ep}$, $E_1$, $E_2\in \mathbb{R}$ with $E_1<E_2$ and
$$|E_1-\hat \mu_m|,|E_2-\hat \mu_m|\le n^{-2/3+\ep}.$$
Suppose that $K:\mathbb{R}\rightarrow \mathbb{R}$ is a smooth function with bounded derivatives up to fifth order.
%$$\max_{x}(|K^{(l)}(x)|)\le C, l=0,1,2,3,4,5.$$
Then there exists a constant $\phi>0$ such that for large enough n %and small enough $\ep$, we have
\begin{eqnarray}\label{1119.9}
|\mathbb{E}K(n\int_{E_1}^{E_2}\Im m_{\bbX^1}(x+i\eta)dx)-\mathbb{E}K(n\int_{E_1}^{E_2}\Im m_{\bbX^0}(x+i\eta)dx)|\le n^{-\phi},
\end{eqnarray}
(see Definition \ref{1222-1} or (\ref{1129.1}) for $\bbX^1$ and $\bbX^0$).
\end{thm}
\begin{proof}
%To show the universality, we use similar argument before. %It is easy to see that
%$$|\mathbb{E}F(n\int_{E_1}^{E_2}\Im m_{\bbX^0}(x+\hat \mu_m+i\eta)dx)-\mathbb{E}F(n\int_{E_1}^{E_2}\Im m_{W}(x+\hat \mu_m+i\eta)dx)|\equiv0,$$
%where W is an independent copy of $\bbX^0$.
%
Unlike \cite{KY14}, \cite{LHY2011} and \cite{BPZ2014a} we use the interpolation method (\ref{1124.14}), which is succinct and powerful when proving green function comparison theorem.
In view of (\ref{b2}) and (\ref{b3}) we have
\begin{eqnarray}\label{1119.10}
&&|\mathbb{E}K(n\int_{E_1}^{E_2}\Im m_{\bbX^1}(x+i\eta)dx)-\mathbb{E}K(n\int_{E_1}^{E_2}\Im m_{\bbX^0}(x+i\eta)dx)|=\non
&&\left|\mathbb{E}K(n\int_{E_1}^{E_2}\Im m_{\bbX^1}(x+i\eta)\mathcal{T}_n(\bbX^1)dx)-\mathbb{E}K(n\int_{E_1}^{E_2}\Im m_{\bbX^0}(x+i\eta)\mathcal{T}_n(\bbX^0)dx)\right|+O(n^{-1}).\non
\end{eqnarray}
Applying (\ref{1124.14}) with $F(\bbX)=K(n\int_{E_1}^{E_2}\Im m_{\bbX}(x+i\eta)\mathcal{T}_n(\bbX ))$ we only need to bound the following
\begin{eqnarray}\label{1120.1}
\sum_{i=1}^m\sum_{\mu=1}^p \Bigg|\mathbb{E}g(X_{i\mu}^1)-\mathbb{E}g(X_{i\mu}^0))\Bigg|,
\end{eqnarray}
where
$$
g(X_{i\mu}^u)=K(n\int_{E_1}^{E_2}\Im m_{\bbX_{(i\mu)}^{t,X_{i\mu}^u}}(x+i\eta)\mathcal{T}_n(\bbX_{(i\mu)}^{t,X_{i\mu}^u})dx),\quad u=0,1.
$$
As in (\ref{1123.4}) and (\ref{1123.5}), we use Taylor's expansion up to order five to expand two functions $g(X_{i\mu}^u),u=0,1$ % and $K(n\int_{E_1}^{E_2}\Im m_{\bbX_{(i\mu)}^{u,X_{i\mu}^1}}(x+\hat \mu_m+i\eta)\mathcal{T}_n(\bbX_{(i\mu)}^{u,X_{i\mu}^0})dx)$
 at the point 0. Then take the difference of the Taylor's expansions of $g(X_{i\mu}^u),u=0,1$. %(here $K(\cdot)$ is regarded as a function of $X_{i\mu}^1$ or $X_{i\mu}^0$).
  By the 3rd moments matching condition it then suffices to bound the fourth derivative
 %-\Im m_{W}(x+\hat \mu_m+i\eta)$. The advantage we use W instead of $\bbX^0$ is that we don't need to differentiate W.
\begin{equation}\label{0524.1}
\sum_{r=1}^4 \sum_{k_1,..,k_r\in \mathbb{N}_+ \atop k_1+..+k_r=4}M_r \max_{x}|K^{(r)}(x)|  \mathbb{E}\prod_{i=1}^r \Bigg(n\int_{E_1}^{E_2}\Bigg| m_{\bbX_{(i\mu)}^{t,0}}^{(k_i)}(x+i\eta)\mathcal{T}_n(\bbX_{(i\mu)}^{t,0}) \Bigg|dx\Bigg),
\end{equation}
and the fifth derivative corresponding to the remainder of integral form
\begin{equation}\label{0524.1*}
\frac{1}{\sqrt{n}}\sum_{r=1}^5 \sum_{k_1,..,k_r\in \mathbb{N}_+ \atop k_1+..+k_r=4}M_r \max_{x}|K^{(r)}(x)|  \mathbb{E}\prod_{i=1}^r \Bigg(n\int_{E_1}^{E_2}\Bigg| m_{\bbX_{(i\mu)}^{t,\theta X_{i\mu}^u}}^{(k_i)}(x+i\eta)\mathcal{T}_n(\bbX_{(i\mu)}^{t,\theta X_{i\mu}^u }) \Bigg|dx\Bigg),
\end{equation}
where $M_r$ is a constant depending on r only, $m_{\bbX_{(i\mu)}^{t,0}}^{(k_i)}(\cdot)$ denotes the $k_i$th derivative with respect to $X_{i\mu}^u$ and $0\leq\theta\leq 1$. Here we ignore the terms involving the derivatives of $\mathcal{T}_n(\bbX_{(i\mu)}^{t,\theta X_{i\mu}^u})$ due to (\ref{b2}), (\ref{b3}) and (\ref{1119.12}). % One should also notice that the remainder (the 5th derivative) can be bounded well similar to Lemma \ref{1123-2} if we can bound (\ref{0524.1}) at the order $O_{\prec}(n^{\frac{1}{3}+\ep}\Psi^2)$.

To investigate (\ref{0524.1}) and (\ref{0524.1*}) we claim that it suffices to prove that
\begin{eqnarray}\label{0524.2}
\Bigg(n\int_{E_1}^{E_2}\Bigg| m_{\bbX_{(i\mu)}^{u,X_{i\mu}^1}}^{(k)}(x+i\eta)\mathcal{T}_n(\bbX_{(i\mu)}^{u,X_{i\mu}^1})\Bigg|dx\Bigg)\prec(n^{\frac{1}{3}+\ep}\Psi^2),
\end{eqnarray}
where $k\ge 1$. Indeed, if (\ref{0524.2}) holds then (\ref{0524.2}) still holds if $X_{i\mu}^1$ is replaced by $\theta X_{i\mu}^1$ by checking on the argument of (\ref{0524.2}). We then conclude that the facts that $(\ref{0524.1})\prec(n^{\frac{1}{3}+\ep}\Psi^2)$ and that $(\ref{0524.1*})\prec(n^{-\frac{1}{2}+\frac{1}{3}+\ep}\Psi^2)$ follow from Lemma \ref{1123-3}, (\ref{1119.12}) and an application of (\ref{1123.4}).

By (\ref{0526.2}) and (\ref{0526.1}) we have for $k\ge 1$
$$\Bigg| m_{\bbX_{(i\mu)}^{u,X_{i\mu}^1}}^{(k)}(x+i\eta)\mathcal{T}_n(\bbX_{(i\mu)}^{u,X_{i\mu}^1})\Bigg|\prec \Psi^2,$$
which implies that $(\ref{0524.2})\prec(n^{\frac{1}{3}+\ep}\Psi^2)$. Here we would point out that the derivatives $m_{\bbX_{(i\mu)}^{u,X_{i\mu}^1}}^{(k)}(\cdot)$ are of the form $\frac{1}{m}\sum\limits_{k=1}^mC_{k,k,i,\mu}(g_h)$ from (\ref{a16}), (\ref{0526.2}), (\ref{a38}), (\ref{a37}), (\ref{a34}) and (\ref{1119.2}). By Lemma 2.3 of \cite{BPWZ2014b} we have
$$\Psi^2 \asymp \frac{1}{n\sqrt{\eta}}=O(n^{-\frac{2}{3}+\ep/2}).$$
Summarizing the above we have shown that
$$|\mathbb{E}K(n\int_{E_1}^{E_2}\Im m_{\bbX^1}(x+i\eta)dx)-\mathbb{E}K(n\int_{E_1}^{E_2}\Im m_{\bbX^0}(x+i\eta)dx)|\prec n^{-\frac{1}{3}+2\ep}.$$
The proof is complete by choosing an appropriate $\ep$. %, we can finish the proof.
% Noticing that when $\frac{1}{m}\sum_{v=1}^n| G^{(k)}_{vv}(\bbX_{(i\mu)}^{u})|\ge n^{\ep} \Psi^2$, we use the smooth cutoff function $\mathcal{T}_n(\bbX)$ to bound it at the order $o(n^{-10})$.
% Combining them together, we have shown that
%$$\frac{1}{m}\sum_{v=1}^m\mathbb{E}\left|G^{(k)}_{vv}(\bbX_{(i\mu)}^{u})\mathcal{T}_n(\bbX_{i\mu}^u)\right|\prec \Psi^2, \ \ k\ge 1,$$
%which implies  $(\ref{0524.2})\le C\Psi^{1/2}$ immediately. This inequality implies that $(\ref{0524.2})\le C\Psi^{1/2}$. Combining with (\ref{1119.10}), we finish our proof.
%By (\ref{1119.4}), we have
%\begin{eqnarray}\label{1119.11}
%|m_Q(x+\hat \mu_m+i\eta)-m_W(x+\hat \mu_m+i\eta)|&\le& |m_Q(x+\hat \mu_m+i\eta)-m(x+\hat \mu_m+i\eta)|\non
%&&+|m(x+\hat \mu_m+i\eta)-m_W(x+\hat \mu_m+i\eta)|
%\prec \Psi^2.\non
%\end{eqnarray}
%So that
%$$|n\int_{E_1}^{E_2}\Im m_Q(x+\hat \mu_m+i\eta)dx-n\int_{E_1}^{E_2}\Im m_W(x+\hat \mu_m+i\eta)dx|\prec n^{1/3+\ep}\Psi^2.$$
%Use the trivial bound $|\Im m_Q(x+\hat \mu_m+i\eta)|\le \eta^{-1}$ outside the event
%$$|n\int_{E_1}^{E_2}\Im m_Q(x+\hat \mu_m+i\eta)dx-n\int_{E_1}^{E_2}\Im m_W(x+\hat \mu_m+i\eta)dx|\le n^{1/3+2\ep}\Psi^2$$
%and $\Im m \le C n^{-1/3+\ep/2}$ by (\ref{1119.7}), we finish our proof.
\end{proof}
In order to prove the Tracy-Widom law, we need to connect the probability $\mathbb{P}(\lambda_1\le E)$ with Theorem \ref{1221-1}.

By Lemma \ref{0525-1} we can fix $E^*\prec n^{-\frac{2}{3}}$ such that it suffices to consider $\lambda_1\leq \hat \mu_m+ E^*$.  Choosing $|E-\hat \mu_m|\prec n^{-\frac{2}{3}}$, $\eta=n^{-\frac{2}{3}-9\ep}$ and $l=\frac{1}{2}n^{-\frac{2}{3}-\ep}$,  then for some sufficiently small constant $\ep>0$ and sufficiently large constant D, there exists a constant $n_0(\ep,D)$ such that
\begin{eqnarray}\label{0526.3}\mathbb{E}K(\frac{n}{\pi}\int_{E-l}^{\hat \mu_m+E^*}\Im m_{\bbX^1}(x+i\eta)dx)\le \mathbb{P}(\lambda_1\le E) \le \mathbb{E}K(\frac{n}{\pi}\int_{E+l}^{\hat \mu_m+E^*}\Im m_{\bbX^1}(x+i\eta)dx)+n^{-D},\non \end{eqnarray}
where $n\ge n_0(\ep,D)$ and K is a smooth cutoff function satisfying the condition of K in Theorem \ref{1221-1}.  We omit the proof of (\ref{0526.3}) because it is a standard procedure and one can refer to \cite{LHY2011} or Corollary 5.1 of \cite{BPWZ2014b} for instance. Combining (\ref{0526.3}) with Theorem \ref{1221-1} one can prove Tracy-Widom's law directly (see the proof of Theorem 1.3 of \cite{BPZ2014a}).

\bigskip\noindent{\bf Acknowledgment}. {G. M. Pan was partially supported by a MOE Tier 2 grant 2014-T2-2-060 and by a MOE Tier 1 Grant RG25/14 at the Nanyang Technological University, Singapore.}
%%%%%%%%%%%%%%%%%%%%%%%%%%%%%%%%%%%%%%%%%%%%%%%%%%%%%%%%%%%%%%%%%%%%%%%%%%%%%%%%%%


\begin{thebibliography}{20}


%\bibitem{BYK87} Bai, Z., Yin, Y.Q. and Krishnaiah P.R., (1987). On the limiting empirical distribution
%function of the eigenvalues of a multivariate F-matrix. Probab. Theo. Appli. 32, 490?%500.
\bibitem{BS06}
\textsc{BAI, Z. D.} and \textsc{SILVERSTEIN, J. W.} (2006).
\textit{Spectral analysis of large dimensional random matrices}, 1st ed.
Springer, New York.

\bibitem{BJ06}
\textsc{BAIK, J.} and \textsc{SILVERSTEIN, J. W.} (2006).
Eigenvalues of Large Sample Covariance Matrices of Spiked Population Models.
\textit{J. Multivariate Anal.}
\textbf{97}, 1382--1408.

\bibitem[BAO, Z. G., PAN, G. M. and ZHOU, W.(2015)]{BPZ2014a}
\textsc{BAO, Z. G., PAN, G. M.} and \textsc{ZHOU, W.} (2015).  Universality for the largest eigenvalue of sample covariance matrices with general population. \textit{Ann. Statist.} \textbf{43}(1), 382--421.

\bibitem[BAO, Z. G., PAN, G. M. and ZHOU, W.(2014)]{BPWZ2014b} \textsc{BAO, Z. G., PAN, G. M.} and \textsc{ZHOU, W.} Local density of the spectrum on the edge for sample covariance matrices with general population. \emph{Preprint.} Available at http://www. ntu. edu. sg/home/gmpan/publications. html.

    \bibitem{DJO15}
 DHARMAWANSA, P., JOHNSTONE, I. M. and ONATSKI, A. (2014). Local Asymptotic Normality of the spectrum of high-dimensional spiked F-ratios. http://arxiv.org/pdf/1411.3875.pdf.

\bibitem[EL. KAROUI, N.(2007)]{K2007} \textsc{EL KAROUI, N.} (2007).  Tracy-Widom Limit for the Largest Eigenvalue of a Large Class of Complex Sample Covariance
Matrices, \emph{Ann. Probab.} \textbf{35},663-714.


    \bibitem{ESY20092}
 ERD\"{O}S, L., SCHLEIN, B. and  YAU, H.-T. (2009). Local Semicircle Law and Complete Delocalization for Wigner Random Matrices. Communications in Mathematical Physics, {\bf 287}(2), 641-655.

\bibitem[ERD\"{O}S, L., YAU, H.-T., and YIN, J.(2011)]{LHY2011}
\textsc{ERD\"{O}S, L., YAU, H.-T.,} and \textsc{YIN, J.}(2011). Rigidity of Eigenvalues of Generalized Wigner Matrices
, \emph{Advances in Mathematics}, \textbf{229(3)},   1435-1515.


\bibitem[ERD\"{O}S, L., KNOWLES, A. and YAU, H.-T.(2013)]{LAY2013}
\textsc{ERD\"{O}S, L., KNOWLES, A.} and \textsc{YAU, H.-T.}(2013). Averaging fluctuations in resolvents of random band matrices, \emph{Ann. H.
Poincar\'e}, \textbf{14}, 1837¨C1926.

\bibitem{FP2009}
 F\'{E}RAL, D.,  P\'{E}CH\'{E}, S.(2009). The largest eigenvalues of sample covariance matrices for a spiked population: Diagonal case. J. Math. Phys. {\bf 50}, 073302.

% \bibitem[Z. Bai and J. W. Silverstein.(1998)]{BS1998} \textsc{Z. Bai and J. W. Silverstein.} No eigenvalues outside the support of the limiting spectral distribution of large-
%dimensional sample covariance matrices,\emph{ Ann. Prob.} \textbf{26}, (1998), 316-345.

\bibitem{Johansson2000}
 JOHANSSON, K. (2000). Shape fluctuations and random matrices. Comm. Math. Phys. {\bf 209}, No. 2, 437-476.

\bibitem[JOHNSTONE, I. M.(2001)]{J2001}  \textsc{JOHNSTONE, I.M.} (2001).  On the Distribution of the Largest Eigenvalue in Principal Component Analysis, \emph{Ann. Statist.} \textbf{29}, 295-327.

\bibitem{J08}
\textsc{JOHNSTONE, I. M.} (2008).
Multivatiate analysis and Jacobi ensembles:Largest eigenvalue,Tracy-Widom limits and rates of convergence.
\textit{Ann. Statist.}
\textbf{36} 2638--2716.

\bibitem{J09}
\textsc{JOHNSTONE, I. M.} (2009).
Approximation null distribution of the largest root in multivariate analysis.
\textit{Ann. Appl. Statist.}
\textbf{3} No.4 1616--1633.


\bibitem{KY14}
\textsc{KNOWLES, A.} and \textsc{YIN, J.} (2015).
Anisotropic local laws for random matrices.
\textit{arXiv:1410.3516v3}.






\bibitem{JK14}
\textsc{LEE, J. O.} and \textsc{SCHNELLI, K.} (2014).
Tracy-Widom Distribution for the Largest Eigenvalue of Real Sample Covariance Matrices with General Population.
\textit{arXiv:1409.4979v1}.




\bibitem{MP67}
\textsc{MAR$\check{C}$ENKO, V. A.} and \textsc{PASTUR, L. A. } (1967).
Distribution of eigenvalues for some sets of random matrices.
\textit{Sb. Math.}
\textbf{4} 457--483.

\bibitem[MUIRHEAD, R. J. (1982)]{M1982} \textsc{MUIRHEAD, R. J.} (1982). Aspects of Multivariate Statistical Theory. \emph{Wiley, New York.}
MR0652932.

\bibitem{PY11}
\textsc{PILLALI, N. S.} and \textsc{Yin, J.} (2011).
Universality of covariance matrices.
\textit{Ann. Appl. Prob.}
\textbf{24} No.3,935--1001.

%\bibitem[Tracy, C. A. and Widom, H.]{TW96}
%\textsc{Tracy, C. A.} and \textsc{Widom, H. } (1996).
%On orthogonal and symplectic matrix ensembles.
%\textit{Comm. Math. Phys.}
%\textbf{177} 727--754.


% \bibitem[Alon et al.(1999)]{Alon1999}
% Alon, U.,   Barkai, N.,  Motterman, D.,  Gish, K.,  Mack, S. and  Levine, J. (1999) Broad patterns of gene expression revealed by clustering analysis of tumor and normal
%colon tissues probed by oligonucleotide arrays. {\emph{Proc. Natl. Acad. Sci.}}, \textbf{96}, 6745-6750.
%\bibitem[Antti Knowles, Jun Yin(2014)]{KY14}
%\textsc{Antti Knowles, Jun Yin.} Anisotropic local laws for random matrices. arXiv:1410.3516  .

%\bibitem[A. Bloemendal, L. Erd\"{o}s, A. Knowles, H.-T. Yau, and J. Yin(2014)]{AEY2014}
%\textsc{A. Bloemendal, L. Erd\"{o}s, A. Knowles, H.-T. Yau, and J. Yin.}
% Isotropic local laws for sample covariance and
%generalized wigner matrices, {\emph{Electron. J. Probab.}}, \textbf{19}(2014), 1-53.



%\bibitem[Alex Bloemendal, Antti Knowles, Horng-Tzer Yau, Jun Yin.(2014)]{AAHJ2014}
%\textsc{Alex Bloemendal, Antti Knowles, Horng-Tzer Yau, Jun Yin.} On the principal components of sample covariance matrices
%,\textit{arXiv:1404.0788v2}.

%\bibitem[Pan, G. M  and Yang, \bbY. R.(2014)]{PMY2014} \textsc{Pan, G. M  and Yang, Y. R.} Independence Test For High Dimensional Data Based On Regularized Canonical Correlation Coefficients.  \emph{Accepted, Annals of Statistics}.

%\bibitem[Johansson, K.(2000)]{J2000} \textsc{Johansson, K.} Shape Fluctuations and Random Matrices, \emph{Comm. Math. Phys.} \textbf{209}, (2000), 437-476.



%\bibitem[Silverstein(1995)]{s85}
%J.W. Silverstein, The limiting eigenvalue distribution of a multivariate F matrix. SIAM J. MATH. ANAL. Vol. 16, No. 3, May 1985

\bibitem[Silverstein(1995)]{s1}
SILVERSTEIN, J. W. and  CHOI,S.-I (1995). Analysis of the Limiting Spectral Distribution of Large Dimensional Random Matrices.
\emph{Journal of Multivariate Analysis}, 54(2), 295¨C309.



\bibitem{Soshnikov2002}
 SOSHNIKOV, A. (2002). A note on universality of the distribution of the largest eigenvalues in certain sample covariance matrices. Jour. Stat. Phys. {\bf 108}(5), 1033-1056.
\bibitem{TV2011}
 TAO, T. and VU,V. (2011). Random matrices: Universality of local eigenvalue statistics. Acta Mathematica, {\bf 206}(1), 127-204.
\bibitem{TV2012}
 TAO, T. and VU, V. (2012). Random covariance matrices: Universality of local statistics of eigenvalues.  Ann. Probab. {\bf 40}(3), 1285-1315.
\bibitem{TW1994}
TRACY, C. A. and WIDOM,H. (1994). Level-spacing distributions and the Airy kernel. Comm. Math. Phys. {\bf 159}, No. 1, 151-174.
\bibitem{TW1996}
TRACY, C. A. and WIDOM,H. (1996). On orthogonal and symplectic matrix ensembles. Comm. Math. Phys. {\bf 177}, No. 3, 727-754.

\bibitem{WK80}
WACHTER, K. (1980) The limiting empirical measure of multiple discriminant ratios, The Annals of Statistics 8, 937-957.

\bibitem{WK2012}
WANG, K. (2012). Random covariance matrices: Universality of local statistics of eigenvalues up to the edge. Random matrices: Theory and Applications, {\bf 1}(1), 1150005.

\bibitem{WY} WANG, Q. W. and YAO, J. F.(2015). Extreme eigenvalues of large-dimensional spiked Fisher matrices with application. http://arxiv.org/pdf/1504.05087.pdf.


\bibitem{ZSR} ZHENG, S. R. (2012). Central Limit Theorem for Linear Spectral Statistics of Large Dimensional
F Matrix. Ann. Institut Henri Poincare Probab. Statist. 48, 444-476.






\end{thebibliography}
\end{document}